\documentclass{aart}

\usepackage{titlesec}
\titleformat{\subsection}[runin]{\normalfont\bfseries}{\thesubsection.}{.5em}{}[.]\titlespacing{\subsection}{0pt}{2ex plus .1ex minus .2ex}{.8em}
\titleformat{\subsubsection}[runin]{\normalfont\itshape}{\thesubsubsection.}{.3em}{}[.]\titlespacing{\subsubsection}{0pt}{1ex plus .1ex minus .2ex}{.5em}

\usepackage[labelfont=sc,font=small,labelsep=period]{caption}
\usepackage[letterpaper, hmargin=1in, top=1.5in, bottom=1.9in, footskip=0.7in]{geometry}

\usepackage{amsmath} 
\usepackage{amssymb}
\usepackage{amsfonts}
\usepackage{latexsym}
\usepackage{amsthm}
\usepackage{amsxtra}
\usepackage{amscd}
\usepackage{bbm}
\usepackage{mathrsfs}
\usepackage{bm}

\usepackage{graphicx, color}
\usepackage[nottoc,notlof,notlot]{tocbibind} 
\usepackage{cite} 
\usepackage{xspace} 

\numberwithin{equation}{section}
\numberwithin{figure}{section}
\newcommand{\f}[1]{\boldsymbol{\mathrm{#1}}} 
\newcommand{\bb}{\mathbb} 
\renewcommand{\cal}{\mathcal} 
 
\newcommand{\fra}{\mathfrak} 
\newcommand{\ol}[1]{\overline{#1} \!\,} 
\newcommand{\wh}{\widehat}
\newcommand{\wt}{\widetilde}

\newcommand{\ii}{\mathrm{i}}
\newcommand{\dd}{\mathrm{d}}
\newcommand{\col}{\mathrel{\mathop:}}

\newcommand{\deq}{\mathrel{\mathop:}=}

\renewcommand{\leq}{\leqslant}
\renewcommand{\geq}{\geqslant}

\renewcommand{\ge}{\geqslant}

\newcommand{\ind}[1]{\f 1 (#1)}
\newcommand{\indb}[1]{\f 1 \pb{#1}}

\renewcommand{\epsilon}{\varepsilon}

\newcommand{\al}{\alpha}
\newcommand{\de}{\delta}

\renewcommand{\P}{\mathbb{P}}
\newcommand{\E}{\mathbb{E}}
\newcommand{\R}{\mathbb{R}}
\newcommand{\C}{\mathbb{C}}
\newcommand{\N}{\mathbb{N}}

\newcommand{\p}[1]{({#1})}
\newcommand{\pb}[1]{\bigl({#1}\bigr)}
\newcommand{\pB}[1]{\Bigl({#1}\Bigr)}
\newcommand{\pbb}[1]{\biggl({#1}\biggr)}
\newcommand{\pBB}[1]{\Biggl({#1}\Biggr)}

\newcommand{\q}[1]{[{#1}]}

\newcommand{\qB}[1]{\Bigl[{#1}\Bigr]}

\newcommand{\qBB}[1]{\Biggl[{#1}\Biggr]}

\newcommand{\h}[1]{\{{#1}\}}
\newcommand{\hb}[1]{\bigl\{{#1}\bigr\}}

\newcommand{\hBB}[1]{\Biggl\{{#1}\Biggr\}}

\newcommand{\abs}[1]{\lvert #1 \rvert}
\newcommand{\absb}[1]{\bigl\lvert #1 \bigr\rvert}

\newcommand{\absbb}[1]{\biggl\lvert #1 \biggr\rvert}

\newcommand{\norm}[1]{\lVert #1 \rVert}

\newcommand{\scalar}[2]{\langle{#1} \mspace{2mu}, {#2}\rangle}

\DeclareMathOperator{\tr}{Tr}

\DeclareMathOperator{\re}{Re}
\DeclareMathOperator{\im}{Im}

\newcommand{\beqa}{\begin{eqnarray}}
\newcommand{\eeqa}{\end{eqnarray}}
\newcommand{\e}{\varepsilon}

\newcommand{\bv}{{\bf{v}}}
\newcommand{\bw}{{\bf{w}}}

\newcommand{\be}{\begin{equation}}
\newcommand{\ee}{\end{equation}}

\theoremstyle{plain} 
\newtheorem{theorem}{Theorem}[section]
\newtheorem*{theorem*}{Theorem}
\newtheorem{lemma}[theorem]{Lemma}
\newtheorem*{lemma*}{Lemma}

\newtheorem*{corollary*}{Corollary}
\newtheorem{proposition}[theorem]{Proposition}
\newtheorem*{proposition*}{Proposition}
\newtheorem{definition}[theorem]{Definition}
\newtheorem*{definition*}{Definition}

\theoremstyle{definition} 

\newtheorem*{example*}{Example}
\newtheorem{remark}[theorem]{Remark}

\newtheorem*{remark*}{Remark}
\newtheorem*{remarks*}{Remarks}

\begin{document}
\title{Isotropic Local Laws for Sample Covariance and Generalized Wigner Matrices}

\author{Alex Bloemendal\footnote{Harvard University, alexb@math.harvard.edu.}
\and L\'aszl\'o Erd\H{o}s\footnote{IST Austria, Am Campus 1, Klosterneuburg A-3400, lerdos@ist.ac.at. On leave from Institute of Mathematics,
University of Munich. Partially supported by SFB-TR 12 Grant of the German Research Council by ERC Advanced Grant RANMAT 338804.} \and Antti Knowles\footnote{ETH Z\"urich, knowles@math.ethz.ch. Partially supported by Swiss National Science Foundation grant 144662.}
 \and Horng-Tzer Yau\footnote{Harvard University, htyau@math.harvard.edu. Partially supported by  NSF  Grant DMS-1307444 and Simons investigator fellowship.}  \and Jun Yin\footnote{University of Wisconsin, jyin@math.wisc.edu. Partially supported by NSF Grant DMS-1207961.}}
 \date{April 4, 2014}
 
\maketitle

\begin{abstract} 
We consider sample covariance matrices of the form $X^*X$, where $X$ is an $M \times N$ matrix with independent random entries.  We prove the \emph{isotropic local Marchenko-Pastur law}, i.e.\ we prove that the resolvent $(X^* X - z)^{-1}$ converges to a multiple of the identity in the sense of quadratic forms. More precisely, we establish sharp high-probability bounds on the quantity $\scalar{\f v}{(X^* X - z)^{-1} \f w} - \scalar{\f v}{\f w} m(z)$, where $m$ is the Stieltjes transform of the Marchenko-Pastur law and $\f v, \f w \in \C^N$. We require the logarithms of the dimensions $M$ and $N$ to be comparable. Our result holds down to scales $\im z \geq N^{-1+\e}$ and throughout the entire spectrum away from 0. We also prove analogous results for generalized Wigner matrices.
\end{abstract}

%
%
%
%
%
%
%

\section{Introduction}

The empirical density of eigenvalues of large $N\times N$ random matrices typically converges to
a deterministic limiting law.  For Wigner matrices this law is the celebrated Wigner semicircle law \cite{Wig} and for sample covariance matrices it is the Marchenko-Pastur law \cite{MP}.  Under some additional
moment conditions this  convergence also holds in very small spectral windows,  all the way down 
 to the scale of the eigenvalue spacing.
In this paper  we normalize the matrix so that the support of
 its spectrum remains bounded 
 as $N$ tends to infinity. In
particular, the typical eigenvalue spacing is of order $1/N$ away from the spectral edges.
The  empirical eigenvalue density is  conveniently, and commonly, 
 studied via its Stieltjes transform -- the normalized trace of the resolvent, $\frac{1}{N} \tr (H-z)^{-1}$,
where $z=E+\ii\eta$ is the spectral parameter with positive imaginary part $\eta$. Understanding the eigenvalue density on 
small  scales of order $\eta$ around a fixed value $E \in \R$  is roughly equivalent to 
 understanding its Stieltjes transform with spectral parameter $z=E+\ii \eta$. 
The  smallest  scale on which a deterministic limit is  expected to emerge  is $\eta \gg N^{-1}$; 
below this scale the empirical eigenvalue density  remains a fluctuating object even in the limit of large $N$, driven by the fluctuations of individual eigenvalues. We remark that a local law on the optimal scale $1/N$ (up to logarithmic corrections) was first obtained in \cite{ESY2}.

In recent years there has been substantial progress in understanding the local 
versions of the semicircle and the Marchenko-Pastur laws (see \cite{EKYY4, ECDM} for 
an overview and detailed references). This research was originally motivated by the Wigner-Dyson-Mehta universality
 conjecture for the local spectral statistics  of random matrices. 
The celebrated sine kernel universality and related results for other symmetry classes
concern higher-order correlation functions, and not just the eigenvalue density. Moreover, they  pertain to scales of order $1/N$, smaller than the scales on which local laws hold.  Nevertheless, local laws (with precise error bounds) are essential  ingredients for proving universality. In particular, one of their consequences, the precise localization of the eigenvalues (called \emph{rigidity bounds}), has
played a fundamental role in the relaxation flow analysis of the Dyson Brownian Motion,
which has led to the proof of the Wigner-Dyson-Mehta universality  conjecture for
all symmetry classes \cite{ESY4, ESYY}.

The basic approach  behind the proofs of local laws  is the analysis of a self-consistent equation for the
Stieltjes transform, a scalar equation which controls the trace of the resolvent (and hence the empirical eigenvalue density).
A vector self-consistent equation
for the diagonal resolvent matrix entries, $[(H-z)^{-1}]_{ii}$, was introduced in \cite{EYY1}. Later, a matrix self-consistent equation was derived in \cite{EKYY3}.
Such self-consistent equations provide \emph{entrywise} control of the resolvent and not only its trace. 
 This latter fact has proved a key ingredient in the \emph{Green function comparison method} (introduced in \cite{EYY1} and extended to the spectral edge in \cite{EYY3}), which allows the comparison of local statistics via moment matching even below the scale of eigenvalue spacing.

In this paper we are concerned with \emph{isotropic local laws}, in which the control of the matrix entries $[(H - z)^{-1}]_{ij}$ is generalized to a control of quantities of the form $\scalar{\f v}{(H - z)^{-1} \f w}$, where $\f v , \f w \in \C^N$ are deterministic vectors. This may be interpreted as basis-independent control on the resolvent.  The fact that the matrix entries are independent distinguishes the standard basis of $\C^N$ in the analysis of the resolvent.  Unless the entries of $H$ are Gaussian, this independence of the matrix entries is destroyed after a change of basis, and the isotropic law is a nontrivial  generalization of the entrywise law. 
 The first isotropic local law was proved in \cite{KY2}, where it was established for Wigner matrices.

The main motivation for isotropic local laws is the study of \emph{deformed matrix ensembles}. A simple example is the sum $H + A$ of a Wigner matrix $H$ and a deterministic finite-rank matrix $A$. As it turns out, a powerful means to study the eigenvalues and eigenvectors of such deformed matrices is to derive large deviation bounds and central limit theorems for quantities of the form $\scalar{\f v}{(H - z)^{-1} \f w}$, where $\f v$ and $\f w$ are eigenvectors of $A$. Deformed matrix ensembles are known to exhibit rather intricate spectra, depending on the spectrum of $A$. In particular, the spectrum of $H + A$ may contain \emph{outliers} -- lone eigenvalues separated from the bulk spectrum. The creation or annihilation of an outlier occurs at a sharp transition when an eigenvalue of $A$ crosses a critical value. This transition is often referred to as the \emph{BBP transition} and was first established in \cite{BBP} for unitary matrices and extended in \cite{BV1, BV2} to other symmetry classes.
Similarly to the above deformed Wigner matrices, one may introduce a class of deformed sample covariance matrices, commonly referred to as \emph{spiked population models} \cite{Johnstone}, which describe populations with nontrivial correlations (or ``spikes'').

The isotropic local laws established in this paper serve as a key input in establishing detailed results about the eigenvalues and eigenvectors of deformed matrix models. These include:
\begin{enumerate}
\item[(a)]
A complete picture of the distribution of outlier eigenvalues/eigenvectors, as well as the non-outlier eigenvalues/eigenvectors near the spectral edge. 
\item[(b)]
An investigation of the BBP transition using that, thanks to the optimality of the high-probability bounds in the local laws, the results of (a) extend even to the case when some eigenvalues of $A$
are very close to the critical value.
\end{enumerate}
This programme for the eigenvalues of deformed Wigner matrices was carried out in \cite{KY2, KY3}. In \cite{KYY}, we carry it out for the eigenvectors of covariance matrices.

In this paper we prove the isotropic Marchenko-Pastur law for sample covariance matrices as well as the isotropic semicircle law for generalized Wigner matrices. Our proofs are based on a novel method, which is considerably more robust than that of \cite{KY2}. Both proofs (the one from \cite{KY2} and the one presented here) crucially rely on the entrywise local law as input,
but follow completely different approaches to obtain the isotropic law from the entrywise one.
The basic idea of the proof in \cite{KY2} is to use the Green function comparison method to compare the resolvent of a given Wigner matrix to the resolvent of a Gaussian  random  matrix, for which the isotropic law is a trivial corollary of the entrywise one (by basis transformation). Owing to various moment matching conditions imposed by the Green function comparison, the result of \cite{KY2} required the variances of all matrix entries to coincide and, for results in the bulk spectrum, the third moments to vanish.
In contrast, our current approach does not rely on Green function comparison. Instead, it consists of a precise analysis of the cancellation of fluctuations in Green functions. We use a graphical expansion method inspired by techniques recently developed in \cite{EKY2} to control fluctuations in  Green functions of random band matrices.

Our first main result is the isotropic local Marchenko-Pastur law for sample covariance matrices $H = X^* X$, where $X$ is an $M \times N$ matrix. We allow the dimensions of $X$ to differ wildly: we only assume that $\log N \asymp \log M$. In particular, the aspect ratio $\phi = M/N$ -- a key parameter in the
Marchenko-Pastur law -- may scale as a power of $N$.
Our entrywise law  (required as input for the proof of the isotropic law) is a generalization of the one given in \cite{PY}. In addition to generalizing the proof of \cite{PY}, we simplify and streamline it, so as to obtain a short and self-contained proof.

Our second main result is the isotropic local semicircle law for generalized Wigner matrices. This extends the isotropic law of \cite{KY2} from Wigner matrices to generalized Wigner matrices, in which the variances of the matrix entries need not coincide. It also dispenses with the third moment assumption of \cite{KY2} mentioned previously. In fact, our proof applies to even more general matrix models, provided that an entrywise law has been established.  As an application of the isotropic laws, we also prove a basis-independent version of eigenvector delocalization for both sample covariance and generalized Wigner matrices. 

We conclude with an outline of the paper. In Section \ref{sec:results} we define our models and state our results, first for sample covariance matrices (Section \ref{sec: sample covariance}) and then for generalized Wigner matrices (Section \ref{sec: gen_wigner}). The rest of the paper is devoted to the proofs. Since they are very similar for sample covariance matrices and generalized Wigner matrices, we only give the details for sample covariance matrices. Thus, Sections \ref{sec:tools}--\ref{sec:5} are devoted to the proof of the isotropic Marchenko-Pastur law for sample covariance matrices; in Section \ref{sec:appendix}, we describe how to modify the arguments to prove the isotropic semicircle law for generalized Wigner matrices. Section \ref{sec:tools} collects some basic identities and estimates that we shall use throughout the proofs. In Section \ref{sec: MP law} we prove the entrywise local Marchenko-Pastur law, generalizing the results of \cite{PY}. The main argument and the bulk of the proof, i.e.\ the proof of the isotropic law, is given in Section \ref{sec:4}. For a sketch of the argument we refer to Section \ref{sec:sketch}. Finally, in Section \ref{sec:5} we draw some simple consequences from the isotropic law: optimal control outside of the spectrum and isotropic delocalization bounds.

\subsubsection*{Conventions}
We use $C$ to denote a generic large positive constant, which may depend on some fixed parameters and whose value may change from one expression to the next. Similarly, we use $c$ to denote a generic small positive constant. For two positive quantities $A_N$ and $B_N$ depending on $N$ we use the notation $A_N \asymp B_N$ to mean $C^{-1} A_N \leq B_N \leq C A_N$ for some positive constant $C$.

\section{Results} \label{sec:results}

\subsection{Sample covariance matrix} \label{sec: sample covariance}

Let $X$ be an $M\times N$ matrix whose entries $X_{i \mu}$  are independent complex-valued random variables satisfying
\begin{equation} \label{cond on entries of X}
\E X_{i \mu}\;=\;0\,,\qquad \E\abs{X_{i \mu}}^2\;=\;\frac{1}{\sqrt {N M}}\,.
\end{equation}
We shall study the $N\times N$ matrix $X^*X$; hence we regard $N$ as the fundamental large 
parameter, and write $M \equiv M_N$.
Our results also apply to the matrix $X X^*$ provided one replaces $N \leftrightarrow M$. See Remark \ref{rem:XXstar} below for more details.

We always assume that $M$ and $N$ satisfy the bounds
\begin{equation} \label{NM gen}
N^{1/C} \;\leq\; M \;\leq\; N^C
\end{equation}
for some positive constant $C$.
We define the ratio 
\[
 \phi \;=\; \phi_N  \;\deq\; \frac MN\,,
\]
which may depend on $N$. Here, and throughout the following, 
in order to unclutter notation we omit the argument $N$ in quantities, 
such as $X$ and $\phi$, that depend on it.

We make the following technical assumption on the tails of the entries of $X$. We assume that,  for all $p \in \N$, the random variables $(N M)^{1/4} X_{i\mu}$ have a uniformly bounded $p$-th moment:
there is a constant $C_p$ such that
\begin{equation} \label{moments of X-1}
\E \absb{(N M)^{1/4} X_{i\mu}}^p \;\leq\; C_p\,.
\end{equation}

It is well known that the empirical distribution of the eigenvalues of the $N\times N$ matrix $X^*X$ has the same asymptotics as the \emph{Marchenko-Pastur} law\cite{MP}
\begin{equation}\label{def: rhog}
\varrho_\phi(\dd x) \;\deq\;\frac{\sqrt\phi}{2\pi}\sqrt{\frac{\q{(x-\gamma_-)(\gamma_+-x)}_+}{x^2}}\, \dd x + (1-\phi)_+ \, \delta(\dd x)\,,
\end{equation}
where we defined
\begin{equation} \label{def:gamma_pm}
 \gamma_\pm  \;\deq\;\sqrt\phi+\frac 1{\sqrt\phi}\pm 2\,
\end{equation}
to be the edges of the limiting spectrum.
Note that \eqref{def: rhog} is normalized so that its integral is equal to one. 
The Stieltjes transform of the Marchenko-Pastur law \eqref{def: rhog} is
\begin{equation} \label{S_MP}
m_\phi(z)\;\deq\; \int \frac{\varrho_\phi (\dd x)}{x-z} \;=\; \frac{\phi^{ 1/2}-\phi^{-1/2}-z+\ii \sqrt{(z-\gamma_-)(\gamma_+-z)}}{2\,\phi^{-1/2}\, z}\,,
\end{equation}
where the square root is chosen so that $m_\phi$ is holomorphic in the upper half-plane and satisfies $m_\phi(z) \to 0$ as $z \to \infty$. 
The function $m_\phi = m_\phi(z)$ is also characterized as the unique solution of the equation
\begin{equation} \label{identity for m MP}
m+\frac{1}{z+z\phi^{-1/2}m-(\phi^{ 1/2}-\phi^{-1/2})} \;=\; 0
\end{equation}
satisfying $\im m (z) > 0$ for $\im z >0$.
The formulas \eqref{def: rhog}--\eqref{identity for m MP} were originally derived for the case when $\phi=M/N$ is independent of $N$ 
 (or, more precisely, when $\phi$ has a limit in $(0,\infty)$ as $N \to \infty$).  Our results
allow $\phi$ to depend on $N$ under the constraint \eqref{NM gen}, so that $m_\phi$ and $\varrho_\phi$ may also depend on $N$ through $\phi$.

Throughout the following we use a spectral parameter
\begin{equation*}
z \;=\; E + \ii \eta\,,
\end{equation*}
with $\eta > 0$, as the argument of Stieltjes transforms and resolvents.
Define the resolvent
\begin{equation} \label{def_R}
R(z) \;\deq\; (X^*X-z)^{-1}\,.
\end{equation}
For $z\in \C$, define $\kappa\deq\kappa(z)$ to be the distance of $E=\re z$ to the spectral edges $\gamma_\pm$, i.e.
\begin{equation} \label{def kappa sc}
\kappa \;\deq\; \min \hb{\abs{\gamma_+-E}\,,\,\abs{\gamma_- -E}}\,.
\end{equation}

The following notion of a high-probability bound was introduced in \cite{EKY2}, and has been subsequently used in a number of works on random matrix theory. It provides a simple way of systematizing and making precise statements of the form ``$\xi$ is bounded with high probability by $\zeta$  up to small powers of $N$''.

\begin{definition}[Stochastic domination]\label{def:stocdom}
Let
\begin{equation*}
\xi = \pb{\xi^{(N)}(u) \col N \in \N, u \in U^{(N)}} \,, \qquad
\zeta = \pb{\zeta^{(N)}(u) \col N \in \N, u \in U^{(N)}}
\end{equation*}
be two families of nonnegative random variables, where $U^{(N)}$ is a possibly $N$-dependent parameter set. 
We say that $\xi$ is \emph{stochastically dominated by $\zeta$, uniformly in $u$,} if for all (small) $\epsilon > 0$ and (large) $D > 0$ we have
\begin{equation*}
\sup_{u \in U^{(N)}} \P \qB{\xi^{(N)}(u) > N^\epsilon \zeta^{(N)}(u)} \;\leq\; N^{-D}
\end{equation*}
for large enough $N\geq N_0(\e, D)$.
Throughout this paper the stochastic 
domination will always be uniform in all parameters (such as matrix indices and the spectral parameter $z$) that are not explicitly fixed. Note that $N_0(\e, D)$ may depend on the constants from \eqref{NM gen} and \eqref{moments of X-1} as well as any constants fixed in the assumptions of our main results.
If $\xi$ is stochastically dominated by $\zeta$, uniformly in $u$, we use the notation $\xi \prec \zeta$. Moreover, if for some complex family $\xi$ we have $\abs{\xi} \prec \zeta$ we also write $\xi = O_\prec(\zeta)$.
\end{definition}

\begin{remark} \label{rem:stochdomMN} 
Because of \eqref{NM gen}, all (or some) factors of $N$ in Definition \eqref{def:stocdom} could be replaced with $M$ without changing the definition of stochastic domination. 
\end{remark}

Sometimes we shall need the following notion of high probability.
\begin{definition} \label{def: high probability}
An $N$-dependent event $\Xi \equiv \Xi_N$ holds with \emph{high probability} if $1 - \ind{\Xi} \prec 0$ (or, equivalently, if $1 \prec \ind{\Xi}$).
\end{definition}

We introduce the quantity
\begin{equation} \label{def_K}
K \;\equiv\; K_N \;\deq\; \min\{M,N\}\,,
\end{equation}
which is the number of nontrivial (i.e.\ nonzero) eigenvalues of $X^* X$; the remaining $N - K$ eigenvalues of $X^* X$ are zero. (Note that the $K$ nontrivial eigenvalues of $X^* X$ coincide with those of $X X^*$.) Fix a (small) $\omega \in (0,1)$
and define the domain
\begin{equation} \label{def_S_theta}
\f S \;\equiv\; \f S(\omega, K)  \;\deq\; \hb{z = E+\ii\eta \in \C \col  \kappa \leq   \omega^{-1} \,,\, K^{-1+\omega} \leq \eta \leq  \omega^{-1} \,,\, \abs{z} \geq \omega}\,.
\end{equation}
Throughout the following we regard $\omega$ as fixed once and for all, and do not track the dependence of constants on $\omega$.

\begin{theorem}[Isotropic local Marchenko-Pastur law] \label{thm: IMP gen}
Suppose that \eqref{cond on entries of X}, \eqref{NM gen}, and \eqref{moments of X-1} hold. 
Then
\begin{equation}\label{bound: Rij isotropic gen}
\absb{\scalar{\f v}{R(z) \f w} - m_\phi(z) \scalar{\f v}{\f w}} \;\prec\; \sqrt{\frac{\im m_{\phi}(z)}{N \eta}} + \frac{1}{N \eta}
\end{equation}
uniformly in $z \in \f S$ and any deterministic unit vectors $\f v, \f w \in \C^N$.
Moreover, we have
\begin{equation}\label{bound: trR gen}
\absbb{\frac{1}{N}\tr R(z) - m_\phi(z)} \;\prec\; \frac{1}{N\eta}
\end{equation}
uniformly in $z \in \f S$.
\end{theorem}

Beyond the support of the limiting spectrum, one has stronger control all the way down to the real axis. For fixed (small) $\omega \in (0,1)$ define the region
\begin{equation} \label{def_S_theta_wt}
\wt{\f S} \;\equiv\; \wt {\f S}(\omega, K) \;\deq\; \hb{z=E+\ii \eta \in \C \col E \notin [\gamma_-, \gamma_+]\,,\, K^{-2/3 + \omega} \leq \kappa \leq \omega^{-1} \,,\, \abs{z} \geq \omega\,,\, 0 < \eta \leq \omega^{-1}}
\end{equation}
of spectral parameters separated from the asymptotic spectrum by $K^{-2/3 + \omega}$, which may have an arbitrarily small positive imaginary part $\eta$.

\begin{theorem}[Isotropic local Marchenko-Pastur law outside the spectrum] \label{thm: IMP outside gen}
Suppose that \eqref{cond on entries of X}, \eqref{NM gen}, and \eqref{moments of X-1} hold.
Then
\begin{equation}\label{bound: Rij isotropic outside sc gen}
\absb{\scalar{\f v}{R(z) \f w} - m_\phi(z) \scalar{\f v}{\f w}} \;\prec\; \sqrt{\frac{\im m_\phi(z)}{N\eta}} \;\asymp\; \frac{1}{1 + \phi^{-1}} (\kappa + \eta)^{-1/4} K^{-1/2}
\end{equation}
uniformly in $z \in \wt{\f S}$ and any deterministic unit vectors $\f v, \f w \in \C^N$.
\end{theorem}

\begin{remark} \label{rem:all_z}
All probabilistic estimates
\eqref{bound: Rij isotropic gen}--\eqref{bound: Rij isotropic outside sc gen} of Theorems \ref{thm: IMP gen} 
and \ref{thm: IMP outside gen}  may be strengthened to hold simultaneously for all $z \in \f S$ and for all $z\in \wt{\f S}$, respectively.
For instance, \eqref{bound: Rij isotropic gen} may be strengthened to
\begin{equation*}
\P \qBB{\bigcap_{z \in \f S} \hBB{\absb{\scalar{\f v}{R(z) \f w} - m_\phi(z) \scalar{\f v}{\f w}} \leq N^\epsilon \pBB{\sqrt{\frac{\im m_{\phi}(z)}{N \eta}} + \frac{1}{N \eta}}}} \;\geq\; 1 - N^{-D}\,,
\end{equation*}
for all $\epsilon > 0$, $D > 0$, and $N \geq N_0(\epsilon, D)$.

In the case of Theorem \ref{thm: IMP outside gen} this generalization is an immediate consequence of its proof, and in the case of Theorem \ref{thm: IMP gen} it follows from a simple lattice argument combined with the Lipschitz continuity of $R$ and  $m_\phi$ on $\f S$. See e.g.\ \cite[Corollary 3.19]{EKYY1} for the details.
\end{remark}

\begin{remark}
The right-hand side of \eqref{bound: Rij isotropic outside sc gen}
is stable under the limit $\eta \to 0$, and may therefore be extended to $\eta = 0$. Recalling the previous remark, we conclude that \eqref{bound: Rij isotropic outside sc gen} also holds for $\eta = 0$.
\end{remark}

The next results are on the nontrivial eigenvalues of $X^* X$ as well as the corresponding eigenvectors. As remarked above, the matrix $X^* X$ has $K$ nontrivial eigenvalues, which we order according to $\lambda_1 \geq \lambda_2 \geq \cdots \geq \lambda_K$.
Let $\f u^{(1)}, \dots, \f u^{(K)} \in \C^N$ be the normalized eigenvectors of $X^* X$ associated with the nontrivial eigenvalues $\lambda_1, \dots, \lambda_{K}$. 
\begin{theorem}[Isotropic delocalization] \label{thm: I deloc gen}
Suppose that \eqref{cond on entries of X}, \eqref{NM gen}, and \eqref{moments of X-1} hold. Then for any $\epsilon > 0$ 
we have the bound
\begin{equation}\label{smfy gen}
\abs{\scalar{\f u^{(\alpha)}}{\f v}}^2 \;\prec\; N^{-1}
\end{equation}
uniformly for $\alpha \leq (1 - \epsilon) K$ and all normalized $\f v \in \C^N$.
If in addition $\abs{\phi - 1} \geq c$ for some constant $c > 0$, then \eqref{smfy gen} holds uniformly for all $\alpha \leq K$. 
\end{theorem}

\begin{remark}
Isotropic delocalization bounds in particular imply that the entries $u_i^{(\alpha)}$ of the eigenvectors $\f u^{(\alpha)}$ are strongly oscillating in the sense that $\sum_{i = 1}^N \abs{u_i^{(\alpha)}} \succ N^{1/2}$ but $\absb{\sum_{i = 1}^N u_i^{(\alpha)}} \prec 1$. To see this, choose $\f v = \f e_i$ in \eqref{smfy gen}, which implies $\abs{u_i^{(\alpha)}} \prec N^{-1/2}$, from which the first estimate follows using $\sum_{i = 1}^N \abs{u_i^{(\alpha)}}^2 = 1$. On the other hand, choosing $\f v = N^{-1/2} (1,1, \dots, 1)$ in \eqref{smfy gen} yields the second estimate. Note that, if $\f u = (u_1, \dots, u_N)$ is uniformly distributed on the unit sphere $\bb S^{N-1}$, the high-probability bounds $\sum_{i = 1}^N \abs{u_i} \succ N^{1/2}$ and $\abs{\sum_{i = 1}^N u_i} \prec 1$ are sharp  (in terms of the power of $N$ on the right-hand side). 
\end{remark}

The following result is on the rigidity of the nontrivial eigenvalues of $X^* X$, which coincide with the nontrivial eigenvalues of $X X^*$. 
Let $\gamma_1 \geq \gamma_2 \geq \cdots \geq \gamma_K$ be the \emph{classical eigenvalue locations according to $\varrho_{\phi}$} (see \eqref{def: rhog}), defined through
\begin{equation} \label{def:gamma_alpha}
\int_{\gamma_\alpha}^\infty \varrho_{\phi}(\dd x) \;=\; \frac{\alpha}{N}\,.
\end{equation}

\begin{theorem}[Eigenvalue rigidity]\label{thm: cov-rig}
Fix a (small) $\omega \in (0,1)$ and suppose that \eqref{cond on entries of X}, \eqref{NM gen}, and \eqref{moments of X-1} hold.
Then
\begin{equation}\label{rigidity1}
\absb{\lambda_\alpha-\gamma_\alpha}   \;\prec\; \alpha^{-1/3}K^{-2/3}
\end{equation}
uniformly   for all $\alpha \in \{1, \dots, [ (1-\omega) K]\}$.
If in addition $\abs{\phi - 1} \geq c$ for some constant $c > 0$ then 
\begin{equation}\label{rigidity2}
\absb{\lambda_\alpha-\gamma_\alpha}\;\prec\;   (K+1-\al)^{-1/3} K^{-2/3}, 
\end{equation}
uniformly   for all $\alpha  \in \{[K/2], \dots, K\}$.
\end{theorem}

\begin{remark} \label{rem:XXstar}
We stated our results for the matrix $X^* X$, but they may be easily applied to the matrix $X X^*$ as well. Indeed, Theorems \ref{thm: IMP gen}, \ref{thm: IMP outside gen}, \ref{thm: I deloc gen}, and \ref{thm: cov-rig} remain valid after the following changes: $X \mapsto X^*$, $M \mapsto N$, $N \mapsto M$, and $\phi \mapsto \phi^{-1}$. (In the case of Theorem \ref{thm: cov-rig} these changes leave the statement unchanged.)  Note that the empirical distribution of the eigenvalues of $XX^*$ has the same asymptotics as $\varrho_{\phi^{-1}}(\dd x)$, whose Stieltjes transform is
\begin{equation} \label{mphi_minvphi}
m_{\phi^{-1}}(z) \;=\; \frac{1}{\phi} \pbb{m_\phi(z) + \frac{1 - \phi}{z}}\,.
\end{equation}
\end{remark}

\subsection{Generalized Wigner matrix} \label{sec: gen_wigner}

Let $H = H^*$ be an $N\times N$ Hermitian matrix whose entries $H_{ij}$ are independent complex-valued random variables for $i \leq j$. We always assume that entries are centred, i.e.\ $\E H_{ij} = 0$. Moreover, we assume that the variances
\begin{equation} \label{def_Sij}
S_{ij} \;\deq\; \E\abs{H_{ij}}^2
\end{equation}
satisfy
\begin{equation} \label{S_ass}
C^{-1} \;\leq\; N  S_{ij} \;\leq\; C\,,\qquad \sum_{j} S_{ij}\;=\;1\,,
\end{equation}
for some constant $C > 0$.
We assume that all moments of the entries of $\sqrt{N} H$ are finite in the sense that for all $p \in \N$ there exists a constant $C_p$ such that
\begin{equation} \label{tail_Wig}
\E \absb{\sqrt{N} H_{ij}}^p \;\leq\; C_p
\end{equation}
for all $N$, $i$, and $j$.

Let
\[
\varrho(\dd x) \;\deq\; \frac{1}{2 \pi} \sqrt{(4-x^2)_+} \, \dd x
\]
denote the semicircle law, and
\begin{equation} \label{def_msc}
m(z)\;\deq\;\int\frac{\varrho(\dd x)}{x-z} \;=\; \frac{-z + \sqrt{z^2 - 4}}{2}
\end{equation}
its Stieltjes transform; here we chose the square root so that $m$ is holomorphic in the upper half-plane and satisfies $m(z) \to 0$ as $z \to \infty$. 
Note that $m=m(z)$ is also characterized as the unique solution of
\begin{equation} \label{m identity}
m+\frac{1}{z+m}\;=\;0
\end{equation}
that satisfies $\im m>0$ for $\eta>0$.
Let
\[
G(z)\;\deq\;(H-z)^{-1}
\]
be the resolvent of $H$.

Fix a (small) $\omega \in (0,1)$ and define
\begin{equation}
\f S_W \;\equiv\; \f S_W(\omega, N)  \;\deq\; \hb{z=E+\ii\eta \in\C \col \abs{E}\leq \omega^{-1}\,,\, N^{-1 + \omega} \leq \eta \leq \omega^{-1}}\,.
\end{equation}
The subscript $W$ in $\f S_W$ stands for Wigner, and is added to distinguish this domain from the one defined in \eqref{def_S_theta}. 

\begin{theorem}[Isotropic local semicircle law] \label{thm: ILSC}
 Suppose that \eqref{S_ass} and \eqref{tail_Wig} hold.  Then
\begin{equation} \label{ISSC_estimate}
\absb{\scalar{\bv}{G(z)\bw}-\scalar{\bv}{\bw}m(z)} \;\prec\; \sqrt{\frac{\im m(z)}{N\eta}}+\frac{1}{N\eta}
\end{equation}
uniformly in $z\in\f S_W$ and any deterministic unit vectors $\bv, \bw \in \C^N$.
\end{theorem}

Theorem \ref{thm: ILSC} is the isotropic generalization of the following result, proved in \cite{EKYY4}. A similar result first appeared in \cite{EYY3}.
\begin{theorem}[Local semicircle law, \cite{EKYY4, EYY3}] \label{thm: LSC}
 Suppose that \eqref{S_ass} and \eqref{tail_Wig} hold.  Then
\begin{equation}\label{eq:lscold}
\absb{G_{ij}(z)-\de_{ij}m(z)} \;\prec\; \sqrt{\frac{\im m(z)}{N\eta}}+\frac{1}{N\eta}
\end{equation}
uniformly in $z \in \f S_W$ and $i,j = 1, \dots, N$.
\end{theorem}

The proof of the isotropic law \eqref{ISSC_estimate} uses the regular, entrywise, law  from Theorem \ref{thm: LSC}  as input, in which $\f v$ and $\f w$ are taken to be parallel to the coordinate axes.
Assuming the entrywise law \eqref{eq:lscold} has been established, the proof of \eqref{ISSC_estimate} is very robust, and holds under more general assumptions than \eqref{S_ass}. For instance, we have the following result.

\begin{theorem} \label{thm:gen}
Let $\wt {\f S} \subset \f S_W$ be an $N$-dependent spectral domain, $\wt m(z)$ a deterministic function on $\wt {\f S}$ 
 satisfying $c \leq \abs{\wt m(z)} \leq C$ for $z \in \wt {\f S}$,  and $\wt \Psi(z)$ a deterministic control parameter satisfying $c N^{-1} \leq \wt\Psi(z) \leq N^{-c}$
for $z \in \wt {\f S}$ and some constant $c > 0$.
Suppose that
\begin{equation*}
\abs{G_{ij}(z) - \delta_{ij} \wt m(z)} \;\prec\; \wt \Psi(z)
\end{equation*}
uniformly in $z\in \wt {\f S}$. Suppose that the entries of $H$ satisfy \eqref{tail_Wig} and the variances \eqref{def_Sij} of $H$ satisfy $S_{ij} \leq C N^{-1}$ (which replaces the stronger assumption \eqref{S_ass}).  Then we have
\begin{equation}
\absb{\scalar{\bv}{G(z)\bw}-\scalar{\bv}{\bw}\wt m(z)} \;\prec\; \wt \Psi(z)
\end{equation}
uniformly in $z\in\wt {\f S}$ and any deterministic unit vectors $\bv, \bw \in \C^N$.
\end{theorem}

The proof of Theorem \ref{thm:gen} is the same as that of
Theorem \ref{thm: ILSC}. Below we give the proof for Theorem \ref{thm: ILSC}, which can be trivially adapted to yield Theorem \ref{thm:gen}.

Combining Theorem \ref{thm:gen} with the isotropic local semicircle law from \cite{EKYY4}, we may for instance obtain an isotropic local semicircle law for matrices where the lower bound of \eqref{S_ass} is relaxed, so that some matrix entries may vanish.

Beyond the support of the limiting spectrum $[-2,2]$, the statement of Theorem \ref{thm: ILSC} may be improved to a bound that is stable all the way down to the real axis. For fixed (small) $\omega \in (0,1)$ define the region
\begin{equation}
\wt{\f S}_W \;\equiv\; \wt {\f S}_W(\omega, N) \;\deq\; \hb{z=E+\ii\eta \in\C \col 2 + N^{-2/3 + \omega} \leq \abs{E} \leq \omega^{-1} \,,\, 0 < \eta \leq \omega^{-1}}
\end{equation}
of spectral parameters separated from the asymptotic spectrum by $N^{-2/3 + \omega}$, which may have an arbitrarily small positive imaginary part $\eta$.

\begin{theorem}[Isotropic local semicircle law outside the spectrum] \label{thm: ILSC outside}
Suppose that \eqref{S_ass} and \eqref{tail_Wig} hold. Then
\begin{equation}\label{bound: Gij isotropic outside}
\absb{\scalar{\f v}{G(z) \f w} -  m(z) \scalar{\f v}{\f w}} \; \prec \; \sqrt{\frac{\im m(z)}{N\eta}}
\end{equation}
uniformly in $z\in \wt {\f S}_W$ and any deterministic unit vectors $\bv, \bw \in \C^N$.
\end{theorem}
The statements in Theorems~\ref{thm: ILSC} and \ref{thm: ILSC outside} can also be strengthened to simultaneously apply for all $z \in \f S_W$ and $z \in \wt{\f S}_W$, respectively; see Remark~\ref{rem:all_z}.

Let $\f u^{(1)}, \dots, \f u^{(N)}$ denote the normalized eigenvectors of $H$ associated with the eigenvalues $\lambda_1, \dots, \lambda_N$.

\begin{theorem}[Isotropic delocalization] \label{thm: I deloc Wig}
 Suppose that \eqref{S_ass} and \eqref{tail_Wig} hold.  Then
\begin{equation*}
\abs{\scalar{\f u^{(\alpha)}}{\f v}}^2 \;\prec\; N^{-1}
\end{equation*}
uniformly for all $\alpha = 1, \dots, N$ and all deterministic unit vectors $\f v \in \C^N$.
\end{theorem}

Finally, in analogy to Theorem \ref{thm: cov-rig}, we record the following rigidity result, which is a trivial consequence of \cite[Theorem 7.6]{EKYY4} with $X = C N^{-2/3}$ and $Y = C N^{-1}$; see also \cite[Theorem 2.2]{EYY3}. Write $\lambda_1\geq \lambda_2 \geq \cdots \geq \lambda_N$ for the eigenvalues of $H$, and let $\gamma_1 \geq \gamma_2  \geq \cdots \geq \gamma_N$ be their \emph{classical locations according to $\varrho$}, defined through
\begin{equation}
\int_{\gamma_\alpha}^\infty \varrho(\dd x) \;=\; \frac{\alpha}{N}\,.
\end{equation}
Then we have
\begin{equation}
\abs{\lambda_\alpha-\gamma_\alpha} \;\prec\; \pb{\min \h{\alpha, N + 1 - \alpha}}^{-1/3} N^{-2/3}
\end{equation}
for all $\alpha = 1, \dots, N$.

\section{Preliminaries} \label{sec:tools}

The rest of this paper is devoted to the proofs of our main results. They are similar for sample covariance matrices and generalized Wigner matrices, and in Sections \ref{sec:tools}--\ref{sec:5} we give the argument for sample covariance matrices (hence proving the results of Section \ref{sec: sample covariance}). How to modify these arguments to generalized Wigner matrices (and hence prove the results of Section \ref{sec: gen_wigner}) is explained in Section \ref{sec:appendix}. We choose to present our method in the context of sample covariance matrices mainly for two reasons. First, we take this opportunity to give a version of the entrywise local law (Section \ref{sec: MP law}) -- required as input for the proof of the isotropic law -- which is more general and has a simpler proof than the local law previously established in \cite{PY}. Second, the proof of the isotropic law in the case of sample covariance matrices is conceptually slightly clearer due to a natural splitting of summation indices into two categories (which we distinguish by the use of Latin and Greek letters); this splitting is an essential structure behind our proof in Section \ref{sec:4}, and is also used in the case of generalized Wigner matrices, in which case it is  however purely artificial.

We now move on to the proofs. In order to unclutter notation, we shall often omit the argument $z$ from quantities that depend on it. Thus, we for instance often write $G$ instead of $G(z)$. We put the arguments $z$ back when needed, typically if we are working with several different spectral parameters $z$.

\subsection{Basic tools} 
We begin by recording some basic large deviations estimates.
We consider complex-valued random variables  $\xi$  satisfying
\begin{equation} \label{cond on X}
\E \xi \;=\; 0\,, \qquad \E \abs{\xi}^2 \;=\; 1\,, \qquad \p{\E \abs{\xi}^p}^{1/p} \;\leq\; C_p
\end{equation}
for all $p \in \N$ and some constants $C_p$.

\begin{lemma}[Large deviation bounds] \label{lemma: LDE}
Let $\pb{\xi_i^{(N)}}$ and $\pb{\zeta_i^{(N)}}$ be independent families of random variables and 
$\pb{a_{ij}^{(N)}}$ and $\pb{b_i^{(N)}}$ be deterministic; here $N \in \N$ and $i,j = 1, \dots, N$. Suppose that all entries $\xi_i^{(N)}$ and $\zeta_i^{(N)}$ are independent and satisfy \eqref{cond on X}. Then we have the bounds
\begin{align} \label{LDE}
\sum_i b_i \xi_i &\;\prec\; \pbb{\sum_i \abs{b_i}^2}^{1/2}\,,
\\ \label{two-set LDE}
\sum_{i,j} a_{ij} \xi_i \zeta_j &\;\prec\; \pbb{\sum_{i,j} \abs{a_{ij}}^2}^{1/2}\,,
\\ \label{offdiag LDE}
\sum_{i \neq j} a_{ij} \xi_i \xi_j &\;\prec\; \pbb{\sum_{i \neq j} \abs{a_{ij}}^2}^{1/2}\,.
\end{align}
 If the coefficients $a_{ij}^{(N)}$ and $b_i^{(N)}$
depend on an additional parameter $u$, then all of these estimates are uniform in $u$ (see Definition \ref{def:stocdom}), i.e.\ the threshold $N_0= N_0(\e, D)$ in the definition of $\prec$ depends only on the family $C_p$ from \eqref{cond on X}; in particular, $N_0$ does not depend on $u$. 
\end{lemma}

\begin{proof}
These estimates are an immediate consequence of Lemmas B.2, B.3, and B.4 in \cite{EKYY3}. See also Theorem C.1 in \cite{EKYY4}. 
\end{proof}

The following lemma collects basic algebraic properties of stochastic domination $\prec$. We shall use them tacitly throughout the following.

\begin{lemma} \label{lemma: basic properties of prec}
\begin{enumerate}
\item
Suppose that $\xi (u,v) \prec  \zeta  (u,v)$ uniformly in $u \in U$ and $v \in V$. If $\abs{V} \leq N^C$ for some constant $C$ then
\begin{equation*}
\sum_{v \in V} \xi(u,v) \;\prec\; \sum_{v \in V} \zeta(u,v)
\end{equation*}
uniformly in $u$.
\item
Suppose that $\xi_{1}(u) \prec \zeta_{1}(u)$ uniformly in $u$ and $\xi_{2}(u) \prec \zeta_{2}(u)$ uniformly in $u$. Then
\begin{equation*}
\xi_{1}(u) \xi_{2}(u) \;\prec\; \zeta_{1}(u) \zeta_{2}(u)
\end{equation*}
uniformly in $u$.
\item
Suppose that $\Psi(u) \geq N^{-C}$ is deterministic 
 and $\xi(u)$ is a nonnegative random variable satisfying 
 $\E \xi(u)^2 \leq N^{C}$ for all $u$. Then, provided that $\xi(u) \prec \Psi(u)$ uniformly in $u$, we have
\begin{equation*}
\E \xi(u) \;\prec\; \Psi(u)
\end{equation*}
uniformly in $u$.
\end{enumerate}
\end{lemma}
\begin{proof}
The claims (i) and (ii) follow from a simple union bound. For (iii), pick $\epsilon > 0$ and assume to simplify notation that $\xi$ and $\Psi$ do not depend on $u$. Then
\begin{equation*}
\E \xi \;=\; \E \xi \ind{\xi \leq N^{\epsilon} \Psi} + \E \xi \ind{\xi > N^{\epsilon} \Psi} \;\leq\; N^{\epsilon} \Psi + \sqrt{\E \xi^2} \sqrt{\P(\xi > N^\epsilon \Psi)} \;\leq\; N^\epsilon \Psi + N^{C/2 - D/2}\,,
\end{equation*}
for arbitrary $D > 0$. The claim (iii) therefore follows by choosing $D \geq 3 C$.
\end{proof}

Next, we give some basic facts about the Stieltjes transform $m_\phi$ of the Marchenko-Pastur law defined in \eqref{S_MP}. They have an especially simple form in the case $\phi \geq 1$; the complementary case $\phi < 1$ can be easily handled using \eqref{mphi_minvphi}. Recall the definition \eqref{def kappa sc} of $\kappa$. We record the following elementary properties of $m_\phi$, which may be proved e.g.\ by starting from the explicit form \eqref{S_MP}.

\begin{lemma} \label{lemma: mg}
For $z \in \f S$ and $\phi \geq 1$ we have
\begin{equation} \label{bounds on mg}
\abs{m_\phi(z)} \;\asymp\; 1 \,, \qquad \abs{1 - m_\phi(z)^2} \;\asymp\; \sqrt{\kappa + \eta}\,,
\end{equation}
and
\begin{equation} \label{im m gamma}
\im m_\phi(z) \;\asymp\;
\begin{cases}
\sqrt{\kappa + \eta} & \text{if $E \in [\gamma_-, \gamma_+]$}
\\
\frac{\eta}{\sqrt{\kappa + \eta}} & \text{if $E \notin [\gamma_-, \gamma_+]$}\,.
\end{cases}
\end{equation}
(All implicit constants depend on $\omega$ in the definition \eqref{def_S_theta} of $\f S$.)
\end{lemma}

¥
The basic object is the $M \times N$ matrix $X$. We use the indices $i,j = 1, \dots, M$ to denote the rows of $X$ and $\mu,\nu = 1, \dots, N$ to denote its columns. Unrestricted summations over Latin indices always run over the set $\{ 1, 2, \ldots , M\}$ while
Greek indices run over $\{1, 2,\ldots, N\}$.

In addition to the resolvent $R$ from \eqref{def_R}, we shall need another resolvent, $G$:
\begin{equation*}
G(z) \;\deq\; (XX^* - z)^{-1}\,, \qquad R(z) \;\deq\; (X^* X - z)^{-1}\,.
\end{equation*}
Although our main results only pertain to $R$, the resolvent $G$ will play a crucial role in the proofs, in which we consider both $X^* X$ and $X X^*$ in tandem.
In the following formulas the spectral parameter $z$ plays no explicit role, and we therefore omit it from the notation, as explained at the beginning of this section.

\begin{remark} \label{rem:index sets}
More abstractly, we may introduce two sets of indices, $I_{\rm population}$ and $I_{\rm sample}$, such that $I_{\rm population}$ indexes the entries of $G$ and $I_{\rm sample}$ the entries of $R$.
Elements of $I_{\rm population}$ are always denoted by Latin letters and elements of $I_{\rm sample}$ by Greek letters. These two sets are to be viewed as distinct. For convenience of notation, we shall always use the customary representations $I_{\rm population} \deq \{1, \dots, M\}$ and $I_{\rm sample} \deq \{1, \dots, N\}$, bearing in mind that neither should be viewed as a subset of the other. This terminology originates from the statistical application of sample covariance matrices. The idea is that we are observing the statistics of a population of size $M$ by making $N$ independent measurements (``samples'') of the population. Each observation is a column of $X$. Hence the population index $i$ labels the rows of $X$ and the sample index $\mu$ the columns of $X$.
\end{remark}

\begin{definition}[Removing rows] \label{def: removing rows}
For $T \subset \{1, \dots, M\}$ define
\begin{equation*}
(X^{(T)})_{i\mu} \;\deq\; \ind{i \notin T} X_{i\mu}\,.
\end{equation*}
Moreover, for $i,j \notin T$ we define the resolvents entries
\begin{equation*}
G^{(T)}_{ij}(z) \;\deq\; \pb{X^{(T)}(X^{(T)})^* - z}^{-1}_{ij}\,, \qquad
R^{(T)}_{\mu \nu}(z) \;\deq\; \pb{(X^{(T)})^*X^{(T)} - z}^{-1}_{\mu \nu}\,.
\end{equation*}
When $T = \{a\}$, we abbreviate $(\{a\})$ by $(a)$ in the above definitions; similarly, we write $(ab)$ instead of $(\{a,b\})$.
\end{definition}

We shall use the following expansion formulas for $G$. They are elementary consequences of Schur's complement formula; see e.g.\ Lemma 4.2 of \cite{EYY1} and Lemma 6.10 of \cite{EKYY2} for proofs of similar identities. 

\begin{lemma}[Resolvent identities for $G$] \label{lem: res identity 1}
For $i,j,k \notin T$ and $i,j \neq k$ we have
\begin{equation} \label{Gij Gijk sc}
G_{ij}^{(T)} \;=\; G_{ij}^{(Tk)} + \frac{G_{ik}^{(T)} G_{kj}^{(T)}}{G_{kk}^{(T)}}\,,
\qquad
\frac{1}{G_{ii}^{(T)}} =  \frac{1}{G_{ii}^{(Tk)}} - \frac{ G_{ik}^{(T)} G_{ki}^{(T)}}{  G_{ii}^{(T)} G_{ii}^{(Tk)} G_{kk}^{(T)}}\,.
\end{equation}
For $i \notin T$ we have
\begin{equation} \label{Gii expanded sc}
\frac{1}{ G_{ii}^{(T)}} \;=\; - z -  z \sum_{\mu,\nu} X_{i\mu} R^{(T i)}_{\mu \nu} X^*_{\nu i}\,.
\end{equation}
Moreover, for $i,j \notin T$ and $i \neq j$ we have
\begin{equation} \label{Gij expanded sc}
G_{ij}^{(T)} \;=\;  z  G_{ii}^{(T)}  G_{jj}^{(i T)} \sum_{\mu,\nu} X_{i\mu} R^{(T i j)}_{\mu \nu} X^*_{\nu j}\,.
\end{equation}
\end{lemma}

Definition \ref{def: removing rows} and Lemma \ref{lem: res identity 1} have the following analogues for removing columns and identities for $R$.

\begin{definition}[Removing columns] \label{def: removing columns}
For $T \subset \{1, \dots, N\}$ define
\begin{equation*}
(X^{[T]})_{i\mu} \;\deq\; \ind{\mu \notin T} X_{i\mu}\,.
\end{equation*}
Moreover, for $\mu,\nu \notin T$ we define the resolvents entries
\begin{equation*}
G^{[T]}_{ij}(z) \;\deq\; \pb{X^{[T]}(X^{[T]})^* - z}^{-1}_{ij}\,, \qquad
R^{[T]}_{\mu \nu}(z) \;\deq\; \pb{(X^{[T]})^*X^{[T]} - z}^{-1}_{\mu \nu}\,.
\end{equation*}
When $T = \{\mu\}$, we abbreviate $[\{\mu\}]$ by $[\mu]$ in the above definitions; similarly, we write $[\mu \nu]$ instead of $[\{\mu ,\nu\}]$.
\end{definition}

We use the following expansion formulas for $R$, which are analogoues to those of Lemma \ref{lem: res identity 1}.
\begin{lemma}[Resolvent identities for $R$] \label{lem: res identity 2}
For $\mu,\nu,\rho \notin T$ and $\mu,\nu \neq \rho$ we have
\begin{equation}\label{RT+}
R_{\mu \nu}^{[T]} \;=\; R_{\mu \nu}^{[T\rho]} + \frac{R_{\mu \rho}^{[T]} R_{\rho \nu}^{[T]}}{R_{\rho \rho}^{[T]}}\,, \qquad
\frac{1}{R_{\mu \mu}^{[T]}} \;=\; \frac{1}{R_{\mu \mu}^{[T \rho]}} - \frac{ R_{\mu\rho}^{[T]} R_{\rho\mu}^{[T]}}{  R_{\mu\mu}^{[T]} R_{\mu\mu}^{[T\rho]} R_{\rho\rho}^{[T]}}\,. 
\end{equation}
For $\mu \notin T$ we have
\begin{equation}\label{RTuu}
\frac{1}{R_{\mu \mu}^{[T]}} \;=\; - z -  z \sum_{i,j} X_{\mu i}^* G^{[T\mu]}_{ij} X_{j \mu}\,.
\end{equation}
Moreover, for $\mu,\nu \notin T$ and $\mu \neq \nu$ we have
\begin{equation}\label{RTuv}
R_{\mu \nu}^{[T]} \;=\;  z  R_{\mu \mu}^{[T]} R_{\nu \nu}^{[\mu T]} \sum_{i,j} X^*_{\mu i} G^{[T \mu \nu]}_{ij} X_{j \nu}\,.
\end{equation}
\end{lemma}

Here we draw attention to a fine notational distinction. Parentheses $(\cdot)$ in $X^{(T)}$ indicate removal of rows (indexed by Latin letters),
while square brackets $[\cdot]$ in $X^{[T]}$ indicate removal of columns (indexed by Greek letters). In particular, this notation makes it unambiguous whether $T$ is required to be a subset of $\{1, \dots, M\}$ or of $\{1, \dots, N\}$.

The following lemma is an immediate consequence of the fact that for $\phi \geq 1$ the spectrum of $X X^*$ is equal to the spectrum of $X^* X$ plus $M - N$ zero eigenvalues. (A similar result holds for for $\phi \leq 1$, and if $X$ is replaced with $X^{[T]}$ or $X^{(U)}$.) 

\begin{lemma} \label{lem: tr identity}
Let $T \subset \{1, \dots, N\}$ and $U \subset \{1, \dots, M\}$. Then we have
\begin{equation*}
\tr R^{[T]}-\tr G^{[T]}  \;=\; \frac{M-(N-\abs{T})}{z}
\end{equation*}
as well as
\begin{equation*}
\tr R^{(U)}-\tr G^{(U)} \;=\; \frac{(M-\abs{U})-N}{z}
\end{equation*}
\end{lemma}

In particular,  we find
\begin{equation} \label{trG_tr_R}
\frac{1}{M} \tr G \;=\; \frac{1}{\phi} \, \frac {1}{N} \tr R + \frac{1}{\phi} \frac{1 - \phi}{z}\,,
\end{equation}
in agreement with the
relation \eqref{mphi_minvphi} and the heuristics $M^{-1} \tr G \sim m_{\phi^{-1}}$ and $N^{-1} \tr R \sim m_\phi$.

The following lemma is an easy consequence of the well-known interlacing property of the eigenvalues of $X X^*$ and $X^{(i)} (X^{(i)})^*$, as well as the eigenvalues of $X^* X$ and $(X^{[\mu]})^* X^{[\mu]}$. 

\begin{lemma}[Eigenvalue interlacing] \label{lem: inter}
For any $T \subset \{1, \dots, N\}$ and $U \subset \{1, \dots, M\}$, there exists a constant $C$, depending only on $\abs{T}$ and $\abs{U}$, such that 
\begin{equation*}
\absb{\tr R^{[T]}-\tr R} \;\leq\; C\eta^{-1} \,, \qquad \absb{\tr R^{(U)}-\tr R} \;\leq\; C\eta^{-1}\,.
\end{equation*}
\end{lemma}

Finally, we record the fundamental identity
\begin{equation} \label{Ward}
\sum_{j} \absb{G_{ij}^{[T]}}^2 \;=\; \frac{1}{\eta} \im G_{ii}^{[T]}\,,
\end{equation}
which follows easily by spectral decomposition of $G^{[T]}$.

\subsection{Reduction to the case $\phi \geq 1$} \label{sec:phi_geq_1}
We shall prove Theorems \ref{thm: IMP gen} and \ref{thm: IMP outside gen} by restricting ourselves to the case $\phi \geq 1$ but considering both $X^* X$ and $X X^*$ simultaneously.  In this short section we give the details of this reduction. Define the control parameter
\begin{equation} \label{def of Psi MP}
\Psi(z) \;\equiv\; \Psi_\phi(z) \;\deq\; \sqrt{\frac{\im m_\phi(z)}{N\eta}}+\frac{1}{N\eta}\,.
\end{equation}

We shall in fact prove the following. Recall the definitions \eqref{def_S_theta} of $\f S$ and \eqref{def_S_theta_wt} and of $\wt {\f S}$.
\begin{theorem} \label{thm: IMP}
Suppose that \eqref{cond on entries of X}, \eqref{NM gen}, \eqref{moments of X-1}, and $\phi \geq 1$ hold.
Then
\begin{equation}\label{bound: Gij isotropic}
\absb{\scalar{\f v}{G(z) \f w} - m_{\phi^{-1}}(z) \scalar{\f v}{\f w}} \; \prec \; \frac{1}{\phi}\, \Psi(z)
\end{equation}
uniformly in $z \in \f S$ and any deterministic unit vectors $\f v, \f w \in \C^M$. 
Similarly,
\begin{equation}\label{bound: Rij isotropic}
\absb{\scalar{\f v}{R(z) \f w} - m_\phi(z) \scalar{\f v}{\f w}} \;\prec\; \Psi(z)
\end{equation}
uniformly in $z \in \f S$ and any deterministic unit vectors $\f v, \f w \in \C^N$.
Moreover, we have
\begin{equation}\label{bound: trR}
\absbb{\frac{1}{N}\tr R(z) - m_\phi(z)} \;\prec\;   \frac{1}{N\eta}\,, \qquad
\absbb{\frac{1}{M}\tr G(z) - m_{\phi^{-1}}(z)} \;\prec\;   \frac{1}{M\eta}
\end{equation}
uniformly in $z \in \f S$.
\end{theorem}

\begin{theorem} \label{thm: IMP outside}
Suppose that \eqref{cond on entries of X}, \eqref{NM gen}, \eqref{moments of X-1}, and $\phi \geq 1$ hold. Then
\begin{equation}\label{bound: Gij isotropic outside sc}
\absb{\scalar{\f v}{G(z) \f w} - m_{\phi^{-1}}(z) \scalar{\f v}{\f w}} \; \prec \; \frac{1}{\phi} \, \sqrt{\frac{\im m_\phi(z)}{N\eta}}
\end{equation}
uniformly in $z \in \wt {\f S}$ and any deterministic unit vectors $\f v, \f w \in \C^M$. Similarly,
\begin{equation}\label{bound: Rij isotropic outside sc}
\absb{\scalar{\f v}{R(z) \f w} - m_\phi(z) \scalar{\f v}{\f w}} \;\prec\; \sqrt{\frac{\im m_\phi(z)}{N\eta}}
\end{equation}
uniformly in $z \in \wt{\f S}$ and any deterministic unit vectors $\f v, \f w \in \C^N$.
\end{theorem}

Let $\phi \geq 1$, i.e.\ $N\leq M$.
Recall that $\f u^{(1)}, \dots, \f u^{(N)} \in \C^N$ denote the normalized eigenvectors of $X^* X$ associated with the nontrivial eigenvalues $\lambda_1, \dots, \lambda_N$, and let  $\wt {\f u}\!\,^{(1)}, \dots, \wt {\f u}\!\,^{(N)} \in \C^M$ denote
 the normalized eigenvectors of $X X^*$ associated with the same eigenvalues $\lambda_1, \dots, \lambda_N$.
\begin{theorem} \label{thm: I deloc}
Suppose that \eqref{cond on entries of X}, \eqref{NM gen}, \eqref{moments of X-1}, and $\phi \geq 1$ hold. For any $\epsilon > 0$ 
we have the bounds
\begin{equation}\label{smfy}
\abs{\scalar{\f u^{(\alpha)}}{\f v}}^2 \;\prec\; N^{-1}\,, \qquad \abs{\scalar{\wt{\f u}\,\!^{(\alpha)}}{\f w}}^2 \;\prec\; M^{-1}\,
\end{equation}
uniformly for $\alpha \leq (1 - \epsilon) N$ and all normalized $\f v \in \C^N$ and $\f w \in \C^M$.
If in addition $\phi \geq 1 + c$ for some constant $c > 0$, then \eqref{smfy} holds uniformly for all $\alpha \leq N$. 
\end{theorem}

Theorems \ref{thm: IMP gen}, \ref{thm: IMP outside gen}, and \ref{thm: I deloc gen} are easy consequences of Theorems \ref{thm: IMP}, \ref{thm: IMP outside}, and \ref{thm: I deloc} respectively, combined with the observation that
\begin{equation} \label{im_m_swap}
\im m_{\phi^{-1}}(z) \;\asymp\; \frac{1}{\phi} \im m_\phi(z)
\end{equation}
for $z \in \f S$; in addition, the asymptotic equivalence in \eqref{bound: Rij isotropic outside sc gen} follows from \eqref{im m gamma} and \eqref{im_m_swap}. The estimate \eqref{im_m_swap} itself can be proved by noting that \eqref{mphi_minvphi} implies
\begin{equation*}
\im m_{\phi^{-1}}(z) \;=\; \frac{1}{\phi} \pbb{\im m_\phi(z) + \frac{\phi-1}{\abs{z}^2} \eta}\,.
\end{equation*}
Since $\abs{z}^2 \asymp \phi$ for $z \in \f S$, \eqref{im_m_swap} for $\phi \geq 1$ follows from Lemma \ref{lemma: mg}. Replacing $\phi$ with $\phi^{-1}$ in \eqref{im_m_swap}, we conclude that \eqref{im_m_swap} holds for all $\phi$. 

What remains therefore is to prove Theorems \ref{thm: cov-rig}, \ref{thm: IMP}, \ref{thm: IMP outside}, and \ref{thm: I deloc}. We shall prove Theorem \ref{thm: cov-rig} in Section \ref{sec:rigi_proof}, Theorem \ref{thm: IMP} in Section \ref{sec:4}, and Theorems \ref{thm: IMP outside} and \ref{thm: I deloc} in Section \ref{sec:5}.

\section{The entrywise local Marchenko-Pastur law} \label{sec: MP law}

In this section we prove a entrywise version of Theorem \ref{thm: IMP}, in which the vectors $\f v$ and $\f w$ from \eqref{bound: Rij isotropic} and \eqref{bound: Gij isotropic} are assumed to lie in the direction of a coordinate axis. A similar result was previously proved in \cite[Theorem 3.1]{PY}. Recall the definition of $\Psi$ from \eqref{def of Psi MP}.

\begin{theorem}[Entrywise local Marchenko-Pastur law]\label{thm: cov-loc}
Suppose that \eqref{cond on entries of X}, \eqref{NM gen}, \eqref{moments of X-1}, and $\phi \geq 1$ hold.
Then
\begin{equation}\label{bound: Rij}
\absb{R_{\mu \nu}(z)-\delta_{\mu \nu} m_\phi(z)}  \;\prec\;   \Psi(z)\,,
\end{equation}
uniformly in $z\in \f S$ and $\mu,\nu \in \{1, \dots, N\}$.
Similarly,
\begin{equation}\label{bound: Gij}
\absb{G_{ij}(z)- \delta_{ij} m_{\phi^{-1}}(z)} \;\prec\; \frac{1}{\phi} \Psi(z)\,,
\end{equation}
uniformly in $z \in \f S$ and $i,j \in \{1, \dots, M\}$.
Moreover, \eqref{bound: trR} holds uniformly in $z \in \f S$.
\end{theorem}

Theorem \ref{thm: cov-loc} differs from Theorem 3.1 in \cite{PY} in the following two ways. 
\begin{enumerate}
\item
The restriction $1 \leq \phi \leq C$ in \cite{PY} is relaxed to $1 \leq \phi \leq N^C$  (and hence, as explained in Section \ref{sec:phi_geq_1}, to $N^{-C} \leq \phi \leq N^C$). 
\item
The uniform subexponential decay assumption of \cite{PY} is relaxed to \eqref{moments of X-1}. On the other hand, thanks to the stronger subexponential decay assumption the statement of Theorem 3.1 of \cite{PY} is slightly stronger than Theorem \ref{thm: cov-loc}: in Theorem 3.1 of \cite{PY}, the error bounds $N^\epsilon$ in the definition of $\prec$ are replaced with $(\log N)^{C \log \log N}$.
\end{enumerate}

The difference (ii) given above is technical and amounts to using Lemma \ref{lemma: LDE}, which is tailored for random variables satisfying \eqref{moments of X-1}, for the large deviation estimates. We remark that all of the arguments of the current paper
may be translated to the setup of \cite{PY}, explained in (ii) above, by modifying the definition of $\prec$. 
The essence of  the proofs remains unchanged;  the only nontrivial difference is
 that in Section \ref{sec:4} we have to control moments  whose power depends
 weakly on $N$; this entails keeping track of some  basic combinatorial bounds.
 We do not pursue this modification any further. 

The difference (i) is more substantial, and requires to keep track of the $\phi$-dependence of all appropriately rescaled
 quantities throughout the proof.
 In addition, we take this opportunity to simplify and streamline the argument from \cite{PY}.
 This provides a short and self-contained proof of Theorem 4.1,
 up to a fluctuation averaging result, Lemma \ref{lem: fluct_avg} below, which was proved in the current simple and general form in \cite{EKYY4}.

\subsection{A weak local Marchenko-Pastur law}

We begin with the proof of \eqref{bound: Rij} and \eqref{bound: trR}. For the following it will be convenient to use the rescaled spectral parameters
\begin{equation} \label{def_rescaled_z}
\wt z \;\deq\; \phi^{-1/2} z\,, \qquad \wh z \;\deq\; z - \phi^{1/2} + \phi^{-1/2}\,.
\end{equation}
Using $\wt z$ and $\wh z$ we may write the the defining equation \eqref{identity for m MP} of $m_\phi$ as
\begin{equation} \label{MP_id_z}
m(z)+\frac{1}{\wh z+\wt z m(z)} \;=\; 0\,.
\end{equation}
From the definition \eqref{def_S_theta} of $\f S$, we find
\begin{equation} \label{rescaled_z}
\abs{\wt z} \;\asymp\; 1 \,, \qquad \abs{\wh z} \;\leq\; C
\end{equation}
for all $z \in \f S$. We remark that, as in \cite{PY}, the Stieltjes transform $m_\phi$ satisfies $\abs{m_\phi(z)} \asymp 1$ for $z \in \f S$; see \eqref{bounds on mg}.

We define the $z$-dependent random control parameters
\begin{equation} \label{def_control_param}
\Lambda(z) \;\deq\; \max_{\mu, \nu} \abs{R_{\mu \nu}(z) - \delta_{\mu \nu} m_\phi(z)}\,, \qquad \Lambda_o(z) \;\deq\; \max_{\mu \neq \nu} \abs{R_{\mu \nu}(z)}\,, \qquad \Theta(z) \;\deq\; \abs{m_R(z) - m_\phi(z)}\,,
\end{equation}
where we defined the Stieltjes transform of the empirical density of $X^* X$,
\begin{equation*}
m_R(z) \;\deq\; \frac{1}{N} \tr R(z)\,. 
\end{equation*}

The goal of this subsection is to prove the following weaker variant of Theorem \ref{thm: cov-loc}.

\begin{proposition} \label{prop:weak_law}
We have $\Lambda \prec (N \eta)^{-1/4}$ uniformly in $z \in \f S$.
\end{proposition}

The rest of this subsection is devoted to the proof of Proposition \ref{prop:weak_law}. We begin by introducing the basic $z$-dependent event
\begin{equation*}
\Xi(z) \;\deq\; \hb{\Lambda(z) \leq (\log N)^{-1} }\,.
\end{equation*}

\begin{lemma}\label{lem: new 6.6}
For any $\ell \in \N$ there exists a constant $C \equiv C_\ell$ such that for $z \in \f S$, all $T \subset \{1,2,\ldots, N\}$ satisfying $\abs{T} \leq \ell$, and all $\mu,\nu \notin T$ we have
\begin{equation} \label{nxnwmyqzdnh}
\ind{\Xi} \absb{R_{\mu \nu}^{[T]} - R_{\mu \nu}} \;\leq\; C \Lambda_o^2
\end{equation}
and
\begin{equation}
\ind{\Xi} C^{-1} \;\leq\; \ind{\Xi} \abs{R_{\mu \mu}^{[T]}} \;\leq\; C
\end{equation}
for large enough $N$ depending on $\ell$.
\end{lemma}
\begin{proof}
The proof is a simple induction argument using \eqref{RT+} and the bound $\abs{m_\phi} \geq c$ from \eqref{bounds on mg}. We omit the details.
\end{proof}

As in the works \cite{EYY3, PY}, the main idea of the proof is to derive a self-consistent equation for $m_R = \frac{1}{N} \sum_\mu R_{\mu \mu}$ using the resolvent identity \eqref{RTuu}. To that end, we introduce the conditional expectation
\begin{equation} \label{cond_exp}
\E^{[\mu]}  (\cdot) \;\deq\; \E \p{\,\cdot\, | X^{[\mu]}}\,,
\end{equation}
i.e.\ the partial expectation in the randomness of the $\mu$-th column of $X$. 
We define
\begin{equation} \label{eqn:Zi}
Z_\mu \;\deq\; \pb{1-\E^{[\mu]}}z \sum_{i,j} X_{\mu i}^* G^{[\mu]}_{ij} X_{j \mu}
\;=\; z \sum_{i, j} X_{\mu i}^* G^{[\mu]}_{ij} X_{j \mu}  - \frac{\wt z}{N} \tr  G^{ [\mu]}\,,
\end{equation}
where in the last step we used \eqref{cond on entries of X} and  \eqref{def_rescaled_z}. 
Using \eqref{RTuu} with $T=\emptyset$, Lemma~\ref{lem: tr identity}, and \eqref{def_rescaled_z}, we find
\begin{equation} \label{futu}
\frac{1}{R_{\mu\mu}}
\;=\; -z - \frac{\wt z}{N}\tr  G^{ [\mu]} -  Z_\mu
\;=\; -\wh z- \frac{\wt z}{N} \tr  R^{ [\mu]} -  Z_\mu - \frac{1}{\sqrt{\phi} N}\,.
\end{equation}
The following lemma contains the key estimates needed to control the error terms $Z_\mu$ and $\Lambda_o$. The errors are controlled using of the (random) control parameter
\be \label{eqn:defpsi}
\Psi_\Theta \;\deq\; \sqrt{\frac{\im m_\phi+\Theta}{N\eta}}\,,
\ee
whose analogue in the context of Wigner matrices first appeared in \cite{EYY3}.

\begin{lemma} \label{lem: new 6.8}
For $z \in \f S$ we have
\be\label{ajaj}
\ind{\Xi} \pb{\abs{Z_\mu} + \Lambda_o} \;\prec\; \Psi_\Theta
\ee
as well as
\be\label{ajaj2}
\indb{\eta \geq (\log N)^{-1}} \, \pb{\abs{Z_\mu} + \Lambda_o} \;\prec\; \Psi_\Theta\,.
\ee
\end{lemma}

\begin{proof}
The proof is very similar to that of Theorems 6.8 and 6.9 of \cite{PY}. We consequently only give the details for the estimate of $\Lambda_o$; the estimate of $Z_\mu$ is similar.

For $\mu \neq \nu$ we use \eqref{RTuv} with $T=\emptyset$
to expand $R_{\mu \nu}$. Conditioning on $X^{[\mu \nu]}$ and invoking \eqref{two-set LDE} yields
\begin{equation} \label{R_munu_estimate}
\abs{R_{\mu \nu}} \;\prec\; \absb{R_{\mu \mu} R_{\nu \nu}^{[\mu]}} \, \frac{\abs{\wt z}}{N} \, \sqrt{\sum_{i,j} \absb{G_{ij}^{[\mu \nu]}}^2}\,.
\end{equation}
On the event $\Xi$, we estimate the right-hand side using
\begin{multline} \label{R_munu_estimate_1}
\ind{\Xi} \frac{\abs{\wt z}}N \sqrt{\sum_{i,j} \absb{G_{ij}^{[\mu \nu]}}^2}  \;=\; \ind{\Xi} \frac{\abs{\wt z}}N \sqrt{\frac{\im \tr G^{[\mu\nu]}}{ \eta}}
\;\leq\; 
\ind{\Xi} \frac{C}N \sqrt{\frac{\im \tr R^{[\mu\nu]}- ((\phi - 1)N + 2)  \im z^{-1}}{ \eta}}
\\
\leq\; C \ind{\Xi} \sqrt{\frac{\im m_\phi+\Theta +\Lambda_o^2 }{N\eta}+\frac{1}{N}}
\;\leq\; C \ind{\Xi} \sqrt{\frac{\im m_\phi+\Theta +\Lambda_o^2 }{N\eta}}\,,
\end{multline}
where the first step follows from \eqref{Ward}, the second from Lemma \ref{lem: tr identity}, the third from $\im z^{-1} = - \eta \abs{z}^{-2} \geq - C \eta /\phi$ and \eqref{nxnwmyqzdnh}, and the fourth from the fact that $\im m_\phi \geq c \eta$ by \eqref{im m gamma}.

Recalling \eqref{bounds on mg}, we have therefore proved that
\be\label{gjdzqyq}
\ind{\Xi} \Lambda_o \;\prec\; \ind{\Xi} \pb{\Psi_\Theta+(N\eta)^{-1/2}\Lambda_o}\,.
\ee
Since $(N \eta)^{-1/2} \leq N^{-\omega/2}$ on $\f S$, we find
\begin{equation*}
\ind{\Xi} \Lambda_o \;\prec\; \ind{\Xi} \Psi_\Theta\,,
\end{equation*}
which, together with the analogous bound for $Z_\mu$, concludes the proof of \eqref{ajaj}.

In order to prove the estimate $\Lambda_o \prec \Psi_\Theta$ from \eqref{ajaj2} for $\eta \geq (\log N)^{-1}$, we proceed similarly. From \eqref{R_munu_estimate} and the trivial deterministic bound $\absb{R_{\mu \mu} R_{\nu \nu}^{[\mu]}} \leq \eta^{-2} \leq (\log N)^2$ we get
\begin{equation*}
\abs{R_{\mu \nu}} \;\prec\; \frac{1}{N} \sqrt{\frac{\im \tr G^{[\mu\nu]}}{ \eta}}
\;=\; 
\frac{1}{N} \sqrt{\frac{\im \tr R^{[\mu\nu]}- ((\phi - 1)N + 2)  \im z^{-1}}{ \eta}}
\;\leq\; C \sqrt{\frac{\im m_\phi+\Theta}{N\eta}+\frac{1}{(N \eta)^2}}\,,
\end{equation*}
where the estimate is similar to \eqref{R_munu_estimate_1}, except that in the last step we use Lemma \ref{lem: inter} to estimate $\tr R^{\mu \nu} - \tr R$. 
Since $\eta \geq (\log N)^{-1}$, we easily find that $\abs{R_{\mu \nu}} \prec \Psi_\Theta$. This concludes the proof.
\end{proof}

As in \cite[Equation (6.13)]{PY}, in order to analyse the stability of the equation \eqref{MP_id_z} we introduce the operation $\cal D$ on functions $u \col \f S \to \C$, defined through
\begin{equation} \label{def_calD}
\cal D(u)(z) \;\deq\; \frac{1}{u(z)}+\wt zu(z)+\wh z\,.
\end{equation}
Note that, by \eqref{MP_id_z}, the function $m_\phi$ satisfies $\cal D(m_\phi) = 0$.

Next, we derive a stability result for $\cal D^{-1}$. Roughly, we prove that if $\cal D(u)$ is small then $u$ is close to $m_\phi$. Note that this result is entirely deterministic. It relies on a discrete continuity argument, whose essence is the existence of a sufficiently large gap between the two solutions of $\cal D(\cdot)= 0$.
Once this gap is established, then, together with the fact that $u$ is close to $m_\phi$ for large $\eta$, we may conclude that $u$ is close to $m_\phi$ for smaller $\eta$ as well. We use a discrete version of a continuity argument (as opposed to a continuous one used e.g.\ in \cite{PY}), which allows us to bypass several technical issues when applying it to estimating the random quantity $\abs{m_R - m_\phi}$. For more details of this application, see the explanation following \eqref{cont_main_step}.

For $z \in \f S$ introduce the discrete set
\begin{equation*}
L(z) \;\deq\; \{z\} \cup \hb{w \in \f S \col \re w = \re z \,,\, \im w \in [\im z, 1] \cap (N^{-5} \N)}\,.
\end{equation*}
Thus, if $\im z \geq 1$ then $L(z) = \{z\}$ and if $\im z \leq 1$ then $L(z)$ is a one-dimensional lattice with spacing $N^{-5}$ plus the point $z$. Clearly, we have the bound
\begin{equation} \label{size of L}
\abs{L(z)} \;\leq\; N^5\,.
\end{equation}
\begin{lemma}[Stability of $\cal D^{-1}$] \label{lem:stability}
There exists a constant $\epsilon > 0$ such that the following holds. Suppose that $\delta \col \f S \to \C$ satisfies $N^{-2} \leq \delta(z) \leq \epsilon$ for $z \in \f S$ and that $\delta$ is Lipschitz continuous with Lipschitz constant $N$. Suppose moreover that for each fixed $E$, the function $\eta \mapsto \delta(E + \ii \eta)$ is nonincreasing for $\eta > 0$. Suppose that $u \col \f S \to \C$ is the Stieltjes transform of a probability measure. Let $z \in \f S$, and suppose that for all $w \in L(z)$ we have
\begin{equation*}
\absb{\cal D(u)(w)} \;\leq\; \delta(w)\,.
\end{equation*}
Then we have
\begin{equation*}
\abs{u(z) - m_\phi(z)} \;\leq\; \frac{C \delta(z)}{\sqrt{\kappa + \eta + \delta(z)}}\,.
\end{equation*}
for some constant $C$ independent of $z$ and $N$. 
\end{lemma}

\begin{proof}
Let $u$ be as in Lemma \ref{lem:stability}, and abbreviate $\cal R \deq \cal D(u)$. Hence, by assumption on $u$, we have $\abs{\cal R} \leq \delta$.
We introduce $u_1 \equiv u_1^{\cal R}$ and $u_2 \equiv u_2^{\cal R}$ by setting $u_1 \deq u$ and defining $u_2$ as the other solution of the quadratic equation $\cal D(u) = \cal R$. Note that each $u_i$ is continuous. Explicitly, for $\abs{\cal R} \leq 1/2$ we get
\begin{equation} \label{solved_quadr}
u_{1,2} \;=\; \frac{\cal R - \wh z \pm \ii \sqrt{(z - \lambda_{-,\cal R}) (\lambda_{+,\cal R} - z)}}{2 \wt z} \,, \qquad \lambda_{\pm, \cal R} \;\deq\; \phi^{1/2} + \phi^{-1/2} + \cal R \pm 2 \sqrt{1 + \phi^{-1/2} \cal R}\,,
\end{equation}
where the square root in $\sqrt{1 + \phi^{-1/2} \cal R}$ is the principal branch. (Note that the sign $\pm$ in the expression for $u_{1,2}$ bears no relation to the indices $1,2$, since we have not even specified which complex square root we take.)  In particular, for $\cal R=0$ we have
$\lambda_{\pm,\cal R=0} =\gamma_\pm$, defined in  \eqref{S_MP}.
 Observe that for any complex square root $\sqrt{\cdot}$ and $w,\zeta \in \C$ we have $\abs{\sqrt{w + \zeta} - \sqrt{w}} \leq (\abs{w} + \abs{\zeta})^{-1/2} \abs{\zeta}$ or $\abs{\sqrt{w + \zeta} + \sqrt{w}} \leq (\abs{w} + \abs{\zeta})^{-1/2} \abs{\zeta}$. We use these formulas to compare \eqref{solved_quadr} 
with a small $\cal R$ with \eqref{solved_quadr} with $\cal R=0$.
Thus we conclude from \eqref{solved_quadr} and \eqref{rescaled_z}
 that for $i = 1$ or for $i = 2$ we have
\begin{equation} \label{ui-m}
\abs{u_i - m_\phi} \;\leq\; \frac{C_0 \abs{\cal R}}{\sqrt{\kappa + \eta + \abs{\cal R}}}
\end{equation}
for some constant $C_0 \geq 2$. What remains is to show that \eqref{ui-m} holds for $i = 1$. We shall prove this using a continuity argument.

Note first that \eqref{solved_quadr} and \eqref{rescaled_z} yield
\begin{equation} \label{u1-u2}
C_1^{-1} \sqrt{(\kappa + \eta - \abs{\cal R})_+} \;\leq\; \abs{u_1 - u_2} \;\leq\; C_1 \sqrt{\kappa + \eta + \abs{\cal R}}
\end{equation}
for some constant $C_1 \geq 1$.

Now consider $z = \ii$. Clearly, for $\cal R(\ii) = 0$ we have $u_1^0(\ii) = m_\phi(\ii)$. Note that by the lower bound of \eqref{u1-u2} the two roots $u_1^{\cal R}(\ii)$ and $u_2^{\cal R}(\ii)$ are distinct, and they are continuous in $\cal R$. Therefore there is an $\epsilon \in (0,1/2]$ such that for $\abs{\cal R(\ii)} \leq \epsilon$ we have, after possibly increasing $C_0$, that 
\begin{equation} \label{u_1_at_t}
\abs{u_1 - m_\phi} \;\leq\; C_0 \abs{\cal R}
\end{equation}
at $z = \ii$.
Next, we note that \eqref{ui-m} and \eqref{u1-u2} imply, for any $z$ with $\im z \geq 1$, that $\abs{u_i - m_\phi} \leq C_0 \abs{\cal R}$ for some $i \in \{1,2\}$, and that $\abs{u_1 - u_2} \geq (2 C_1)^{-1}$. Hence, requiring that $\epsilon \leq (8 C_0 C_1)^{-1}$ we find from \eqref{u_1_at_t} with $z = \ii$ and using the continuity of $u_1$ that \eqref{u_1_at_t} holds provided $\im z \geq 1$.

Next, for arbitrary $z \in \f S$ with $\im z < 1$ we consider two cases, depending on whether
\begin{equation} \label{two_cases}
\frac{C_0 \delta}{\sqrt{\kappa + \eta + \delta}} \;\leq\; \frac{1}{4  C_1} \sqrt{(\kappa + \eta - \delta)_+}
\end{equation}
holds or not. If \eqref{two_cases} does not hold, then we have $\kappa + \eta \leq 4  C_0 C_1 \delta$, so that \eqref{ui-m}, 
$\abs{\cal R} \leq \delta$, and the upper bound of \eqref{u1-u2} imply
\begin{equation*}
\abs{u_1 - m_\phi} \;\leq\; \frac{C_0 \abs{\cal R}}{\sqrt{\kappa + \eta + \delta}} + C_1 \sqrt{\kappa + \eta + \abs{\cal R}} \;\leq\; C \sqrt{\delta} \;\leq\; \frac{C \delta}{\sqrt{\kappa + \eta + \delta}}\,.
\end{equation*}

What remains is the case where \eqref{two_cases} holds. We use a continuity argument along the set $L(z)$, which we parametrize as $L(z) = \{z_0, \dots, z_L\}$, where $\im z_0 = 1$, $z_L = z$, and $\im z_{l+1} < \im z_l$. Note that $\abs{z_{l+1} - z_l} \leq N^{-5}$. By assumption, $\abs{\cal R} \leq \delta$ at each $z_l \in L(z)$, so that \eqref{ui-m} and \eqref{u1-u2} yield
\begin{equation} \label{tools for induction}
\exists \, i = 1,2 \;\col\; \abs{u_i - m_\phi} \;\leq\; \frac{C_0 \delta}{\sqrt{\kappa + \eta + \delta}} \,, \qquad
C_1^{-1} \sqrt{(\kappa + \eta - \delta)_+} \;\leq\; \abs{u_1 - u_2}
\end{equation}
at each $z_l \in L(z)$. (Here the quantities $\kappa \equiv \kappa(z_l)$, $\eta \equiv \eta(z_l)$, and $\delta \equiv \delta(z_l)$ are understood as functions of the spectral parameters $z_l$.) 
 Moreover, since \eqref{two_cases} holds at $z = z_L$, by the monotonicity assumption on $\delta$ we find that \eqref{two_cases} holds for all $z_l \in L(z)$. We now prove that
\begin{equation} \label{est_z_l}
\abs{u_1(z_l) - m_\phi(z_l)} \;\leq\; \frac{C_0 \delta}{\sqrt{\kappa + \eta + \delta}} \biggr|_{z_l}
\end{equation}
for all $l = 1, \dots, L$ by induction on $l$. For $l = 0$ the bound \eqref{est_z_l} is simply \eqref{u_1_at_t} proved above.
 Suppose therefore that \eqref{est_z_l} holds for some $l$. Since $u_1$ and $m_\phi$ are Lipschitz continuous with Lipschitz constant $N$, we get
\begin{multline} \label{main induction step}
\abs{u_1(z_{l + 1}) - m_\phi(z_{l + 1})} \;\leq\; 2 N N^{-5} + \absb{u_1(z_l) - m_\phi(z_l)}
\;\leq\; 2 N^{-4} + \frac{C_0 \delta}{\sqrt{\kappa + \eta + \delta}} \biggr|_{z_l}
\\
\leq\; 2 N^{-4} + C N^2 N^{-5} + \frac{C_0 \delta}{\sqrt{\kappa + \eta + \delta}} \biggr|_{z_{l + 1}}
\;\leq\; \frac{2 C_0 \delta}{\sqrt{\kappa + \eta + \delta}} \biggr|_{z_{l + 1}}\,,
\end{multline}
where in the second step we used the induction assumption, in the third step the Lipschitz continuity of $\delta$ and the bound $\eta \geq N^{-1}$, and in the last step the bounds $\delta \geq N^{-2}$ and $\kappa + \eta + \delta \leq C$. Next, recalling \eqref{two_cases}, it is easy to deduce \eqref{est_z_l} with $l$ replaced by $l+1$, using the bounds \eqref{tools for induction} and \eqref{main induction step}. This concludes the proof.
\end{proof}

We may now combine the probabilistic estimates from Lemma \ref{lem: new 6.8} with the stability of $\cal D^{-1}$ from Lemma \ref{lem:stability} to get the following result for $\eta \geq 1$, which will be used as the starting estimate in the bootstrapping in $\eta$.

\begin{lemma}\label{lem: new6.10}
We have $\Lambda \prec N^{-1/4}$ uniformly in $z \in \f S$ satisfying $\im z \geq 1$.
\end{lemma}

\begin{proof}
Let $z \in \f S$ with $\im z \geq 1$. From \eqref{futu} and the estimate on $Z_\mu$ from \eqref{ajaj2} we find
\begin{equation} \label{dxmm2}
R_{\mu\mu} \;=\; \frac{1}{-\wh z-\wt zm_R+O_\prec(\Psi_\Theta+N^{-1}) } \;=\; \frac{1}{-\wh z-\wt zm_R+O_\prec(N^{-1/2}) }\,,
\end{equation}
where in the last step we used that $\Psi_\Theta = O(N^{-1/2})$ since $\eta \geq 1$ and $\im m_\phi+\Theta=O(1)$, as follows from \eqref{bounds on mg} and the trivial bound $\abs{m_R} \leq C$. Taking the average over $\mu$ yields $1/m_R = -\wh z - \wt z m_R + O_\prec(N^{-1/2})$, i.e.\ $\abs{\cal D(m_R)} \prec N^{-1/2}$; see \eqref{def_calD}. Since $L(z) = \{z\}$, we therefore get from Lemma \ref{lem:stability} that $\abs{m_R-m_\phi} \prec N^{-1/4}$. Returning to \eqref{dxmm2} and recalling \eqref{MP_id_z} and \eqref{bounds on mg}, we get $\abs{R_{\mu \mu} - m_\phi} \prec N^{-1/4}$. Together with the estimate on $\Lambda_o$ from \eqref{ajaj2}, we therefore get $\Lambda \prec N^{-1/4}$ uniformly in $z \in \f S$  satisfying $\im z \geq 1$.
\end{proof}

Next, we plug the estimates from Lemma \ref{lem: new 6.8} into \eqref{futu} in order to obtain estimates on $m_R$. The summation in $m_R = \frac{1}{N}\sum_\mu R_{\mu \mu}$ will give rise to an error term of the form
\begin{equation} \label{Z_avg}
[Z] \;\deq\; \frac{1}{N} \sum_\mu Z_\mu\,.
\end{equation}
For the proof of Proposition \ref{prop:weak_law}, it will be enough to estimate $\abs{[Z]} \leq \max_\mu \abs{Z_\mu}$, but for the eventual proof of Theorem \ref{thm: cov-loc}, we shall need to exploit cancellation in the averaging in $[Z]$. Bearing this in mind, we state our estimates in terms of $[Z]$ to avoid repeating the following argument in Section \ref{sec:fluct-avg}.

\begin{lemma}\label{lem: new 6.11} We have
\begin{equation} \label{Rmumu-mR}
\ind{\Xi} \abs{R_{\mu \mu} - m_R} \;\prec\; \Psi_\Theta
\end{equation}
uniformly in $\mu$ and $z \in \f S$, as well as
\be\label{res 611}
\ind{\Xi} \cal D(m_R) \;=\; \ind{\Xi} \pb{-[Z] + O_\prec(\Psi_\Theta^2)}
\ee
uniformly in $z \in \f S$.
\end{lemma}

\begin{proof}
From \eqref{futu}, \eqref{nxnwmyqzdnh}, and \eqref{ajaj} we get
\begin{equation} \label{Xi_Rmumu}
\ind{\Xi} \frac{1}{R_{\mu \mu}} \;=\; \ind{\Xi} \pB{- \wh z - \wt z m_R - Z_\mu + O_\prec(\Psi_\Theta^2)}\,,
\end{equation}
where we absorbed the error term $N^{-1}$ on the right-hand side of \eqref{futu} into $\Psi_\Theta^2$ using \eqref{im m gamma}. Thus, using \eqref{ajaj} we get
\begin{equation*}
\ind{\Xi} (R_{\mu \mu} - R_{\nu \nu}) \;=\; \ind{\Xi} R_{\mu \mu} R_{\nu \nu} O_\prec(\Psi_\Theta) \;=\; O_\prec(\Psi_\Theta)\,.
\end{equation*}
Hence \eqref{Rmumu-mR} follows. Next, expanding $R_{\mu \mu} = m_R + (R_{\mu \mu} - m_R)$ yields
\begin{equation*}
\frac{1}{R_{\mu \mu}} \;=\; \frac{1}{m_R} - \frac{1}{m_R^2} (R_{\mu \mu} - m_R) + \frac{1}{m_R^2} (R_{\mu \mu} - m_R)^2 \frac{1}{R_{\mu \mu}}\,.
\end{equation*}
After taking the average $[\cdot] = \frac{1}{N} \sum_{\mu} \cdot \, $ the second term on the right-hand side vanishes. Taking the average of \eqref{Xi_Rmumu} therefore yields, using \eqref{Rmumu-mR} and \eqref{lem: new 6.6},
\begin{equation*}
\ind{\Xi} \pbb{\frac{1}{m_R} + O_\prec (\Psi_\Theta^2)} \;=\; \ind{\Xi} \pB{- \wh z - \wt z m_R - [Z] + O_\prec(\Psi_\Theta^2)}\,,
\end{equation*}
from which the claim follows.
\end{proof}

From \eqref{res 611} and \eqref{ajaj} we get
\begin{equation} \label{bootstrap_step}
\ind{\Xi} \abs{\cal D(m_R)} \;\prec\; \ind{\Xi} \Psi_\Theta \;\leq\; C(N \eta)^{-1/2}
\end{equation}
uniformly in $\f S$.
In order to conclude the proof of Proposition \ref{prop:weak_law}, we use a continuity argument. The main ingredients are \eqref{bootstrap_step}, Lemma \ref{lem:stability}, Lemma \ref{lem: new6.10}. Choose $\epsilon < \omega / 4$ and an arbitrary $D > 0$. It is convenient to introduce the random function
\begin{equation*}
v(z) \;\deq\; \max_{w \in L(z)} \Lambda(w) (N \im w)^{1/4}\,.
\end{equation*}
Our goal is to prove that with high probability there is a gap in the range of $v$, i.e.
\begin{equation} \label{cont_main_step}
\P \pb{v(z) \leq N^\epsilon \,,\, v(z) > N^{\epsilon/2}} \;\leq\; N^{-D + 5}
\end{equation}
for all $z \in \f S$ and large enough $N \geq N_0(\epsilon, D)$.
This equation says that with high probability the range of $v$ has a gap: it cannot take values in the interval $(N^{\epsilon/2}, N^\epsilon]$. 

The basic idea behind the proof of \eqref{cont_main_step} is to use the deterministic result from Lemma \ref{lem:stability} to propagate smallness of the random variable $\Lambda(z)$ from large values of $\eta$ to smaller values of $\eta$. Since we are dealing with random variables, one has to keep track of probabilities of exceptional events. To that end, we only work on a discrete set of values of $\eta$, which allows us to control the exceptional probabilities by a simple union bound. We remark that the first instance of such a \emph{stochastic continuity
argument} combined with stability of a self-consistent equation was
given in \cite{ESY2} in the context of Wigner matrices. Over the years it has been improved
through several papers in the context of Wigner matrices
\cite{EKYY4, EYY1, EYY3} as well as in the context
of sample covariance matrices \cite{ESYY, PY}.

Next, we prove \eqref{cont_main_step}.
Since $\{v(z) \leq N^\epsilon\} \subset \Xi(z) \cap \Xi(w)$ for all $z \in \f S$ and $w \in L(z)$, we find that \eqref{bootstrap_step} implies for all $z \in \f S$ and $w \in L(z)$ that
\begin{equation*}
\P \pB{v(z) \leq N^\epsilon \,,\, \abs{\cal D(m_R)(w)} (N \im w)^{1/2} > N^{\epsilon/2}} \;\leq\; N^{-D}
\end{equation*}
for large enough $N \geq N_0(\epsilon, D)$ (independent of $z$ and $w$). Using \eqref{size of L} and a union bound, we therefore get
\begin{equation*}
\P \pB{v(z) \leq N^\epsilon \,,\, \max_{w \in L(z)}\abs{\cal D(m_R)(w)} (N \im w)^{1/2} > N^{\epsilon/2}} \;\leq\; N^{-D + 5}\,.
\end{equation*}
Next, we use Lemma \ref{lem:stability} with $u = m_R$ and $\delta = N^{\epsilon / 2} (N \eta)^{-1/2}$ to get
\begin{equation*}
\P \pB{v(z) \leq N^\epsilon \,,\, \max_{w \in L(z)} \Theta(w) (N \im w)^{1/4} > N^{\epsilon/4}} \;\leq\; N^{-D + 5}\,.
\end{equation*}
(Here we used the trivial observation that the conclusion of Lemma \ref{lem:stability} is valid not only at $z$ but in the whole set $L(z)$.) Using \eqref{ajaj} and \eqref{Rmumu-mR} we therefore get \eqref{cont_main_step}.

We conclude the proof of Proposition \ref{prop:weak_law} by combining \eqref{cont_main_step} and Lemma \ref{lem: new6.10} with a continuity argument, similar to the proof of \cite[Proposition 5.3]{EKYY4}. We choose a lattice $\Delta \subset \f S$ such that $\abs{\Delta} \leq N^{10}$ and for each $z \in \f S$ there exists a $w \in \Delta$ satisfying $\abs{z - w} \leq N^{-4}$. Then \eqref{cont_main_step} combined with a union bound yields
\begin{equation} \label{cont_step2}
\P \pb{\exists w \in \Delta \col v(w) \in (N^{\epsilon/2}, N^\epsilon]} \;\leq\; N^{-D + 15}\,.
\end{equation}
From the definitions of $\Lambda$ and $\f S$ we find that $v$ is Lipschitz continuous on $\f S$, with Lipschitz constant $N^2$. Hence \eqref{cont_step2} and the definition of $\Delta$ imply
\begin{equation} \label{cont_step3}
\P \pb{\exists z \in \f S \col v(z) \in (2 N^{\epsilon/2}, 2 N^\epsilon/2]} \;\leq\; N^{-D + 15}\,.
\end{equation}
By Lemma \ref{lem: new6.10}, we have
\begin{equation} \label{cont_step4}
\P\pb{v(z) > N^{\epsilon/2}} \;\leq\; N^{-D + 15}
\end{equation}
for some (in fact any) $z \in \f S$ satisfying $\im z \geq 1$. It is not hard to infer from \eqref{cont_step3} and \eqref{cont_step4} that
\begin{equation} \label{result on v}
\P\pB{\max_{z \in \f S} v(z) > 2 N^{\epsilon/2}} \;\leq\; 2N^{-D + 15}\,.
\end{equation}
Since $\epsilon$ can be made arbitrarily small and $D$ arbitrarily large, Proposition \ref{prop:weak_law} follows from \eqref{result on v}.

\subsection{Fluctuation averaging and conclusion of the proof of Theorem \ref{thm: cov-loc}} \label{sec:fluct-avg}

In order to improve the negative power of $(N \eta)$ in Proposition \ref{prop:weak_law}, and hence prove the optimal bound in Theorem \ref{thm: cov-loc}, we shall use the following result iteratively. Recall the definition of $\Theta$ from \eqref{def_control_param} and definition of $[Z]$ from \eqref{Z_avg}.

\begin{lemma}\label{lem: new 7.1}
Let $\tau \in (0,1)$ and suppose that we have

\be\label{LPrec}
\Theta \;\prec\;  (N\eta)^{-\tau}
\ee
uniformly in $z \in \f S$.
Then we have
\be\label{ZPrec}
\abs{[Z]} \;\prec\; \frac{\im m_\phi+(N\eta)^{-\tau}}{N\eta}
\ee
uniformly in $z \in \f S$.
\end{lemma}

In order to prove Lemma \ref{lem: new 7.1}, we invoke the following fluctuation averaging result. 
We remark that the fluctuation averaging mechanism was first exploited in \cite{EYY2}. Here we use the result from \cite{EKYY4}, where a general version with a streamlined proof was given. Recall the definition of the partial expectation $\E^{[\mu]}$ from \eqref{cond_exp}.

\begin{lemma}[Fluctuation averaging \cite{EKYY4}] \label{lem: fluct_avg}
Suppose that $\Phi$ and $\Phi_o$ are positive, $N$-dependent, deterministic functions on $\f S$ satisfying $N^{-1/2} \leq \Phi, \Phi_o \leq N^{-c}$ for some constant $c > 0$. Suppose moreover that $\Lambda \prec \Phi$ and $\Lambda_o \prec \Phi_o$. Then
\begin{equation} \label{avg_conclusion}
\frac{1}{N} \sum_{\mu} \pb{1 - \E^{[\mu]}} \frac{1}{R_{\mu \mu}} \;=\; O_\prec(\Phi_o^2)\,.
\end{equation}
\end{lemma}
\begin{proof}
This result was given in a slightly different context in Theorem 4.7 in \cite{EKYY4}. However, it is a triviality that the proof of Theorem 4.7 in \cite{EKYY4} carries over word for word, provided one replaces $G_{ij}^{(T)}$ there with $R_{\mu \nu}^{[T]}$; see Remark B.3 in \cite{EKYY4}. The proof relies only on the identity \eqref{RT+}, which is the analogue of Equation (4.6) in \cite{EKYY4}.
\end{proof}

\begin{remark} \label{rem: fluct_avg}
The conclusion of Lemma \ref{lem: fluct_avg} remains true under somewhat more general hypotheses, whereby $\Lambda$ is not required to be small. Indeed, \eqref{avg_conclusion} holds provided that $\Phi_o$ is as in Lemma \ref{lem: fluct_avg} and that
\begin{equation*}
\absbb{\frac{1}{R_{\mu \mu}}} \;\prec\; 1 \,, \qquad \absbb{\pb{1 - \E^{[\mu]}} \frac{1}{R_{\mu \mu }}} \;\prec\; \Phi_o \,, \qquad \Lambda_o \;\prec\; \Phi_o\,.
\end{equation*}
The proof is the same as that of Theorem 4.7 in \cite{EKYY4}.
\end{remark}

\begin{proof}[Proof of Lemma \ref{lem: new 7.1}]
We apply Lemma \ref{lem: fluct_avg} to
\begin{equation} \label{Ztrick}
- Z_\mu \;=\; \pb{1 - \E^{[\mu]}} \frac{1}{R_{\mu \mu}}\,,
\end{equation}
where we used \eqref{RTuu}. From Proposition \ref{prop:weak_law} we get $1 \prec \ind{\Xi}$, i.e.\ $\Xi$ holds with high probability. Therefore \eqref{ajaj} yields $\Lambda_o \prec \Psi_\Theta$. Using this bound for $\Lambda_o$ and Proposition \ref{prop:weak_law} again to estimate $\Lambda \prec (N \eta)^{-1/4}$, we therefore get
\begin{equation*}
\Lambda_o \;\prec\; \Phi_o \,, \qquad \Lambda \;\prec\; \Phi \,, \qquad 
\Phi_o \;\deq\; \sqrt{\frac{\im m_\phi + (N \eta)^{-\tau}}{N \eta}}\,, \qquad
\Phi \;\deq\; (N \eta)^{-1/4}\,.
\end{equation*}
Using \eqref{im m gamma}, it is easy to check that these definitions of $\Phi_o$ and $\Phi$ satisfy the assumptions of Lemma \ref{lem: fluct_avg}. Hence the claim follows from Lemma \ref{lem: fluct_avg} and \eqref{Ztrick}.
\end{proof}

Now suppose that $\Theta \prec (N \eta)^{-\tau}$. From Lemma \ref{lem: new 7.1}, the fact that $1 - \ind{\Xi} \prec 0$ from Proposition \ref{prop:weak_law}, Lemma \ref{LPrec}, and \eqref{res 611}, we find
\begin{equation*}
\abs{\cal D(m_R)} \;\prec\; \frac{\im m_\phi + (N \eta)^{-\tau}}{N \eta}
\end{equation*}
uniformly in $z \in \f S$. Using \eqref{size of L} and a simple union bound,  we may invoke Lemma \ref{lem:stability} to get
\begin{equation*}
\Theta \;\prec\; \frac{\im m_\phi}{N \eta} \frac{1}{\sqrt{\kappa + \eta}} + \sqrt{\frac{(N \eta)^{-\tau}}{N \eta}} \;\leq\; \frac{C}{N \eta} + (N \eta)^{-1/2 - \tau /2} \;\leq\; C (N \eta)^{-1/2 - \tau /2}\,,
\end{equation*}
where in the second step we used \eqref{im m gamma}.
Summarizing, we have proved the self-improving estimate
\begin{equation} \label{self-improving}
\Theta \;\prec\; (N \eta)^{-\tau} \qquad \Longrightarrow \qquad \Theta \;\prec\; (N \eta)^{-1/2 - \tau /2}\,.
\end{equation}
From Proposition \ref{prop:weak_law} we know that $\Theta \prec (N \eta)^{-1/4}$. Thus, for any $\epsilon > 0$, we iterate \eqref{self-improving} an order $C_\epsilon$ times to get $\Theta \prec (N \eta)^{-1 + \epsilon}$. This concludes the proof of the first bound of  \eqref{bound: trR}. 
 The second bound of \eqref{bound: trR} follows from the first one and the identity \eqref{trG_tr_R}.

Next, \eqref{bound: Rij} follows from \eqref{bound: trR} and $1 - \ind{\Xi} \prec 0$, combined with \eqref{Rmumu-mR} and \eqref{ajaj}.

What remains is the proof of \eqref{bound: Gij}. To that end, in analogy to the partial expectation $\E^{[\mu]}$ defined above, we define $\E^{(i)}(\cdot) \deq \E(\cdot \vert X^{(i)})$. Introducing $1 = \E^{(i)} + \pb{1 - \E^{(i)}}$ into the right-hand side of \eqref{Gii expanded sc} yields
\begin{equation*}
\frac{1}{G_{ii}} \;=\; - z -  z \sum_{\mu,\nu} X_{i\mu} R^{(i)}_{\mu \nu} X^*_{\nu i} \;=\; -z - \frac{\wt z}{N} \sum_{\mu} R^{(i)}_{\mu \mu} - (1 - \E^i) z \sum_{\mu,\nu} X_{i\mu} R^{(i)}_{\mu \nu} X^*_{\nu i}\,.
\end{equation*}
Using Lemma \ref{lem: inter}, we rewrite the sum in the second term according to $N^{-1} \sum_\mu R_{\mu \mu}^{(i)} = m_R + O \pb{(N \eta)^{-1}}$.
Moreover, the third term is estimated exactly as $Z_\mu$ in Lemma \ref{lem: new 6.8}, using Lemma \ref{lemma: LDE}. Putting everything together yields
\begin{equation*}
G_{ii} \;=\; \frac{1}{-z-\wt zm_R +O_\prec(\Psi_\Theta)} \;=\; \frac{1}{1/ m_{\phi^{-1}} + O_\prec(\Theta + \Psi_\Theta)}\,,
\end{equation*}
as follows after some elementary algebra using \eqref{MP_id_z}.
Moreover, using \eqref{bounds on mg} it is not hard to see that $m_{\phi^{-1}} \asymp \phi^{-1/2}$
on the domain $\f S$. This yields $G_{ii} = m_{\phi^{-1}} + O(\phi^{-1} \Psi)$, where we used $\Theta \prec (N \eta)^{-1}$ to estimate $\Theta + \Psi_\Theta \prec \Psi$. This concludes the proof of \eqref{bound: Gij} for $i = j$.

In particular, $\abs{G_{ii}} \prec \phi^{-1/2}$. The same argument applied to the matrix $X^{(j)}$ instead of $X$ yields $\abs{G_{ii}^{(j)}} \prec \phi^{-1/2}$. Thus we get from \eqref{Gij expanded sc} that for $i \neq j$ we have
\begin{equation*}
\abs{G_{ij}} \;\prec\; \phi^{-1} \absbb{\sum_{\mu,\nu} X_{i\mu} R^{(i j)}_{\mu \nu} X^*_{\nu j}} \;\prec\; \phi^{-1} \Psi\,,
\end{equation*}
where the last step follows using \eqref{two-set LDE}, exactly as in the proof of Lemma \ref{lem: new 6.8}, and \eqref{bound: Rij}. This concludes the proof of \eqref{bound: Gij}, and hence of Theorem \ref{thm: cov-loc}.

\subsection{Proof of Theorem \ref{thm: cov-rig}} \label{sec:rigi_proof}
The proof of Theorem \ref{thm: cov-rig} is similar to that of Theorem 2.2 in \cite{EYY3} and Theorem 3.3 in \cite{PY}. We therefore only sketch the argument.
First we observe that, since the nontrivial eigenvalues $\lambda_1, \dots, \lambda_K$ of $X^* X$ and $X X^*$ coincide and
\begin{equation*}
N \int_{\gamma}^\infty \varrho_{\phi}(\dd x) \;=\; M \int_{\gamma}^\infty \varrho_{\phi^{-1}}(\dd x)
\end{equation*}
for all $\gamma > 0$, it suffices to prove Theorem \ref{thm: cov-rig} for $\phi \geq 1$, i.e.\ $K=N$.

Define the normalized counting functions
\begin{equation*}
n_\phi(E_1,E_2) \;\deq\; \int_{E_1}^{E_2} \varrho_\phi(\dd x)\,, \qquad
n(E_1, E_2) \;\deq\; \frac{1}{N} \absb{\h{\alpha \col E_1 \leq \lambda_\alpha \leq E_2}}\,.
\end{equation*}
The proof relies on the following key estimates.

\begin{lemma} \label{lem: rigi-tool}
We have
\begin{equation} \label{n-n_gamma}
\absb{n(E_1, E_2) - n_\phi(E_1, E_2)} \;\prec\; \frac{1}{N}
\end{equation}
uniformly for any $E_1$ and $E_2$ satisfying $E_1 + \ii \in \f S$ and $E_2 + \ii \in \f S$. Moreover, we have
\begin{equation} \label{lambda1_plus}
\abs{\lambda_1 - \gamma_+} \;\prec\; N^{-2/3}\,.
\end{equation}
Finally, if $\phi \geq 1 + c$ for some constant $c > 0$, then
\begin{equation} \label{lambdaN_minus}
\abs{\lambda_N - \gamma_-} \;\prec\; N^{-2/3}\,.
\end{equation}
\end{lemma}

Starting from Lemma \ref{lem: rigi-tool}, the proof of Theorem \ref{thm: cov-rig} is elementary. (The details are given e.g.\ on the last page of Section 5 in \cite{EYY3}.)

\begin{proof}[Proof of Lemma \ref{lem: rigi-tool}]
The estimate \eqref{n-n_gamma} is a standard consequence of \eqref{bound: trR}, using Helffer-Sj\"ostrand functional calculus; see e.g.\ \cite[Section 5]{EYY3}.

What remains is the proof of \eqref{lambda1_plus} and \eqref{lambdaN_minus}. Here the argument from \cite[Section 8]{PY} applies with trivial modifications. The key inputs in our case are \eqref{n-n_gamma}, Lemma \ref{lem:stability}, \eqref{ZPrec}, and Lemma \ref{lem: fluct_avg} combined with Remark \ref{rem: fluct_avg}. We omit further details.
\end{proof}

\section{The isotropic law: proof of Theorem \ref{thm: IMP}} \label{sec:4}
In this section we complete the proof of Theorem \ref{thm: IMP}. Since \eqref{bound: trR} was proved in Section \ref{sec: MP law}, we only need to prove \eqref{bound: Gij isotropic} and \eqref{bound: Rij isotropic}. For definiteness, we give the details of the proof of \eqref{bound: Gij isotropic}; the proof of \eqref{bound: Rij isotropic} is very similar, and the required modifications are outlined at the end of Section \ref{sec: exp 9} below.

\subsection{Rescaling} \label{sec:rescaling}
It is convenient to introduce the rescaled quantities
\begin{equation*}
\wt G(z) \;\deq\; \phi^{1/2} \, G(z)\,, \qquad \wt z \;\deq\; \phi^{-1/2} z\,.
\end{equation*}
The reason for this scaling is that for $z \in \f S$ the diagonal entries of $\wt G$ and $\wt z$ are of order one (See \eqref{rescaled_z} as well as \eqref{wt m gamma is bounded} and \eqref{bound: Gij 2} below). Note that all formulas from Lemma \ref{lem: res identity 1} hold after the replacement $(z,G) \mapsto (\wt z, \wt G)$.

We also introduce the rescaled quantity
\begin{equation}\label{def tilde m}
\wt m_\phi \;\deq\; \phi^{1/2} m_{\phi^{-1}} \;=\;  \phi^{-1/2} \pbb{m_\phi + \frac{1 - \phi}{z}}\,.
\end{equation}
 The motivation behind this definition is that 
\begin{equation} \label{wt m gamma is bounded}
\abs{\wt m_\phi} \;\asymp\; 1
\end{equation}
for $z \in \f S$, as can be easily seen using \eqref{bounds on mg}. (Recall that $\phi \geq 1$ by assumption.)
The following result is an immediate corollary of Theorem \ref{thm: cov-loc}. Recall the definition of $\Psi$ from \eqref{def of Psi MP}.
\begin{lemma}
In $\f S$ we have
\begin{align}
\label{bound: Gij 2}
\absb{\wt G_{ij} - \delta_{ij} \wt m_\phi} &\;\prec\; \phi^{-1/2}\Psi\,,
\\
\label{bound: Rij 2}
\abs{R_{\mu \nu} -\delta_{\mu \nu} m_\phi} &\;\prec\;\Psi\,.
\end{align}
\end{lemma}

\begin{lemma} \label{lem: wt G bounds}
Fix $\ell \in \N$. Then we have, uniformly in $\f S$ and for $\abs{T} \leq \ell$ and $i,j \notin T$,
\begin{equation}\label{bound tilde G}
\absb{\wt G_{ij}^{(T)} - \delta_{ij} \wt m_\phi} \;\prec\; \phi^{-1/2} \Psi
\end{equation}
as well as
\begin{equation*}
\absb{\wt G^{(T)}_{ii}} \;\prec\; 1\,, \qquad \absbb{\frac{1}{\wt G^{(T)}_{ii}}} \;\prec\; 1\,.
\end{equation*}
\end{lemma}
\begin{proof}
From \eqref{bound: Gij 2} and \eqref{wt m gamma is bounded} we easily find
\begin{equation*}
\absb{\wt G_{ij} - \delta_{ij} \wt m_\phi} \;\prec\; \phi^{-1/2} \Psi\,, \qquad \absb{\wt G_{ii}} \;\prec\; 1\,, \qquad \absb{1 / \wt G_{ii}} \;\prec\; 1\,.
\end{equation*}
The statement for general $T$ satisfying $\abs{T} \leq \ell$ then follows easily by induction on the size of $T$, using the identity \eqref{Gij Gijk sc} and the fact that $\phi^{-1/2} \Psi \leq 1$.
\end{proof}

\subsection{Reduction to off-diagonal entries}
By linearity and polarization, in order to prove \eqref{bound: Gij isotropic} it suffices to prove that
\begin{equation*}
\absb{\scalar{\f v}{G \f v} - \phi^{-1/2} \wt m_\phi} \;\prec\; \phi^{-1} \Psi
\end{equation*}
for deterministic unit vectors $\f v$. All of our estimates will be trivially uniform in the unit vector $\f v$ and $z \in \f S$, and we shall not mention this uniformity any more. Thus, for the following we fix a deterministic unit vector $\f v \in \C^M$.

We write
\begin{equation*}
\scalar{\f v}{G \f v} - \phi^{-1/2} \wt m_\phi \;=\; \sum_a \abs{v_a}^2 (G_{aa} - \phi^{-1/2} \wt m_\phi) + \cal Z \,,
\end{equation*}
where we defined
\begin{equation} \label{def_calZ}
\cal Z \;\deq\; \sum_{a \neq b} \ol v_a G_{ab} v_b \;=\; \phi^{-1/2} \sum_{a \neq b} \ol v_a \wt G_{ab} v_b\,.
\end{equation}
By \eqref{bound: Gij} we have
\begin{equation*}
\absbb{\sum_a \abs{v_a}^2 (G_{aa} - \phi^{-1/2}\wt m_\phi)} \;\prec\; \phi^{-1} \Psi\,.
\end{equation*}
Hence it suffices to prove that
\begin{equation} \label{claim on X}
\abs{\cal Z} \;\prec\; \phi^{-1} \Psi\,.
\end{equation}
The rest of this section is devoted to the proof of \eqref{claim on X}.

\subsection{Sketch of the proof}\label{sec:sketch}

The basic  reason why \eqref{claim on X} holds is that $G_{ab}$ can be expanded, to leading order,  as a
sum of independent random variables using the identity \eqref{Gij expanded sc}. To simplify the presentation in this sketch, we set $M=N$, so that $\phi=1$ and the rescalings from Section \ref{sec:rescaling} indicated by a tilde are trivial. Hence we drop all tildes. From \eqref{Gij expanded sc}  we get
\begin{equation} \label{sketch 1}
\sum_{a \neq b}   \ol v_a  G_{ab} v_b \;=\;  z \sum_{a\neq b} G_{aa} G_{bb}^{(a)}\sum_{\mu,\nu}  \ol v_a X_{a\mu}  R^{(ab)}_{\mu\nu} X_{\nu b}^* v_b\,.
\end{equation}
If we could replace the diagonal entries by the deterministic value $m_\phi$, it would suffice to estimate the sum $\sum_{a\neq b} \sum_{\mu,\nu}  \ol v_a X_{a\mu}  R^{(ab)}_{\mu\nu} X_{\nu b}^* v_b$. By the independence of the entries of $X$ we have, using \eqref{two-set LDE},
\begin{multline*}
\absbb{\sum_{a\neq b} \sum_{\mu\nu}  \ol v_a X_{a\mu}  R^{(ab)}_{\mu,\nu} X_{\nu b}^* v_b}
\;\prec\; \pbb{ \frac{1}{N^2}\sum_{a\neq b} \sum_{\mu,\nu} \abs{v_a}^2 \abs{v_b}^2 \absb{ R^{(ab)}_{\mu\nu}}^2}^{1/2}
\\
=\;  \pbb{\frac{1}{N^2\eta}\sum_{a\neq b}  \abs{v_a}^2 \abs{v_b}^2   \im \tr R^{(ab)}}^{1/2} \;\prec\; \Psi\,,
\end{multline*}
where we used  the analogue of \eqref{Ward} for $R$, \eqref{bound: Rij}, \eqref{RT+},  and the normalization of $\f v$. 
 Hence, if we could ignore the error arising from the approximation $G_{aa} \approx m_\phi$, the proof of Theorem \ref{thm: IMP} would 
be very simple. 

The error made in the approximation $G_{aa} \approx m_\phi$ is of order $\Psi$ by \eqref{bound: Gij 2}, 
so that the corresponding error term on the right-hand side of \eqref{sketch 1} may be bounded using \eqref{two-set LDE} by 
$$
O_\prec(\Psi) \sum_{a\neq b}  \abs{v_a}\abs{v_b} \absbb{\sum_{\mu\nu}  X_{a\mu}  R^{(ab)}_{\mu\nu} X_{\nu b}^* }
\;\prec\; \Psi^2 \sum_{a\neq b}  \abs{v_a}\abs{v_b} \;\leq\; \Psi^2 \norm{\bv}_1^2\,.
$$
However, the vector $\bv$ is normalized not in $\ell^1$ but in $\ell^2$.  In general, all that can be said about its $\ell^1$-norm is $\norm{\bv}_1\leq M^{1/2} \norm{\bv}_2 = M^{1/2}$. This estimate is sharp if $\f v$ is delocalized, i.e.\ if the entries of $\f v$ have size of order $M^{-1/2}$. The $\ell^1$- and $\ell^2$-norms of $\f v$ are of the same order precisely when only a finite number of entries of $\f v$ are nonzero, in which case Theorem \ref{thm: IMP} is anyway a trivial consequence of Theorem \ref{thm: cov-loc}. 

We conclude that the simple replacement of $G_{aa}$ with its deterministic approximation in \eqref{sketch 1} is not affordable.
Not only the leading term but also every error term has to be expanded in the entries of $X$.  This expansion is most effectively controlled if performed within a high-moment estimate. Thus, for large and even $p$ we shall estimate 
\begin{equation}
  \E \absbb{ \sum_{a\neq b}  \ol v_a  G_{ab} v_b }^p \;=\; \E \sum_{a_1\neq b_1}\ldots \sum_{a_p\neq b_p} \prod_{i=1}^p \ol v_{a_i}  G_{a_ib_i} v_{b_i}\,.
\label{zp}
\end{equation}
 (To simplify notation we drop the unimportant complex conjugations on $p/2$ factors.) 
We shall show that the expectation forces many indices  of the leading-order terms  to coincide, at least in pairs, so that eventually every $v_a$ appears at least to the second power, 
which consistently leads to estimates in terms of the $\ell^2$-norm of $\f v$. Any index that remains single gives rise to a small factor $M^{-1/2}$ which counteracts the large factor $\norm{\bv}_1\leq M^{1/2}$. The trivial bound (arising from estimating each entry $\abs{v_a}$ by 1 and the summation over $a$ and $b$ by $M^2$) is affordable only at a very high order, when
  the number of factors $\Psi \leq N^{-\omega/2}$ that have been generated is sufficient to compensate the loss from the trivial bound. 
 This idea will be used to stop the expansion after  a sufficiently large, but finite,  number of steps.

Before explaining the general strategy, we sketch a second moment calculation.
First, we write
\begin{equation}\label{Zsq}
  \E \absbb{ \sum_{a \neq b} \ol v_a  G_{ab} v_b }^2 \;=\; \E \sum_{a\neq b}\sum_{c\neq d}  \ol v_a  G_{ab} v_b \; \ol v_c  G_{cd}^* v_d\,.
\end{equation}
Using \eqref{Gij Gijk sc}, we \emph{maximally expand} all resolvent entries in the indices $a, b, c,d$. This means that
we use \eqref{Gij Gijk sc} repeatedly until each term in the expansion is independent of all Latin indices that do not explicitly appear among its lower indices;  here an entry is independent of an index if that index is an upper index of the entry. This generates a series of \emph{maximally expanded terms}, whereby a resolvent entry is by definition maximally expanded if we cannot add to it upper indices from the set $a,b,c,d$ by using the identity \eqref{Gij Gijk sc}. In other words, $G_{ij}^{(T)}$ is maximally expanded
if and only if $T = \{a, b,c, d\} \setminus \{i,j\}$. 

To illustrate this procedure,  we assume temporarily that $a,b,c,d$ are all distinct,  and write, using \eqref{Gij Gijk sc}, 
\be\label{Gab}
G_{ab} \;=\; G_{ab}^{(c)} + \frac{G_{ac}G_{cb}}{G_{cc}} \;=\;  G_{ab}^{(cd)} + \frac{G_{ac}G_{cb}}{G_{cc}} + \frac{G_{ad}^{ (c) }
 G_{db}^{ (c) }}{G_{dd}^{(c) }}\,.
\ee
Here the first term is maximally expanded, but the second and third are not; we therefore continue
to expand them  in a similar fashion by applying \eqref{Gij Gijk sc} to each resolvent entry. In general, this procedure does not terminate, but it does generate
finitely many maximally expanded terms with no more than a fixed number, say $\ell$,
of off-diagonal resolvent entries, in addition to finitely many terms that are not
maximally expanded but contain more than $\ell$ off-diagonal entries. 
By choosing $\ell$ large enough, these latter terms may be estimated trivially. 
We therefore focus on the maximally expanded terms, and we write
$$
  G_{ab} \;=\;  G_{ab}^{(cd)} + \frac{G_{ac}^{(bd)}G_{cb}^{(ad)}}{G_{cc}^{(abd)}}+
  \frac{G_{ad}^{(bc)}G_{db}^{(ac)}}{G_{dd}^{(abc)}}  + \ldots\,.
$$
 We get a similar expression for $G_{cd}^*$. We plug both of these expansions into \eqref{Zsq} and multiply out the product. 
The leading term is
$$ 
   \E \sum_{a\neq b} \sum_{c\neq d}  \ol v_a  G_{ab}^{(cd)} v_b \, \ol v_c  G_{cd}^{*(ab)} v_d\,.
$$
 We now expand both resolvent entries using \eqref{Gij expanded sc}, which gives 
\begin{multline*}
\E \sum_{a\neq b}\sum_{c\neq d}  \ol v_a  G_{ab}^{(cd)} v_b \, \ol v_c  G_{cd}^{*(ab)} v_d
\\
=\;  \sum_{a\neq b}\sum_{c\neq d}\sum_{\mu,\nu,\alpha,\beta} \ol v_a v_b \ol v_c v_d \, \E \pB{ G_{aa}^{(cd)} G_{bb}^{(acd)}  X_{a\mu}
 R_{\mu\nu}^{(abcd)} X^*_{\nu b}
   G_{cc}^{*(ab)}G_{dd}^{*(abc)} X_{c\al} R_{\al\beta}^{*(abcd)} X_{\beta d}^*}
\end{multline*}
The goal is to use the expectation to get a pairing (or a more general partition) of the entries of $X$. In order to do that, we shall require all terms that are not entries of $X$ to be independent of the randomness in the rows $a,b,c,d$ of $X$. While the entries of $R$ satisfy this condition, the entries of $G$ do not. We shall hence have to perform a further expansion on them using the identities \eqref{Gij Gijk sc} and \eqref{Gij expanded sc}.
 In fact, these two types of expansions will have to be performed in tandem, using a two-step recursive procedure.
The main reason behind this is that even if all entries of $G$ are maximally expanded, each application of \eqref{Gij expanded sc}
produces a diagonal entry that is not maximally expanded; for such terms the expansion using \eqref{Gij Gijk sc}
has to be repeated.
 For the purposes of this sketch, however, we omit the details of the further expansion of the entries of $G$, and replace them with their deterministic leading order, $m_\phi$ (see \eqref{bound: Gij 2}).
This approximation gives
\begin{equation}\label{square}
\E \sum_{a\neq b}\sum_{c\neq d}  \ol v_a  G_{ab}^{(cd)} v_b \, \ol v_c  G_{cd}^{*(ab)} v_d \;\approx\; \abs{m_\phi}^4 
 \sum_{a\neq b, c\neq d} \ol v_a v_b \ol v_c v_d\sum_{\mu,\nu,\alpha,\beta}\E \pB{ X_{a\mu} R_{\mu\nu}^{(abcd)} X^*_{\nu b}
  X_{c\al} R_{\al\beta}^{*(abcd)} X_{\beta d}^* }\,.
\end{equation}
Since all entries of $R$ are independent of all entries of $X$, we can compute the expectation with respect to the rows $a,b,c,d$.
Note that the only possible pairing is $a=d$, $\mu=\beta$, $b=c$, and $\nu=\al$. This results in the expression
$$
\frac{\abs{m_\phi}^4}{N^2} 
 \sum_{a\neq b} | v_a|^2 \abs{v_b}^2 \E \sum_{\mu\nu} \absb{R_{\mu\nu}^{(abcd)}}^2 \;\approx\; \frac{1}{N\eta}  \sum_{a\neq b} \abs{v_a}^2 \abs{v_b}^2
 \frac{1}{N} \E  \tr  \im R^{(abcd)} \;\approx\; \frac{\im m_{\phi }}{N\eta} \;\sim\; \Psi^2\,,
$$
where we used that $\bv$ is $\ell^2$-normalized.

This calculation, while giving the right order, was in fact an oversimplification, since  the expansion \eqref{Gab}
required that $b\neq c$ and $d\neq a$.  The correct argument requires first a decomposition of the summation over $a,b,c,d$ into a finite number of terms, indexed by the partitions of four elements, according to the coincidences among these four indices. 
Thus, we write
\be\label{part}
 \sum_{a\neq b, c\neq d} \;=\; \sum_{a,b,c,d}^* + \sum_{a=c, b, d}^* + \sum_{a=d, b,c}^* + \sum_{b=c, a, d}^* + \sum_{b=d,a, c}^*
 + \sum_{a=c, b=d}^* + \sum_{a=d, b=c}^* \,,
\ee
where a star over the summation indicates that all summation indices 
 that are not explicitly equal to each other have to be  distinct. The above calculation
leading to \eqref{square} is valid for the first summation of \eqref{part}, whose contribution (up to leading order) is
zero, since the only possible pairing contradicts the condition that the indices $a,b,c,d$ be all distinct. It is not too hard to see
that, among the sums in \eqref{part}, only the last one gives a nonzero contribution (up to leading order), and it is,
going back to \eqref{Zsq}, equal to
$$
 \E \sum_{a\neq b} \abs{v_a}^2 \abs{v_b}^2 \abs{G_{ab}}^2 \;\prec\; \Psi^2\,;
$$
here we used the bound \eqref{bound: Gij 2}. Notice that taking the expectation forced us to chose 
the pairing $a=d$, $b=c$ to get a non-zero term.  This example provides a glimpse into  the mechanism that guarantees
that the $\ell^2$-norm of $\bv$ appears.

Next, we consider a subleading term from the first summation in \eqref{part}, which has three off-diagonal entries:
$$ \sum_{a,b,c,d}^* \ol v_a v_b \ol v_c v_d
\frac{G_{ac}^{(bd)}G_{cb}^{(ad)}}{G_{cc}^{(abd)}} G_{cd}^{*(ab)}\,.
$$
We proceed as before, expanding all off-diagonal entries of $G$ using \eqref{Gij expanded sc}. Up to leading order, we get 
$$
\sum_{a,b,c,d}^*
 \ol v_a v_b \ol v_c v_d\sum_{\mu,\nu,\alpha,\beta,\gamma,\delta}\E \Big[  \big(X_{a\mu} R_{\mu\nu}^{(abcd)} X^*_{\nu c}\big)\big( X_{c\al} R_{\al\beta}^{(abcd)} X^*_{\beta b}
\big)\big(  X_{c\gamma} R_{\gamma\delta}^{*(abcd)} X_{\delta d}^*\big) \Big] 
$$
The expectation again renders this term zero if $a, b,c,d$ are distinct.

Based on these preliminary heuristics, we outline the main steps in estimating a high moment of $\cal Z$.

\begin{description}

\item[Step 1.] Partition the indices in \eqref{zp} according to their coincidence structure: indices in the same block of the partition are required to coincide and indices in different blocks are required to be distinct. This leads to a reduced family, $T$, of distinct indices.

\item[Step 2.]
Make all entries of $G$ maximally expanded by repeatedly applying the identity \eqref{Gij Gijk sc}. Roughly, this entails adding upper indices from the family $T$ to each entry of $G$ using the identity \eqref{Gij Gijk sc}. We stop the iteration if either \eqref{Gij Gijk sc} cannot be applied to any entry of $G$ or we have generated a sufficiently large number of off-diagonal entries of $G$.

\item[Step 3.] Apply \eqref{Gij expanded sc} to each maximally expanded off-diagonal entry of $G$. This yields factors of the form $\sum_{\mu,\nu} X_{a\mu} R^{(T)}_{\mu\nu} X_{\nu b}^*$ with $a,b\in T$ and $R^{(T)}$ is independent of all entries of $X$ by construction. In addition, this application of \eqref{Gij expanded sc} produces new diagonal entries of $G$ that are not maximally expanded.

\item[Step 4.]
Repeat Steps 2 and 3 recursively in tandem until we only have a sum of terms whose factors consist of maximally expanded diagonal entries of $G$, entries of $R^{(T)}$, and entries of $X$ from the rows indexed by $T$.

\item[Step 5.]
Apply \eqref{Gii expanded sc} to each maximally expanded diagonal entry of $G$. We end up with factors consisting only of entries of $R^{(T)}$ and entries of $X$ from the rows indexed by $T$.

\item[Step 6.] Using the fact that all entries of $R$ are independent of all entries of $X$, take a partial expectation over the rows of $X$ indexed by the set $T$; this only involves the entries of $X$. Only those terms give a nonzero contribution whose Greek indices have substantial coincidences. 

\item[Step 7.] For entropy reasons, the leading-order term arises from the smallest number of 
 constraints among the summation vertices that still results in a nonzero contribution. This corresponds to a pairing both among the Greek and the Latin indices. This naturally leads to estimates in terms to the $\ell^2$-norm of $\f v$.

\item[Step 8.] Observe that if a Latin index $i$ remained single in the partitioning of Step 1 (so that the corresponding weight factor will involve the $\ell^1$-norm $\sum_i \abs{v_i}$) then, by a simple parity argument, the number of appearances of the index $i$ will remain odd along the expansion of Steps 2--5. 
This forces us to take at least a third (but in fact at least a fifth) moment of some entry $X_{i\mu}$, which reduces the combinatorics of the summations 
compared with the fully paired situation from Step 7. This combinatorial gain offsets the factor $M^{1/2}$ lost in taking the $\ell^1$-norm of $\f v$.

\end{description}

Steps 1--6 require a careful expansion algorithm and a meticulous bookkeeping
of the resulting terms. We shall develop a graphical language  that encodes
the resulting monomials. 
Expansion steps will be recorded via operations
on graphs such as merging certain vertices or replacing 
some vertex or edge by a small subgraph.
Several ingredients of the 
graphical representation and the concept
of graph operations are inspired by tools from \cite{EKY2} developed for random band matrices.
Once the appropriate graphical
language is in place and  the expansion algorithm has been constructed,  the observations in Steps 7 and 8 will yield the desired
estimate by a power counting  coupled with a parity argument.

\subsection{The $p$-th moment of $\cal Z$ and introduction of graphs} \label{sec: graphs}

We shall estimate $\cal Z$ with high probability by estimating its $p$-th moment for a large but fixed $p$. It is convenient to rename the summation variables in the definition of $\cal Z$ as $(a,b) = (b_1,b_2)$. Let $p$ be an even integer and write
\begin{equation} \label{E Xp 1}
\E \abs{\cal Z}^{p} \;=\; \phi^{-p/2} \, \E \sum_{b_{11} \neq b_{12}} \cdots \sum_{b_{p1} \neq b_{p2}} \pBB{\prod_{k = 1}^{p/2} \ol v_{b_{k1}} \wt G_{b_{k1} b_{k2}} v_{b_{k2}}} \pBB{\prod_{k = p/2 + 1}^{p} \ol v_{b_{k1}} \wt G_{b_{k1} b_{k2}}^* v_{b_{k2}}}\,,
\end{equation}
where we recall the definition of $\cal Z$ from \eqref{def_calZ}.

We begin by partitioning the summation according to the coincidences among the indices $\f b = (b_{kr} \col 1 \leq k \leq p\,,\, r = 1,2)$. Denote by $\cal P(\f b)$ the partition of $\{1, \dots, p\} \times \{1,2\}$ defined by the equivalence relation $(k,r) \sim (l,s)$ if and only if $b_{kr} = b_{ls}$. We define $\fra P_p$ as the set of partitions $P$ of $\{1, \dots, p\} \times \{1,2\}$ such that, for all $k = 1, \dots, p$, the elements $(k,1)$ and $(k,2)$ are not in the same block of $P$. Hence we may rewrite \eqref{E Xp 1} as 
\begin{equation} \label{E Xp 2}
\E \abs{\cal Z}^{p} \;=\; \phi^{-p/2} \, \E \sum_{P \in \fra P_p} \sum_{\f b} \ind{\cal P(\f b) = P} \pBB{\prod_{k = 1}^{p/2} \ol v_{b_{k1}} \wt G_{b_{k1} b_{k2}} v_{b_{k2}}} \pBB{\prod_{k = p/2 + 1}^{p} \ol v_{b_{k1}} \wt G_{b_{k1} b_{k2}}^* v_{b_{k2}}}\,.
\end{equation}
We shall perform the summation by first fixing the partition $P \in \fra P_p$ and by deriving an upper bound that is uniform in $P$; at the very end we shall sum trivially over $P \in \fra P_p$.

In order to handle expressions of the form \eqref{E Xp 2}, as well as more complicated ones required in later stages of the proof, we shall need to develop a graphical notation. The basic idea is to associate matrix indices with vertices and resolvent entries with edges. The following definition introduces graphs suitable for our purposes.

\begin{definition}[Graphs] \label{def: graphs}
By a \emph{graph} we mean a finite, directed, edge-coloured, multigraph
\begin{equation*}
\Gamma \;=\; \pb{V(\Gamma), E(\Gamma), \xi(\Gamma)} \;\equiv\; (V,E,\xi)\,.
\end{equation*}
Here $V$ is a finite set of vertices, $E$ a finite set of directed edges, and $\xi$ is a ``colouring of $E$'', i.e.\ a mapping from $E$ to some finite set of colours. The graph $\Gamma$ may have multiple edges and loops. More precisely, $E$ is some finite set with maps $\alpha, \beta \col E \to V$; here $\alpha(e)$ and $\beta(e)$ represent the source and target vertices of the edge $e \in E$. We denote by $\deg_\Gamma(i)$ the degree of the vertex $i \in V(\Gamma)$.
\end{definition}

We may now express the right-hand side of \eqref{E Xp 2} using graphs. Fix the partition $P \in \fra P_p$. We associate a graph $\Delta \equiv \Delta(P)$ with $P$ as follows. The vertex set $V(\Delta)$ is given by the blocks of $P$, i.e.\ $V(\Delta) = P$. The set of colours, i.e.\ the range of $\xi$, is $\{G, G^*\}$ (we emphasize that these two colours are simply formal symbols whose name is supposed to evoke their meaning). The set of edges $E(\Delta)$ is parametrized as follows by the resolvent entries on the right-hand side of \eqref{E Xp 2}. Each resolvent entry $G_{b_{k1} b_{k2}}^\#$ gives rise to an edge $e \in E(\Delta)$ with colour $\xi(e) = G$ if $\#$ is nothing and $\xi(e) = G^*$ if $\#$ is $*$. The source vertex $\alpha(e)$ of this edge is the unique block of $P$ satisfying $(k,1) \in \alpha(e)$, and its target vertex $\beta(e)$ the unique block of $P$ satisfying $(k,2) \in \beta(e)$.
Figure \ref{fig: construction of Delta} illustrates the construction of $\Delta(P)$, where the two different types of line correspond to the two colours $G, G^*$. The graph $\Delta$ has no loops.
\begin{figure}[ht!]
\begin{center}
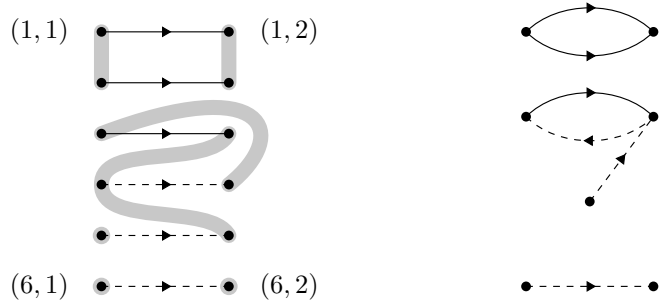
\end{center}
\caption{The construction of $\Delta$. Here we took $p = 6$. On the left we give a graphical representation of the right-hand side of \eqref{E Xp 2}; the vertices $(k,r)$ are labelled by $k = 1, \dots, 6$ and $r = 1,2$; each edge is associated with an entry of $\wt G$ (drawn with a solid line) or $\wt G^*$ (drawn with a dashed line); the partition $P$ is drawn using grey regions representing the blocks of $P$. 
On the right we draw $\Delta(P)$, with $\abs{V(\Delta)} = \abs{P} = 7$ vertices. The partition depicted here corresponds
to the index coincidences $b_{11}=b_{21}$, $b_{12}=b_{22}$, $b_{31}=b_{42}$, $b_{32}= b_{41}= b_{52}$; apart from these constraints, all indices are distinct in the summation over $\f b$ in \eqref{E Xp 2}.  The index associated with block $i \in V(\Delta)$ is denoted by $a_i$, so that $a_i = b_{kr}$ for any $(k,r)$ in the block $i$ of the partition $P$.  \label{fig: construction of Delta}}
\end{figure}

Using the graph $\Delta \equiv \Delta(P)$ we may rewrite the right-hand side of \eqref{E Xp 2}. Each vertex $i \in V(\Delta)$, associated with a block of $P$, is assigned a summation index $a_i \in \{1, 2, \ldots, M\}$,
 and we write $\f a = (a_i)_{i \in V(\Delta)}$. The indicator function on the right-hand side of \eqref{E Xp 2} translates to the condition that $a_i \neq a_j$ for $i \neq j$ (where $i,j \in V(\Delta)$). We use the notation $\sum^{*}$ to denote summation subject to this condition (distinct summation indices). Thus we may rewrite \eqref{E Xp 2} as
\begin{equation} \label{E Zp}
\E \abs{\cal Z}^p \;=\; \phi^{-p/2} \sum_{P \in \fra P_p} Y(\Delta(P))\,,
\end{equation}
where we defined
\begin{equation} \label{def Y Delta}
Y(\Delta) \;\deq\; \sum_{\f a}^{*} w_{\f a}(\Delta) \, \E \cal A_{\f a}(\Delta)
\end{equation}
using the abbreviations
\begin{equation} \label{def w and cal A}
w_{\f a}(\Delta) \;\deq\; \prod_{e \in E(\Delta)} \ol v_{a_{\alpha(e)}} v_{a_{\beta(e)}}\,, \qquad
\cal A_{\f a}(\Delta) \;\deq\; \prod_{e \col \xi(e) = G} \wt G_{a_{\alpha(e)} a_{\beta(e)}} \prod_{e \col \xi(e) = G^*} \wt G_{a_{\alpha(e)} a_{\beta(e)}}^*\,.
\end{equation}
The function $w_{\f a}(\Delta)$ has the interpretation of a deterministic (complex) weight for the summation over $\f a$; it satisfies the basic estimate
\begin{equation} \label{estimate for w}
\abs{w_{\f a}(\Delta)} \;\leq\; \prod_{i \in V(\Delta)} \abs{v_{a_i}}^{\deg_\Delta(i)}\,.
\end{equation}
We record the following basic properties of $\Delta$.
\begin{itemize}
\item
$\abs{E(\Delta)} = p$.
\item
$\Delta$ has no loops, i.e.\ $\alpha(e) \neq \beta(e)$ for all $e \in E(\Delta)$.
\item
$1 \leq \abs{V(\Delta)} \leq 2 p$.
\end{itemize}

Our first goal is to  use the expansion formulas \eqref{Gij Gijk sc}--\eqref{Gij expanded sc}
to express $\cal A_{\f a}(\Delta)$ as a sum of monomials involving only entries of $X$ and $R$,  so that no entries of $G$ remain.
The entries of $R$ and $X$ will be independent by construction, which will make the evaluation of the expectation possible. The result will be given in Proposition \ref{prop:Y_tree} below, which expresses $Y(\Delta)$ as a sum of terms associated with graphs, which are themselves conveniently indexed using a finite binary tree, denoted by $\cal T$. 
To bookkeep this expansion, we shall need a more general class of graphs than $\Delta$.

\subsection{Generalized colours and encoding} \label{sec: exp 1}

For the following we fix $p \in 2 \N$ and a partition $P \in \fra P_p$, and set $\Delta = \Delta(P)$. We shall develop an
 expansion scheme for monomials of type $\cal A_{\f a}(\Delta)$. A fundamental notion in our expansion is that of \emph{maximally expanded entries of $\wt G$}, given in Definition \ref{def: max expanded} below.

We shall need to enlarge the set of colours of edges, so as to be able to encode entries of not only $\wt G$ and $\wt G^*$, but also entries of $R$, $R^*$, $X$, and $X^*$; in addition, we shall need to encode diagonal entries of $\wt G$ and $\wt G^*$ that are in the denominator, as in the formulas \eqref{Gij Gijk sc}, as well as to keep track of upper indices. We need all of these factors, since our expansion relies on a repeated application of the identities \eqref{Gij Gijk sc}, \eqref{Gii expanded sc}, and \eqref{Gij expanded sc}.

In order to define the graphs precisely, we consider graphs $\Gamma$ satisfying Definition \ref{def: graphs} whose vertex set $V(\Gamma) = V_b(\Gamma) \cup V_w(\Gamma)$ is partitioned into \emph{black vertices} $V_b(\Gamma)$ and \emph{white vertices} $V_w(\Gamma)$. Informally, black vertices are incident to edges encoding entries of $\wt G$ and $\wt G^*$, and white vertices to edges encoding entries of $\wt R$ and $\wt R^*$.  In other words, \emph{black vertices} are associated with \emph{Latin} summation indices in the \emph{population space} $\{1, \dots, M\}$; \emph{white vertices} are associated with \emph{Greek} summation indices in the \emph{sample space} $\{1, \dots, N\}$. See Remark \ref{rem:index sets}. 
We shall only consider graphs $\Gamma$ satisfying
\begin{equation}\label{VB}
V_b(\Gamma) \;=\; V(\Delta)\,,
\end{equation}
an assumption we make throughout the following. This means that only new Greek summation indices
but no new Latin indices are generated, corresponding to the repeated applications of \eqref{Gii expanded sc} and \eqref{Gij expanded sc}.
In particular, the vertex colouring for our graph is very simple: the
vertices of $\Delta$ are black and all other vertices are white.

As our set of colours we choose
\begin{equation} \label{set of colours}
\hb{\xi = (\xi_1, \xi_2, \xi_3) \col \xi_1 \in  \{G, G^*, R, R^*, X, X^*\}  \,,\, \xi_2 \in  \{+,-\}  \,,\, \xi_3 \subset V_b(\Gamma)}\,.
\end{equation}
Note that these colours are to be interpreted merely as list of formal symbols; the choice of their names is supposed to evoke their meaning. 
The component $\xi_1$ determines whether the edge encodes an entry of $\wt G$ (corresponding to $\xi_1 = G$), of $\wt G^*$ (corresponding to $\xi_1 = G^*$), of $R$ (corresponding to $\xi_1 = R$), of $R^*$ (corresponding to $\xi_1 = R^*$), of $X$ (corresponding to $\xi_1 = X$), or of $X^*$ (corresponding to $\xi_1 = X^*$).  The component $\xi_2$ determines whether the entry is in the numerator (corresponding to $\xi_2 = +$) or in the denominator (corresponding to $\xi_2 = -$).
Finally, the component $\xi_3$ is used to keep track of the upper indices of entries of $\wt G$ and $\wt G^*$, which we shall set to be $\f a_{\xi_3} \deq \{a_i \col i \in \xi_3\}$. The entries of $R$ and $R^*$ also have upper indices, but they always carry the
maximal set $\f a_b$ of upper indices, i.e.\ they always appear in the form $R^{(\f a_b)}$ and $R^{*(\f a_b)}$. Hence, upper indices need not
be tracked for the entries of $R$ and $R^*$, and for them we set $\xi_3(e) =  \emptyset $.
Let $\Gamma$ be a graph with colour set \eqref{set of colours}.

We now list some properties of all graphs we shall consider. To that end, we call $e \in E(\Gamma)$ a \emph{$G$-edge} if $\xi_1(e) \in \{G,G^*\}$, an \emph{$R$-edge} if  $\xi_1(e) \in \{R,R^*\}$, and an \emph{$X$-edge} if $\xi_1(e) \in \{X, X^*\}$.
\begin{enumerate}
\item
If $e$ is a $G$-edge then $\alpha(e),\beta(e) \in V_b(\Gamma)$.
\item
If $e$ is an $R$-edge then $\alpha(e),\beta(e) \in V_w(\Gamma)$.
\item
If $\xi_1(e) = X$ then $\alpha(e) \in V_b(\Gamma)$ and $\beta(e) \in V_w(\Gamma)$.
\item
If $\xi_1(e) = X^*$ then $\alpha(e) \in V_w(\Gamma)$ and $\beta(e) \in V_b(\Gamma)$.
\item
If $\xi_2(e) = -$ then $\xi_1(e) \in \{G,G^*\}$ and $\alpha(e) = \beta(e)$.
\item
If $\xi_3(e) \neq \emptyset$ then $\xi_1(e) \in \{G,G^*\}$ and $\xi_3(e) \subset V_b(\Gamma) \setminus \{\alpha(e), \beta(e)\}$.
\end{enumerate}
Properties (i)--(iv) are straightforward compatibility conditions which are obvious in light of the type of matrix entry that the edge $e$ encodes. Property (v) states that only diagonal entries of $\wt G$ and $\wt G^*$ may be in the denominator. Property (vi) states that only entries of $G$ or $\wt G$ may have a (nontrivial) upper index and
the lower indices of an entry of $\wt G$ or $\wt G^*$ may not coincide with its upper indices (by definition of minors).

In order to give a precise definition of the monomial encoded by a coloured edge, and hence of a graph $\Gamma$, it is convenient to split the vertex indices as $\f a = (a_i)_{i \in V(\Gamma)} = (\f a_b, \f a_w)$, where
\begin{equation*}
\f a_b \;\deq\; (a_i)_{i \in V_b(\Gamma)} \;\in\; \{1, \dots, M\}^{\abs{V_b(\Gamma)}}\,, \qquad 
\f a_w \;\deq\; (a_i)_{i \in V_w(\Gamma)} \;\in\; \{1, \dots, N\}^{\abs{V_w(\Gamma)}}\,.
\end{equation*}
Under the former convention, indices assigned to black vertices (elements of $\{1, \dots, N\}$) were Latin letters, while indices
assigned to white vertices (elements of $\{1, \dots, M\}$) were Greek letters. In the above expression,
 all indices assigned to vertices of $V(\Gamma)$  (the indices of $\f a$) are also denoted by Latin letters $i$.
This notation is independent of the previous convention: we simply do not have a third alphabet available.
We always assume that the indices $\f a_b$ are distinct; we impose no constraints on the indices $\f a_w$. 
For the following definitions we fix a collection of vertex indices $\f a$. At the end of the proof, we shall sum over $\f a_b$ under the constraint that the indices of $\f a_b$ be distinct.

For $e \in E(\Gamma)$ with $\xi_1(e) \in \{G,G^*\}$ we define the \emph{resolvent entry encoded by $e$ in $\Gamma$} as
\begin{equation} \label{definition of A e}
\cal A_{\f a}(e, \Gamma) \;\deq\;
\begin{cases}
\wt G_{a_{\alpha(e)} a_{\beta(e)}}^{(\f a_{\xi_3(e)})} & \text{if } \xi_1(e) = G \text{ and } \xi_2(e) = +
\\
\wt G_{a_{\alpha(e)} a_{\beta(e)}}^{^* (\f a_{\xi_3(e)})} & \text{if } \xi_1(e) = G^* \text{ and } \xi_2(e) = +
\\
1 / \wt G_{a_{\alpha(e)} a_{\beta(e)}}^{(\f a_{\xi_3(e)})} & \text{if } \xi_1(e) = G \text{ and } \xi_2(e) = -
\\
1 / \wt G_{a_{\alpha(e)} a_{\beta(e)}}^{^* (\f a_{\xi_3(e)})} & \text{if } \xi_1(e) = G^* \text{ and } \xi_2(e) = -\,.
\end{cases}
\end{equation} 
When drawing graphs, we represent a black vertex as a black dot and a white vertex as a white dot. An edge with $\xi_1 = G$ is represented as a solid directed line joining two black dots, and an edge with $\xi_1 = G^*$ as a dashed directed line joining two black dots. If $\xi_2 = -$ we indicate this by decorating the edge with a white diamond (not to be confused with a white dot).
Notice that such edges are always loops, according to property (v).
Sometimes we also indicate the component $\xi_3(e)$ in our graphs, simply by writing it next to the edge $e$.
See Figure \ref{fig: diagonal} for our graphical conventions when depicting edges with $\xi_1 \in \{G,G^*\}$.
\begin{figure}[ht!]
\begin{center}
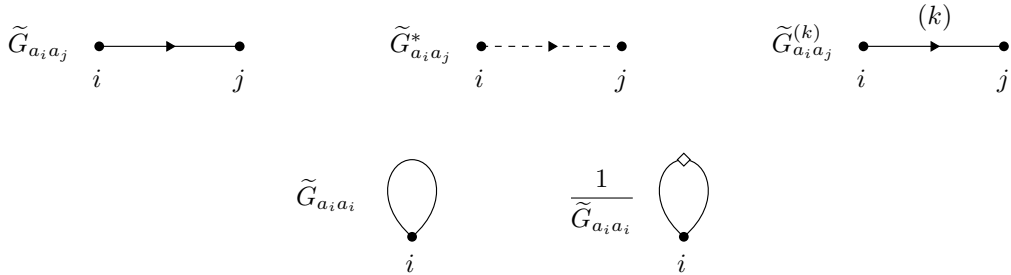
\end{center}
\caption{The graphical conventions for off-diagonal and diagonal edges. Here we draw the case $\xi_1 = G$ (solid lines encoding entries of $\wt G$); if $\xi_1 = G^*$ (encoding entries of $\wt G^*$) we use dashed lines, but the pictures are otherwise identical. The case $\xi_2 = +$ (encoding resolvent entries in the numerator) is drawn without any decorations; if $\xi_2 = -$ (encoding resolvent entries in the denominator) we indicate this with with a white diamond attached to the edge. Note that, since $\xi_2 = -$ only for diagonal entries (encoded by loops), the orientation of the edge is immaterial and the arrow therefore superfluous. \label{fig: diagonal}}
\end{figure}

For the other edges, $e \in E(\Gamma)$ with $\xi_1(e) \in \{R,R^*,X,X^*\}$, we set
\begin{equation*}
\cal A_{\f a}(e, \Gamma) \;\deq\; 
\begin{cases}
R^{(\f a_b)}_{a_{\alpha(e)} a_{\beta(e)}} & \text{if } \xi_1(e) = R
\\
R^{*(\f a_b)}_{a_{\alpha(e)} a_{\beta(e)}} & \text{if } \xi_1(e) = R^*
\\
X_{a_{\alpha(e)} a_{\beta(e)}} & \text{if } \xi_1(e) = X
\\
X_{a_{\alpha(e)} a_{\beta(e)}}^* & \text{if } \xi_1(e) = X^*\,.
\end{cases}
\end{equation*}
When drawing graphs, we represent an edge with $\xi_1 = R$ as a solid directed line joining two white vertices, an edge with $\xi_1 = R^*$ as a dashed directed line joining two white vertices, an edge with $\xi_1 = X$ as a dotted directed line from a black to a white vertex, and an edge with $\xi_1 = X^*$ as a dotted directed line from a white to a black vertex. Note that we use the same line style to draw $X$- and $X^*$-edges, since the orientation of the edge together with the vertex colouring distinguishes them uniquely. 
See Figure \ref{fig: extended definitions} for an illustration of these conventions, and Figure \ref{fig: Gab expansion} for an illustration of \eqref{Gab expansion}.
\begin{figure}[ht!]
\begin{center}
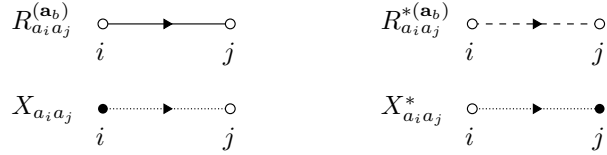
\end{center}
\caption{The graphical conventions for entries of $R^{(\f a_b)}$ (corresponding to $\xi_1 = R$), $R^{*(\f a_b)}$ (corresponding to $\xi_1 = R^*$), $X$ (corresponding to $\xi_1 = X$), and $X^*$ (corresponding to $\xi_1 = X^*$).\label{fig: extended definitions}}
\end{figure}

Having defined $\cal A_{\f a}(e, \Gamma)$ for an arbitrary graph $\Gamma$ with colour set \eqref{set of colours} and $e \in E(\Gamma)$, we define the \emph{monomial encoded by $\Gamma$},
\begin{equation} \label{extended def of cal A}
\cal A_{\f a}(\Gamma) \;\deq\; \prod_{e \in E(\Gamma)} \cal A_{\f a}(e, \Gamma)\,.
\end{equation}
Note that \eqref{extended def of cal A} extends \eqref{def w and cal A}. At this point we introduce a convention that will simplify notation throughout the proof. We allow the monomial $\cal A_{\f a}(\Gamma)$ to be multiplied by a deterministic function of $z$ that is bounded, i.e.\ in general we replace \eqref{extended def of cal A} with
\begin{equation} \label{extended def of cal A with f}
\cal A_{\f a}(\Gamma) \;\deq\; u(\Gamma) \prod_{e \in E(\Gamma)} \cal A_{\f a}(e, \Gamma)\,,
\end{equation}
where $u(\Gamma)$ is some deterministic function of $z$ satisfying $\abs{u(\Gamma,z)} \leq C_\Gamma$ for $z \in \f S$. This will allow us to forget 
signs and various factors of $\wt z$ and $m_\phi$ that are generated along the expansion. The functions $u(\Gamma)$ could be easily tracked throughout the proof, but all that we need to know about them is that they satisfy the conditions 
 listed after \eqref{extended def of cal A with f}.  Not tracking the precise form of these prefactors is sufficient for our purposes, since after 
completing the graphical expansion we shall estimate each graph individually, without making use of further cancellations among different graphs.

\subsection{$R$-groups} \label{sec: exp 1.1}

We define an \emph{$R$-group} to be an induced subgraph of $\Gamma$ consisting of three edges, $e_1$, $e_2$, $e_3$, such that $e_1$ and $e_3$ are $X$-edges and $e_2$ is an $R$-edge, and they form a chain in the sense that $\beta(e_1) = \alpha(e_2)$, $\beta(e_2) = \alpha(e_3)$, and both of these vertices have degree two. We call $e_2$ the \emph{centre} of the $R$-group and define $A(e_2) \deq \alpha(e_1)$ and $B(e_2) \deq \beta(e_3)$. If $A(e_2) = B(e_2)$ we call the $R$-group \emph{diagonal}; otherwise we call it \emph{off-diagonal}. See Figure \ref{fig: R-group} for an illustration.
\begin{figure}[ht!]
\begin{center}
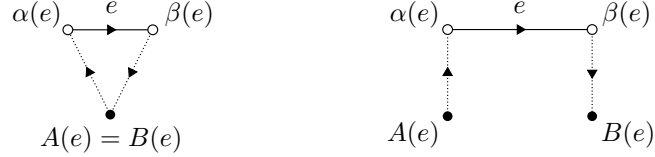
\end{center}
\caption{A diagonal $R$-group (left) and an off-diagonal $R$-group (right). We label the centre of the group by $e$. \label{fig: R-group}}
\end{figure}
We require that our graphs $\Gamma$ satisfy the following property.
\begin{itemize}
\item[(vii)]
Each $X$-edge and $R$-edge of $\Gamma$ belongs to some $R$-group of $\Gamma$. In particular, all white vertices have degree two, and an $R$-group is uniquely determined by its centre.
\end{itemize}
The $R$-groups constitute graphical representations of the monomials on the right-hand sides of \eqref{Gii expanded sc} and \eqref{Gij expanded sc}. 
It is important to stress that there is no restriction on possible coincidences among the white-vertex indices $(a_i)_{i\in V_w}$; this means
that even if two Greek summation indices arising from two different applications \eqref{Gii expanded sc} or \eqref{Gij expanded sc} coincide, they will nevertheless be encoded by distinct white vertices.  This allows us to keep the graphical structure involving $R$ and $X$  edges very simple.

Note that the initial graph $\Delta = \Delta(P)$ trivially satisfies the properties (i)--(vii).

\subsection{Maximally expanded edges and sketch of the expansion} \label{sec:max_exp}

The following definition introduces a notion that underlies our entire expansion. Note that it only applies to $G$-edges.

\begin{definition} \label{def: max expanded}
The $G$-edge $e \in E(\Gamma)$ is \emph{maximally expanded} if $\xi_3(e) = V_b(\Gamma) \setminus \{\alpha(e), \beta(e)\}$. If $e$ is maximally expanded then we also call the entry encoded by it, $\cal A_{\f a}(e, \Gamma)$, \emph{maximally expanded}.
\end{definition}
For instance, if $e$ encodes an entry of the form $\wt G_{a_i a_j}^{(T)}$ then this entry is maximally expanded if and only if $T = \f a_b \setminus \{a_i, a_j\}$. The idea behind this definition is that a maximally expanded entry has as many upper indices from the set $\f a_b$ as possible.

We conclude this section with an outline of the expansion algorithm that will ultimately yield a family of graphs, whose contributions can be explicitly estimated and whose encoded monomials sum up to the monomial encoded by $\Delta = \Delta(P)$ from Section \ref{sec: graphs}. The goal of the expansion is to get rid of all $G$-edges, by replacing them with $R$-groups. Of course, this replacement has to be done in such a manner that the original monomial $\cal A_{\f a}(\Delta)$ can be expressed as a sum of the monomials encoded by the new graphs. Having done the expansion, we shall be able to exploit the fact that the $R$-entries and the $X$-entries are independent.
This independence originates from the upper indices $i$  and $j$ in the entries of $R$ in \eqref{Gii expanded sc} and \eqref{Gij expanded sc}. It allows us to take the expectation in the $X$-variables. Combined with sufficient information about the graphs generated by the expansion, this yields a reduction in the summation that is sufficient to complete the proof.

The expansion relies of three main operations:
\begin{enumerate}
\item[(a)]
make one of the  $G$-entries maximally expanded by adding upper indices using the identity \eqref{Gij Gijk sc};
\item[(b)]
expand  all  off-diagonal maximally expanded $G$-entries of in terms of $X$ using the identity \eqref{Gij expanded sc};
\item[(c)]
expand all  diagonal maximally expanded $G$-entries in terms of $X$ using \eqref{Gii expanded sc}.
\end{enumerate}
We shall  implement  each ingredient by a graph surgery procedure.  Operation (a) is the subject of Section \ref{sec: exp 2}; it creates two new graphs, $\tau_0(\Gamma)$ and $\tau_1(\Gamma)$, from an initial graph $\Gamma$.  Operation (b)  is the subject of Section \ref{sec: exp 3}; it creates one new graph, $\rho(\Gamma)$, from an initial graph $\Gamma$.
As it turns out, Operations (a) and (b) have to
be performed in tandem using a coupled recursion, described by a tree $\cal T$, which is the subject of Section \ref{sec: exp 4}. 
Once this recursion has terminated, Operation (c) may be performed (see Section \ref{sec: exp 5}).

\subsection{Operation (a): construction of  the graphs  $\tau_0(\Gamma)$ and $\tau_1(\Gamma)$} \label{sec: exp 2}
In order to avoid trivial ambiguities, we choose and fix an arbitrary ordering of the vertices $V(\Delta)$ and of the family of resolvent entries $\p{\wt G_{ab}^{(T)} \col a,b \notin T}$. Hence we may speak of the first vertex of $V(\Delta)$ and the first factor of a monomial in the entries $\wt G_{ab}^{(T)}$.

We now describe Operation (a) of the expansion. It relies on the identities
\begin{equation} \label{res identity type 1}
\wt G_{ab}^{(T)} \;=\; \wt G_{ab}^{(T c)} + \frac{\wt G_{ac}^{(T)} \wt G_{cb}^{(T)}}{\wt G_{cc}^{(T)}}\,, \qquad
\frac{1}{\wt G_{aa}^{(T)}} \;=\; \frac{1}{\wt G_{aa}^{(Tc)}} - \frac{\wt G_{ac}^{(T)} \wt G_{ca}^{(T)}}{\wt G_{aa}^{(T)} \wt G_{aa}^{(Tc)} \wt G_{cc}^{(T)}}\,,
\end{equation}
which follow immediately from \eqref{Gij Gijk sc}; here $a,b,c \in \f a_b \setminus T$ and $a,b \neq c$. The same identities hold for $\wt G^*$. The basic idea is to take some graph $\Gamma$ with at least one $G$-entry that is not maximally expanded, to pick the first such $G$-entry, and to apply the first identity of \eqref{res identity type 1} if this entry is in the numerator and the second identity if this entry is in the denominator. By Definition \ref{def: max expanded}, if the $G$-entry is not maximally expanded, there is a $c \in \f a_b$ such that \eqref{res identity type 1} may be applied. The right-hand sides of \eqref{res identity type 1} consist of two terms: the first one has one additional upper index, and the second one at least one additional off-diagonal $G$-entry. 
These two terms can be described by two new graphs, derived from $\Gamma$, denoted by $\tau_0(\Gamma)$ and $\tau_1(\Gamma)$. The graph
$\tau_0(\Gamma)$ is almost identical to $\Gamma$, except that the edge corresponding to the selected entry $\wt G_{ab}^{(T)}$
receives an additional upper index $c$, so that the upper indices of the chosen entry are changed as $T\to (Tc)$. The graph $\tau_1(\Gamma)$
also differs from $\Gamma$ only locally: the single edge of $\wt G_{ab}^{(T)}$ is replaced by two edges and loop
with a diamond.

We now give the precise definition of Operation (a). Take a graph $\Gamma$ that has a $G$-edge that is not maximally expanded. We shall define two new graphs, $\tau_0(\Gamma)$ and $\tau_1(\Gamma)$ as follows. Let $e$ be the first\footnote{Recall that we fixed an arbitrary ordering of the resolvent entries of $G$, which induces an ordering on the edges of $\Gamma$ via the map $e \mapsto \cal A_{\f a}(e, \Gamma)$.} $G$-edge of $\Gamma$ that is not maximally expanded, and let $i$ be the first vertex of $V_b(\Gamma) \setminus \pb{\xi_3(e) \cup \{\alpha(e), \beta(e)\}}$; note that by assumption on $\Gamma$ and $e$ this set of vertices is not empty. We now apply \eqref{res identity type 1} to the entry $\cal A_{\f a}(e, \Gamma)$. We set $a = a_{\alpha(e)}$, $b = a_{\beta(e)}$, $c = a_i$, and $T = \f a_{\xi_3(e)}$ in \eqref{res identity type 1}, and express $\cal A_{\f a}(e, \Gamma)$ as a sum of two terms given by the right-hand sides of \eqref{res identity type 1}; we use the first identity of \eqref{res identity type 1} if $\xi_2(e) = +$ and the second if $\xi_2(e) = -$. This results in a splitting of the whole monomial into a sum of two monomials,
\begin{equation*}
\cal A_{\f a}(\Gamma) \;=\; \cal A_{0, \f a}(\Gamma) + \cal A_{1, \f a}(\Gamma)\,,
\end{equation*}
in self-explanatory notation. By definition, $\tau_0(\Gamma)$ is the graph that encodes $\cal A_{0, \f a}(\Gamma)$ and $\tau_1(\Gamma)$ the graph that encodes $\cal A_{1, \f a}(\Gamma)$. Hence, by definition, we have
\begin{equation} \label{sigma splitting}
\cal A_{\f a}(\Gamma) \;=\; \cal A_{\f a}(\tau_0(\Gamma)) + \cal A_{\f a}(\tau_1(\Gamma))\,.
\end{equation}
Moreover, it follows immediately that the maps $\tau_0$ and $\tau_1$ do not change the vertices, so that we have
\begin{equation} \label{Vb tau}
V_b(\tau_0(\Gamma)) \;=\; V_b(\tau_1(\Gamma)) \;=\; V_b(\Gamma)\,, \qquad
V_w(\tau_0(\Gamma)) \;=\; V_w(\tau_1(\Gamma)) \;=\; V_w(\Gamma)\,.
\end{equation}
The procedure $\Gamma \mapsto (\tau_0(\Gamma), \tau_1(\Gamma))$ may also be explicitly described on the level graphs alone, but we shall neither need nor do this. Instead, we give a graphical depiction of this process in Figure \ref{fig: res identity 1}.
\begin{figure}[ht!]
\begin{center}
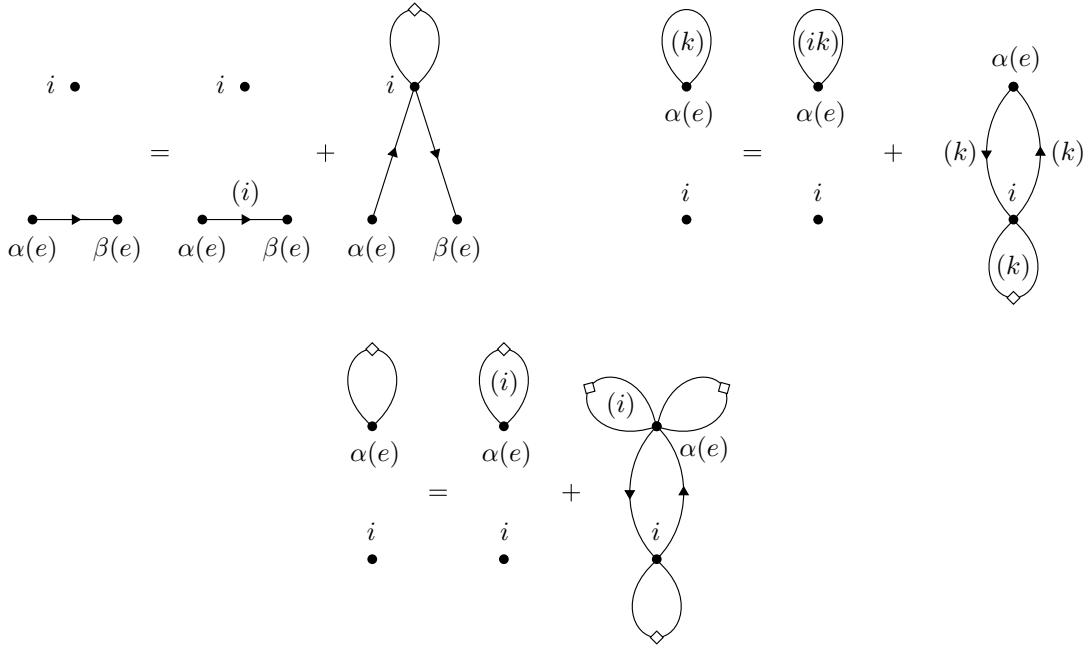
\end{center}
\caption{A graphical depiction of the splitting $\Gamma \mapsto (\tau_0(\Gamma), \tau_1(\Gamma))$ arising from \eqref{res identity type 1}. We only draw the edge $e$ and the vertices $\alpha(e)$, $\beta(e)$, and $i$. All other edges of $\Gamma$ are left unchanged by the operation, and are not drawn. The set $\xi_3$ is indicated in parentheses next to each edge, provided it is not empty. 
 The first graph depicts the operation for the case $\alpha(e) \neq \beta(e)$ (encoding an off-diagonal entry), the second for the case $\alpha(e) = \beta(e)$ and $\xi_2(e) = +$ (encoding a diagonal entry in the numerator), and the third for the case $\alpha(e) = \beta(e)$ and $\xi_2(e) = -$ (encoding a diagonal entry in the denominator). The first  graph on the right-hand side in each identity encodes $\tau_0(\Gamma)$ and the second $\tau_1(\Gamma)$. Recall that the graphs do not track irrelevant signs according to the convention made around
\eqref{extended def of cal A with f}.
\label{fig: res identity 1}}
\end{figure}

The following result is trivial.

\begin{lemma} \label{lem: properties 1}
If $\Gamma$ satisfies the properties (i)--(vii) from Sections \ref{sec: exp 1} and \ref{sec: exp 1.1} then so do $\tau_0(\Gamma)$ and $\tau_1(\Gamma)$.
\end{lemma}

\subsection{Operation (b): construction of the graph $\rho(\Gamma)$} \label{sec: exp 3}

In this section we give the second operation, (b), outlined in Section \ref{sec:max_exp}.
 The idea is that Operation (a) from Section \ref{sec: exp 1} generates off-diagonal $G$-entries that are maximally expanded. They in turn will have to be expanded further using \eqref{Gij expanded sc}, so as to extract their explicit $X$-dependence. Roughly, the map $\rho$ replaces each maximally expanded off-diagonal $G$-edge by an off-diagonal $R$-group.

It will be convenient to have a shorthand for a maximally expanded entry of $\wt G$. To that end, we define, for $a,b \in \f a_b$, the maximally expanded entry
\begin{equation*}
\wh G_{ab} \;\deq\; \wt G_{ab}^{(\f a_b \setminus \{a, b\})}\,.
\end{equation*}
Using \eqref{Gij expanded sc} we may write, for $a \neq b$,
\begin{align}
\wh G_{ab} &\;=\; \wt z \wt G_{aa}^{(\f a_b \setminus \{a,b\})} \, \wh G_{bb} \sum_{\mu,\nu} X_{a \mu} R^{(\f a_b)}_{\mu \nu} X^*_{\nu b}
\notag \\ \label{Gab expansion}
\wh G_{ab}^* &\;=\; \wt z^* \wt G_{aa}^{* (\f a_b \setminus \{a,b\})} \, \wh G_{bb}^* \sum_{\mu,\nu} X_{a \mu} R^{*(\f a_b)}_{\mu \nu} X^*_{\nu b}\,,
\end{align} 
where $\wt z^*$ denotes the complex conjugate of $\wt z$.
Note that the first diagonal term on the right-hand side is not maximally expanded (while the second one is).

The identity \eqref{Gab expansion} may also be formulated in terms of graphs. We denote by $\rho(\Gamma)$ the graph encoding the monomial obtained from $\cal A_{\f a}(\Gamma)$ by applying the identity \eqref{Gab expansion} to each maximally expanded off-diagonal $G$-entry of $\Gamma$.
 This replacement can be done in any order.  By definition of $\rho(\Gamma)$, we have
\begin{equation} \label{identity for rho}
\sum_{\f a_w} \cal A_{\f a_b, \f a_w}(\Gamma) \;=\; \sum_{\f a_w} \cal A_{\f a_b, \f a_w}(\rho(\Gamma))\,.
\end{equation}
Note that both sides depend on $\f a_b$. Each application of \eqref{Gab expansion} adds two white vertices, so that in general $V_w(\rho(\Gamma)) \supset V_w(\Gamma)$. In particular, in \eqref{identity for rho} we slightly abuse notation by using the symbol $\f a_w$ for different families on the left- and right-hand sides. The point is that we always perform an unrestricted summation of the Greek indices 
 associated with the white vertices of the graph.  
 However, the black vertices are left unchanged, so that we have
\begin{equation} \label{Vb rho}
V_b(\rho(\Gamma)) \;=\; V_b(\Gamma)\,.
\end{equation}
Like $\tau_0$ and $\tau_1$, the map $\rho$ may be explicitly defined on the level of graphs, which we shall however not do in order to avoid unnecessary and heavy notation. See Figure \ref{fig: Gab expansion} for an illustration of $\rho$.
\begin{figure}[ht!]
\begin{center}
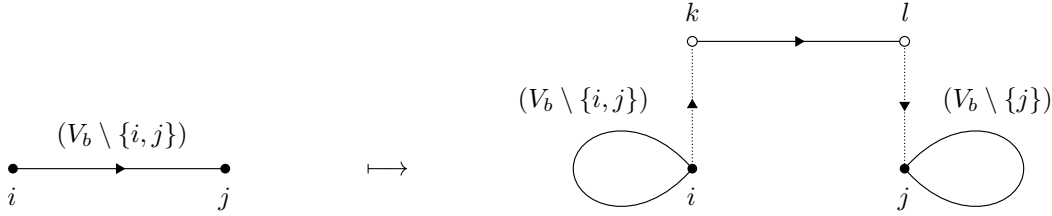
\end{center}
\caption{A graphical depiction of the map $\rho$ resulting from applications of \eqref{Gab expansion}. For simplicity, we draw a graph with a single edge. The indices $a,b,\mu,\nu$ of \eqref{Gab expansion} are associated with the vertices $i,j,k,l$, so that we have $a = a_i$, $b = a_j$, $\mu = a_k$, and $\nu = a_l$. In the picture we abbreviated $V_b = V_b(\Gamma)$. Note that $V_b(\cdot)$ remains unchanged under $\rho$ while $V_w(\cdot)$ is increased by the addition of two new white vertices, $k, l$. The prefactor $\wt z$ is omitted from the graphical representation. \label{fig: Gab expansion}}
\end{figure}

The following result is an immediate corollary of the definition of $\rho$.

\begin{lemma} \label{lem: properties 2}
If $\Gamma$ satisfies the properties (i)--(vii) from Sections \ref{sec: exp 1} and \ref{sec: exp 1.1} then so does $\rho(\Gamma)$.
\end{lemma}

\subsection{Constructing the tree $\cal T$: recursion using (a) and (b)} \label{sec: exp 4}

We now apply Operations (a) and (b) alternately and recursively to the graph $\Delta = \Delta(P)$. 
We start by applying Operation (a) to the graph $\Delta$;  the two new graphs thus produced may have newly created
maximally expanded off-diagonal $G$-entries. We then apply $\rho$ to these edges. Along the procedure 
we get new $R$-groups and additional diagonal entries, some which may not be maximally expanded. We then
repeat the cycle: apply Operation (a) and then Operation (b).
For some graphs the
procedure stops because all $G$-edges have become maximally expanded. For some other graphs, the
algorithm would continue indefinitely, since Operation (b) keeps on producing
diagonal $G$-entries that are not maximally expanded. We shall however show that in such graphs the number of
off-diagonal $G$-edges and $R$-groups increases as the algorithm is run.
Since both of these objects are small, after the accumulation of a sufficiently large number of them
we can stop the recursion and estimate such terms brutally. In summary,
the end result will be a family of graphs encoding monomials in the entries of $R^{(\f a_b)}, R^{* (\f a_b)}, X, X^*$ as well as diagonal entries of $\wh G, \wh G^*$. In addition, by a brutal truncation in this procedure, the algorithm yields terms that do not satisfy this property, but contain a large enough number of off-diagonal $G$-edges and $R$-groups to be negligible.

The algorithm generates a family of graphs $\Theta_\sigma$ which are indexed by finite binary strings $\sigma$, or, equivalently, by vertices of a rooted binary tree $\cal T = (V(\cal T), E(\cal T))$. We start the algorithm with $\Theta_{\emptyset } \deq \Delta$, corresponding to the empty string or the root of the tree. The tree is constructed recursively according to
\begin{equation*}
\Theta_0 \;\deq\; \rho(\tau_0(\Delta))\,, \quad \Theta_1 \;\deq\; \rho(\tau_1(\Delta))\,, \quad
\Theta_{00} \;\deq\; \rho(\tau_0(\Theta_0))\,, \quad \Theta_{10} \;\deq\; \rho(\tau_1(\Theta_0))\,, \quad \Theta_{01}\; \deq\; \rho(\tau_0(\Theta_1))\,,
\end{equation*}
and so on, until a stopping rule is satisfied (see Definition \ref{def:stopping_rule} below). See Figure \ref{fig:tree} for an illustration of the resulting tree.

\begin{figure}[ht!]
\begin{center}
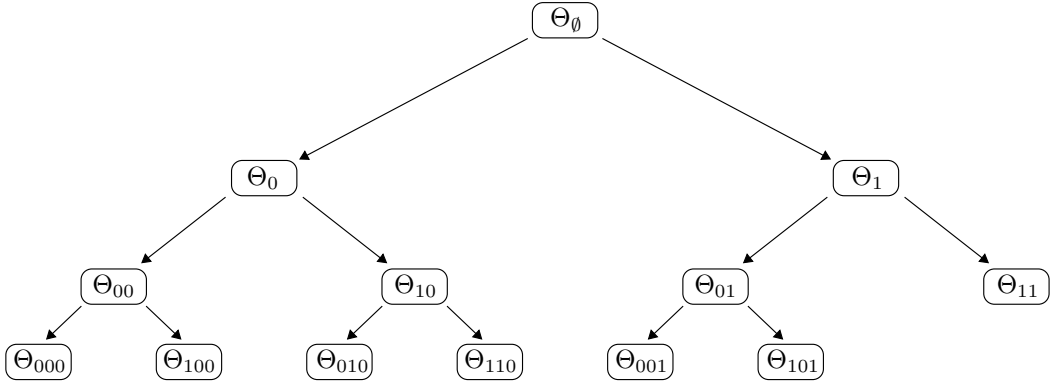
\end{center}
\caption{The tree $\cal T$ whose vertices are binary strings $\sigma$. The root is the empty string $\emptyset$. Each vertex of $\sigma \in V(\cal T)$ encodes a graph $\Theta_\sigma$. The graph associated with the two children of a vertex $\sigma$ are obtained from $\Theta_\sigma$ using the maps $\tau_0$, $\tau_1$, and $\rho$. More precisely, an arrow towards the left corresponds to the map $\rho \circ \tau_0$ and an arrow towards the right to the map $\rho \circ \tau_1$. In this example, the graph $\Theta_{11}$ satisfies the stopping rule from Definition \ref{def:stopping_rule}, and is therefore a leaf of $\cal T$.\label{fig:tree}} 
\end{figure}

We use the notation $i \sigma$, for $i = 0,1$, to denote the binary string $\sigma$ to which $i$ has been appended on the left. The children  in $\cal T$   of the vertex  $\sigma \in V(\cal T)$  are $0 \sigma$ and $1 \sigma$.
The precise construction of $\Theta_\sigma$ and the binary tree $\cal T$ is as follows. 
Let $\ell > 0$ be a cutoff to be chosen later (see \eqref{choice of ell} below); it will be used as a threshold for the stopping rule which ensures that the tree $\cal T$ is finite. Let $d(\Gamma)$ denote the number of off-diagonal $G$-edges plus the number of off-diagonal $R$-groups of $\Gamma$, i.e.
\begin{equation} \label{definition of d}
d(\Gamma) \;\deq\; \sum_{e \in E(\Gamma)} \pB{\ind{\xi_1(e) \in \{G,G^*\}} \indb{\alpha(e) \neq \beta(e)} + \indb{\xi_1(e) \in \{R,R^*\}} \indb{A(e) \neq B(e)}}\,.
\end{equation}
The construction of the tree $\cal T$ relies on the following stopping rule.
\begin{definition}[Stopping rule] \label{def:stopping_rule}
We say that a graph $\Gamma$ \emph{satisfies the stopping rule} if $d(\Gamma) \geq \ell$ or if all $G$-edges of $\Gamma$ are maximally expanded.
\end{definition}
The tree $\cal T$, along with the graphs $(\Theta_\sigma)_{\sigma \in V(\cal T)}$, is constructed recursively from the trivial tree, consistsing of the single vertex $\emptyset$ with $\Theta_\emptyset = \Delta$, as follows. Let $\sigma$ be a leaf of the tree such that $\Theta_\sigma$ does not satisfy the stopping rule. We add the children of $\sigma$, i.e.\ $0 \sigma$ and $1 \sigma$, to the tree, and set
\begin{equation*}
\Theta_{0 \sigma} \;\deq\; \rho (\tau_0(\Theta_\sigma))\,, \qquad \Theta_{1 \sigma} \;\deq\; \rho (\tau_1(\Theta_\sigma))\,.
\end{equation*}
We continue this recursion on each leaf until all leaves satisfy the stopping rule from Definition \ref{def:stopping_rule}. By Lemma \ref{lem: finite tree} below, the resulting tree $\cal T$ is finite, i.e.\ the recursion terminates after a finite number of steps.

\begin{lemma} \label{lem: props of Gamma sigma}
The graphs $\Theta_\sigma$ have the following two properties. First,
\begin{equation} \label{prop 1 of Gamma sigma}
V_b(\Theta_{0 \sigma}) \;=\; V_b(\Theta_{1 \sigma}) \;=\; V_b(\Theta_\sigma)\,.
\end{equation}
In particular,  the set of black vertices remains unchanged throughout the recursion:  $V_b(\Theta_\sigma) = V_b(\Delta)$. Second,
\begin{equation} \label{prop 2 of Gamma sigma}
\sum_{\f a_w} \cal A_{\f a_b, \f a_w} (\Theta_{0 \sigma}) + \sum_{\f a_w} \cal A_{\f a_b, \f a_w} (\Theta_{1 \sigma}) \;=\; \sum_{\f a_w} \cal A_{\f a_b, \f a_w} (\Theta_\sigma)\,.
\end{equation}
Note that both sides depend on $\f a_b$, and we slightly abuse notation as explained after \eqref{identity for rho}.

Moreover, each $\Theta_\sigma$ for $\sigma \in V(\cal T)$ satisfies the properties (i)--(vii) from Sections \ref{sec: exp 1} and \ref{sec: exp 1.1}. 
\end{lemma}
\begin{proof}
The identity \eqref{prop 1 of Gamma sigma} follows immediately from \eqref{Vb tau} and \eqref{Vb rho}. Similarly, \eqref{prop 2 of Gamma sigma} follows from \eqref{sigma splitting} and \eqref{identity for rho}.
 The final statement follows immediately from Lemmas \ref{lem: properties 1} and \ref{lem: properties 2}. 
\end{proof}

The interpretation of \eqref{prop 2 of Gamma sigma} is that the value of any graph $\Theta_\sigma$ is equal to the sum of the values of its two children, $\Theta_{0\sigma}$ and $\Theta_{1 \sigma}$.

The following estimate ensures that the tree $\cal T$ is finite, i.e.\ that the expansion procedure does not produce
and infinite sequence of graphs whose value $d(\cdot)$ remains below $\ell$ indefinitely.
\begin{lemma} \label{lem: finite tree}
The tree $\cal T$ has depth at most $2p (p + 6 \ell)$  and consequently at most $2^{2p(p+6\ell)}$ vertices. 
\end{lemma}

\begin{proof}
Observe that $\tau_0$ and $\rho$ leave the function $d$ defined in \eqref{definition of d}  invariant: 
  $d(\tau_0(\Gamma))=d(\Gamma)$, $d(\rho(\Gamma))=d(\Gamma)$. Moreover, $\tau_1$ increases $d$ by at least one (for a diagonal entry the increase is two): $d(\tau_1(\Gamma)) \geq d(\Gamma)+1$. We conclude that, by the stopping rule from Definition \ref{def:stopping_rule}, any string $\sigma$ of the tree contains at most $\ell$ ones, i.e.\ that $\tau_1$ has been applied at most $\ell$ times.

Next, let $f= f(\Gamma)$ denote the number of $G$-edges minus the number of $R$-edges in the graph $\Gamma$. It follows immediately that $f$ is left
 invariant  by $\tau_0$ and $\rho$, and is increased by at most 4 by $\tau_1$: $f(\tau_0(\Gamma)) = f(\rho(\Gamma)) = f(\Gamma)$ and $f(\tau_1(\Gamma))\leq f(\Gamma)+ 4$.  Since in the initial graph there is no $R$-edge, so that $f(\Delta) =  \abs{E(\Delta)}$, we conclude that $f(\Theta_\sigma) \leq \abs{E(\Delta)} + 4 \ell = p + 4 \ell$ for all $\sigma \in V(\cal T)$. By Definition \ref{def:stopping_rule}, the number of $R$-edges is bounded by $\ell$. (Note that only off-diagonal $R$-groups have been created along the procedure, so that
the number of $R$-edges is the same as the number of off-diagonal $R$ groups. Diagonal $R$-groups will appear in later in Section \ref{sec: exp 6}).
 Hence we conclude that the number of $G$-edges of any $\Theta_\sigma$ is bounded by $p + 5 \ell$.

In order to estimate the number of zeros in the string $\sigma$, we note that, since each $G$-entry can have at most $\abs{V(\Delta)} \leq 2p$ upper indices, the total number of upper indices in all the $G$-entries of $\cal A_{\f a}(\Theta_\sigma)$ is bounded by $2p (p + 5 \ell)$. We conclude by noting that $\tau_1$ and $\rho$ do not decrease the total number of upper indices in the $G$-entries, while $\tau_0$ increases this number by one. Hence the total number of zeros in any string $\sigma$ is bounded by $2p (p + 5 \ell)$. Thus, the total length of $\sigma$ is 
bounded by $2p (p + 5 \ell)+\ell \leq 2p(p +6\ell)$. This  concludes the proof.
\end{proof}

Next, we express $Y(\Delta)$ from \eqref{def Y Delta} in terms of the graphs we just introduced.
By Lemma \ref{lem: props of Gamma sigma} and the fact that $V_b(\Delta)=V(\Delta)$, we have for all $\sigma \in V(\cal T)$ that
\begin{equation}
V_b(\Theta_\sigma) \;=\; V(\Delta).
\end{equation}
Let $L(\cal T) \subset  V(\cal T)$ denote the leaves of $\cal T$. The identity \eqref{prop 2 of Gamma sigma} states that if $\sigma \in V(\cal T)$ is not a leaf of $\cal T$, we may replace the value of $\Theta_\sigma$ by the sum of the values of its two children. Starting from the root $\emptyset$ and the graph $\Theta_\emptyset = \Delta$, we may propagate this identity recursively from the root down to the leaves. We conclude that
\begin{equation} \label{Delta as sum over leaves}
\cal A_{\f a_b}(\Delta) \;=\; \sum_{\sigma \in L(\cal T)} \sum_{\f a_w} \cal A_{\f a_b, \f a_w}(\Theta_\sigma)\,.
\end{equation}
Recalling the definition \eqref{def Y Delta} of $Y(\Delta)$, we get the following result.
\begin{proposition} \label{prop:Y_tree}
The quantity $Y(\Delta)$ defined in \eqref{def Y Delta} may be written in terms of the tree $\cal T$ as
\begin{equation} \label{Y in terms of tree}
Y(\Delta) \;=\; \sum_{\sigma \in L(\cal T)} \sum_{\f a_b}^{*} w_{\f a_b}(\Delta) \sum_{\f a_w}  \E \, \cal A_{\f a_b, \f a_w}(\Theta_\sigma)\,.
\end{equation}
\end{proposition}

For the following we partition $L(\cal T) = L_0(\cal T) \cup L_1(\cal T)$ into the \emph{trivial leaves} $L_0(\cal T)$ and the \emph{nontrivial leaves} $L_1(\cal T)$. By definition, the trivial leaves of $\cal T$ are those $\sigma \in  V(\cal T)$ satisfying $d(\Theta_\sigma) \geq \ell$. We shall estimate the contribution of the trivial leaves brutally in Section \ref{sec: exp 5} below, using the fact that they contain a large enough number of small factors.

By Definition \ref{def:stopping_rule}, if $\sigma \in L_1(\cal T)$ is a nontrivial leaf then all $G$-edges of $\Theta_\sigma$ are diagonal and maximally expanded. The estimate of the nontrivial leaves will be performed in Sections \ref{sec: exp 6}--\ref{sec: exp 8}.

\subsection{The trivial leaves} \label{sec: exp 5}
In this section we estimate the contribution of $\Theta_\sigma$ for a trivial leaf $\sigma \in L_0(\cal T)$. Thus, fix $\sigma \in L_0(\cal T)$. From \eqref{Gab expansion} and Lemma \ref{lem: wt G bounds} we get for $a \neq b$
\begin{equation} \label{estimates on R from rho}
\sum_{\mu,\nu} X_{a \mu} R^{(\f a_b)}_{\mu \nu} X^*_{\nu b} \;\prec\; \phi^{-1/2} \Psi\,, \qquad
\sum_{\mu,\nu} X_{a \mu} R^{*(\f a_b)}_{\mu \nu} X^*_{\nu b} \;\prec\; \phi^{-1/2} \Psi\,.
\end{equation}
We therefore conclude that each off-diagonal $R$-group of $\Gamma$ yields a contribution of size $O_\prec (\phi^{-1/2} \Psi)$ after summation over the indices associated with the vertices incident to its centre. Moreover, by definition of $\cal T$, each $R$-group of $\Theta_\sigma$ is off-diagonal. In addition, each off-diagonal $G$-edge yields a contribution of size $\phi^{-1/2} \Psi$ by Lemma \ref{lem: wt G bounds}. Thus we get, summing out all indices associated with white vertices (i.e.\ inner vertices of $R$-groups),
\begin{equation*}
\sum_{\f a_w} \cal A_{\f a_b, \f a_w}(\Theta_\sigma) \;\prec\; (\phi^{-1/2} \Psi)^{d(\Theta_\sigma)} \;\leq\; (\phi^{-1/2} \Psi)^{\ell}\,.
\end{equation*}
Hence the contribution of $\Theta_\sigma$ to the right-hand side of \eqref{Y in terms of tree} may be bounded by
\begin{equation*}
\sum_{\f a_b}^{*} w_{\f a_b}(\Delta) \sum_{\f a_w}  \E \, \cal A_{\f a_b, \f a_w}(\Theta_\sigma) \;\prec\; M^{2p} (\phi^{-1/2} \Psi)^\ell\,,
\end{equation*}
where we estimated the summation over $\f a_b$ by $M^{2p}$ using the trivial bound $\abs{w_{\f a_b}(\Delta)} \leq 1$ (from \eqref{estimate for w} and $\norm{\f v}_2=1$). In the last step we used Lemma \ref{lemma: basic properties of prec} (i) and (iii). The assumption $\E Z^2 \leq N^C$ of Lemma \ref{lemma: basic properties of prec} (iii) for the random variable $Z = \absb{\sum_{\f a_w} \cal A_{\f a_b, \f a_w}(\Theta_\sigma)}$ follows from the following lemma combined with H\"older's inequality, and from the fact that the number of white vertices of $\Theta_\sigma$ is independent of $N$, so that the sum $\sum_{\f a_w}$ contains $O(N^C)$ terms.

\begin{lemma}
For any $p$ there exists a constant $C_p$ such that for any graph $\Gamma$ and any $e \in E(\Gamma)$ we have
\begin{equation*}
\E \absb{\cal A_{\f a}(e, \Gamma)}^p \;\leq\; M^{C_p}\,.
\end{equation*}
\end{lemma}
\begin{proof}
The cases $\xi_1(e) \in \{G,G^*,R,R^*\}$ and $\xi_2(e) = +$ are dealt with the deterministic estimates
\begin{equation*}
\absb{\wt G_{ij}^{(T)}} \;\leq\; N \phi^{1/2} \;\leq\; M^2\,, \qquad \absb{R_{ij}^{(T)}} \;\leq\; N \leq M\,,
\end{equation*}
which follows from $\abs{G_{ij}^{(T)}}, \abs{R_{ij}^{(T)}} \leq \eta^{-1} \leq N$.
The cases $\xi_1(e) \in \{X,X^*\}$ follow immediately from \eqref{moments of X-1}. Finally, the cases $\xi_1(e) \in \{G,G^*\}$ and $\xi_2(e) = -$ follow easily from \eqref{Gii expanded sc}.
\end{proof}

Using Lemma \ref{lem: finite tree}, we therefore conclude that the contribution of all trivial leaves to the right-hand side of \eqref{Y in terms of tree} is bounded by
\begin{equation} \label{main estimate of trivial leaves}
\sum_{\sigma \in L_0(\cal T)} \sum_{\f a_b}^{*} w_{\f a_b}(\Delta) \sum_{\f a_w}  \E \, \cal A_{\f a_b, \f a_w}(\Theta_\sigma) \;\prec\; C_{p, \ell} \, M^{2p} (\phi^{-1/2} \Psi)^\ell \;\leq\; C_{p, \omega} (\phi^{-1} \Psi)^p\,,
\end{equation}
where $C_{p, \ell} = 2^{2p(p+6\ell)}$ estimates the number of vertices in $\cal T$ (see Lemma \ref{lem: finite tree}). The last step holds provided we choose
\begin{equation} \label{choice of ell}
\ell \;\deq\; \pbb{\frac{8}{\omega} + 2} p\,.
\end{equation}
Here we used the bound $\Psi \leq C N^{-\omega / 2}$, which follows from the definitions \eqref{def of Psi MP},  \eqref{def_S_theta}, and \eqref{bounds on mg}.

\subsection{The nontrivial leaves I: Operation (c)} \label{sec: exp 6}

From now on we focus on the nontrivial leaves, $\sigma \in L_1(\cal T)$. Our goal is to prove the following estimate, which is analogous to \eqref{main estimate of trivial leaves}. Its proof will be the content of this and the two following subsections, and will be completed at the end of Section~\ref{sec: exp 8}.

\begin{proposition} \label{prop: nontrivial leaves}
We have the bound
\begin{equation*}
\sum_{\sigma \in L_1(\cal T)} \sum_{\f a_b}^{*} w_{\f a_b}(\Delta) \sum_{\f a_w}  \E \, \cal A_{\f a_b, \f a_w}(\Theta_\sigma) \;\prec\; C_{p, \omega} (\phi^{-1} \Psi)^p\,.
\end{equation*}
\end{proposition}

By definition of $L_1(\cal T)$, all $G$-edges of $\Theta_\sigma$ are diagonal and
maximally expanded for any $\sigma \in L_1(\cal T)$. The first step behind the proof of Proposition \ref{prop: nontrivial leaves} uses Operation (c) from Section \ref{sec: exp 1}, i.e.\ expanding all diagonal $G$-entries of $\cal A_{\f a}(\Theta_\sigma)$ using \eqref{Gii expanded sc}. Roughly, this amounts to replacing diagonal $G$-edges by (a collection of) diagonal $R$-groups. More precisely, for entries in the denominator we use the identity
\begin{equation} \label{Gaa denom}
\frac{1}{\wh G_{aa}} \;=\; -\wt z - \wt z \sum_{\mu, \nu} X_{a \mu} R^{(\f a_b)}_{\mu \nu} X_{\nu a}^*\,.
\end{equation}
In order to handle entries in the numerator, we rewrite this identity in the form
\begin{equation} \label{Gaa numer 0}
\frac{1}{\wh G_{aa}} \;=\; \frac{1}{\wt m_\phi} - \pBB{\wt z \sum_{\mu, \nu} X_{i \mu} R^{(\f a_b)}_{\mu \nu} X_{\nu a}^* - \wt z \phi^{-1/2} m_\phi}\,,
\end{equation}
where used the definition \eqref{def tilde m} of $\wt m_\phi$ and
that $m_\phi$ satisfies the identity \eqref{identity for m MP}. From Lemma \ref{lem: wt G bounds} and \eqref{wt m gamma is bounded} we get $1/ \wh G_{aa} - 1 / \wt m_\phi \prec \phi^{-1/2} \Psi$. Thus we may expand the inverse of \eqref{Gaa numer 0} up to order $\ell$:
\begin{equation} \label{Gaa numer}
\wh G_{aa} \;=\; \sum_{k = 0}^{\ell - 1} \wt m_\phi^{k + 1} \pBB{\wt z \sum_{\mu, \nu} X_{i \mu} R^{(\f a_b)}_{\mu \nu} X_{\nu a}^* - \wt z \phi^{-1/2} m_\phi}^k + O_\prec \pb{(\phi^{-1/2} \Psi)^\ell}\,.
\end{equation}
This is our main expansion for the diagonal $G$-entries in the numerator. Both formulas \eqref{Gaa denom} and \eqref{Gaa numer} have trivial analogues for the Hermitian conjugate $\wh G_{aa}^*$.

Recall that all $G$-entries of $\cal A_{\f a}(\Theta_\sigma)$ are diagonal and maximally expanded. We apply \eqref{Gaa denom} or \eqref{Gaa numer} to each $G$-entry of $\cal A_{\f a}(\Theta_\sigma)$, and multiply everything out. The result may be written in terms graphs as
\begin{equation} \label{A Theta expanded}
\sum_{\f a_w} \cal A_{\f a_b, \f a_w}(\Theta_\sigma) \;=\; \sum_{\Gamma \in \fra G(\Theta_\sigma)} \sum_{\f a_w} \cal A_{\f a_b, \f a_w}(\Gamma) + O_\prec \pb{(\phi^{-1/2} \Psi)^\ell}\,,
\end{equation}
where the error term $O_\prec \pb{(\phi^{-1/2} \Psi)^\ell}$ contains all terms containing at least one error term from the expansion \eqref{Gaa numer}. The sum on the right-hand side of \eqref{A Theta expanded} consists of monomials in the entries of $R^{(\f a_b)}$, $R^{* (\f a_b)}$, $X$, and $X^*$ (note that entries of $\wh G$ and $\wh G^*$ no longer appear), and can hence be encoded using a family graphs which we call $\fra G(\Theta_\sigma)$. By construction, the family $\fra G(\Theta_\sigma)$ is finite. (In fact, it satisfies $\abs{\fra G(\Theta_\sigma)} \leq \ell^{6 \ell}$, where we used that the number of $G$-entries of $\cal A_{\f a}(\Theta_\sigma)$ to which \eqref{Gaa denom} or \eqref{Gaa numer} are applied is bounded by $p+5\ell \leq 6 \ell$; see the proof of Lemma \ref{lem: finite tree}.)

Exactly as in Section \ref{sec: exp 5}, we may brutally estimate the contribution of the rest term on the right-hand side of \eqref{A Theta expanded} by
\begin{equation*}
\sum_{\f a_b}^{*} w_{\f a_b}(\Delta) O_\prec \pb{(\phi^{-1/2} \Psi)^\ell} \;\prec\; C_{p, \omega} (\phi^{-1} \Psi)^p
\end{equation*}
with $\ell$ defined in \eqref{choice of ell}; we omit the details.

Hence, in order to complete the proof of Proposition \ref{prop: nontrivial leaves}, it suffices to prove that for all $\sigma \in L_1(\cal T)$ and all $\Gamma \in \fra G(\Theta_\sigma)$ we have
\begin{equation}\label{expanded terms}
\sum_{\f a_b}^{*} w_{\f a_b}(\Delta) \sum_{\f a_w} \E \cal A_{\f a_b, \f a_w}(\Gamma) \;\prec\; C_{p, \omega} (\phi^{-1} \Psi)^p\,.
\end{equation}
As before, the map $\Theta_\sigma \mapsto \fra G(\Theta_\sigma)$ may be explicitly given on the level of graphs, but we refrain from doing so. Instead, we illustrate this process for some simple cases in Figure \ref{fig: expanding Gaa}.
\begin{figure}[ht!]
\begin{center}
\includegraphics{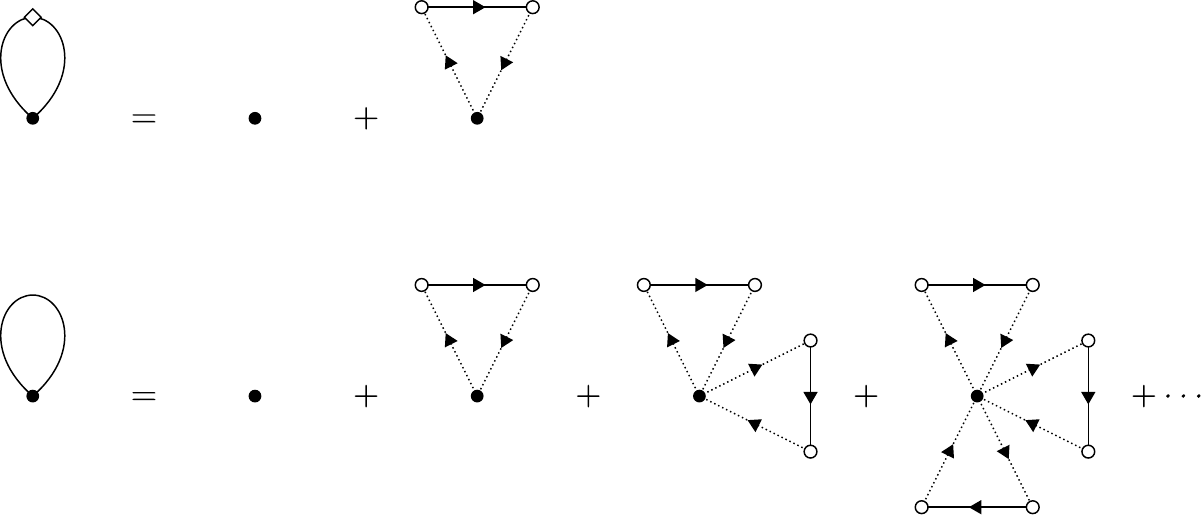}
\end{center}
\caption{A graphical depiction of the identities \eqref{Gaa denom} and \eqref{Gaa numer} that generate $\fra G(\Theta_\sigma)$ from $\Theta_\sigma$. A $G$-edge encoding an entry in the denominator is replaced by either nothing (leaving just the vertex) or a diagonal $R$-group. A $G$-edge encoding an entry in the numerator is replaced by either nothing or up to $\ell - 1$ diagonal $R$-groups. \label{fig: expanding Gaa}}
\end{figure}

\subsection{The nontrivial leaves II: taking the expectation} \label{sec: exp 7}

Let us now consider a nontrivial leaf $\sigma \in L_1(\cal T)$. By definition of $L_1(\cal T)$, all $G$-edges of $\Theta_\sigma$ are diagonal and maximally expanded. Therefore, any $\Gamma \in \fra G(\Theta_\sigma)$ does not contain any $G$-edges. This was the goal of the expansion
 generated by Operations  (a)--(c). Hence, each $\Gamma \in \fra G(\Theta_\sigma)$ consists solely of $R$-groups.

Let $\sigma \in L_1(\cal T)$ and $\Gamma \in \fra G(\Theta_\sigma)$. Fix the summation indices $\f a_b$, and recall that $a_i \neq a_j$ for $i,j \in V_b(\Gamma)$ and $i \neq j$. By definition of $R^{(\f a_b)}$, the $\abs{V_b(\Gamma)} + 1$ families $\pb{R^{(\f a_b)}_{\mu \nu}}_{\mu,\nu = 1}^N$ and
$\p{X_{a_i \mu}}_{\mu = 1}^N$, $i \in V_b(\Gamma)$, are independent. Therefore we may take the expectation of the $R$-entries and the $X$-entries separately. The expectation of the $X$-entries may be kept track of using partitions, very much like in Section \ref{sec: graphs}, except in this case the partition is on the white vertices. In fact, the combinatorics here are much simpler, since two white vertices may only be in the same block of the partition if they are adjacent to a common black vertex. Indeed, the (Latin) indices associated with two different black vertices are different, so that the two entries of $X$ encoded by two $X$-edges incident to two different black vertices are independent, since  $X_{a\mu}$ and $X_{b\nu}$ are independent if $a\neq b$ for all $\mu$ and $\nu$ (even if $\mu=\nu$). The precise definition is the following.

We recall from Property (vii) in Section \ref{sec: exp 1.1} that each white vertex $j \in  V_w(\Gamma)$ is adjacent in $\Gamma$  to a unique black vertex $\pi(j)  \equiv \pi_\Gamma(j)$. For each $i \in V_b(\Gamma)$ we introduce a partition $\zeta_i$ of the subset of white vertices $\pi^{-1}(\{i\})$, and constrain the values of the indices $(a_j \col \pi(j) = i)$ to be compatible with $\zeta_i$. On the level of graphs, such a partition amounts to merging vertices in $\pi^{-1}(\{i\})$. Abbreviate $\zeta = (\zeta_i)_{i \in V_b(\Gamma)}$, and denote by $\Gamma_\zeta$ the graph obtained from $\Gamma$ by merging, for each $i \in V_b(\Gamma)$, the vertices adjacent to $i$ according to $\zeta_i$. Note that, like $\Gamma$, each $\Gamma_\zeta$ satisfies the properties (i)--(vi) from Section \ref{sec: exp 1}, but, unlike $\Gamma$, in general $\Gamma_\zeta$ does not satisfy the property (vii) from Section \ref{sec: exp 1.1}. See Figure \ref{fig: merging} for an illustration of the mapping $\Gamma \mapsto \Gamma_\zeta$.
\begin{figure}[ht!]
\begin{center}
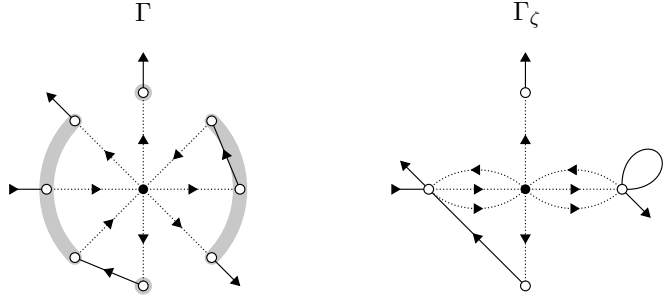
\end{center}
\caption{The process $\Gamma \mapsto \Gamma_\zeta$. Since this operation is local at each black vertex, we only draw the neighbourhood of (more precisely the unit ball around) a selected black vertex $i \in V_b(\Gamma)$. The depicted black vertex is part of two diagonal $R$-blocks and four off-diagonal $R$-blocks; the latter ones are not drawn completely.
The blocks of the partition $\zeta_i$ are drawn in grey. On the right we draw the corresponding neighbourhood of $\Gamma_\zeta$. \label{fig: merging}}
\end{figure}

Define the indicator function
\begin{equation*}
\chi_{\f a_w} (\Gamma) \;\deq\; \prod_{i \in V_b(\Gamma)} \indb{a_j \neq a_{j'} \text{ for } j,j' \in \pi^{-1}_{\Gamma }(\{i\}) \text{ and } j \neq j'}\,,
\end{equation*}
which constrains the summation indices associated with different white vertices adjacent to the same black vertex to have different values.
By definition of $\Gamma_\zeta$, we therefore have
\begin{equation*}
\sum_{\f a_w} \cal A_{\f a_b, \f a_w}(\Gamma) \;=\; \sum_\zeta \sum_{\f a_w}  \chi_{\f a_w}(\Gamma_\zeta)  \cal A_{\f a_b, \f a_w}(\Gamma_\zeta)\,,
\end{equation*}
where the sum ranges over all families of partitions $\zeta = (\zeta_i)_{i \in V_b(\Gamma)}$. 
As before, the summation of the white indices ${\f a}_w$ runs over different sets
on the left- and the right-hand sides, owing to the merging white vertices.
Taking the expectation yields
\begin{equation} \label{split Gamma Gamma zeta}
\sum_{\f a_w} \E \cal A_{\f a_b, \f a_w}(\Gamma) \;=\; \sum_\zeta \sum_{\f a_w} \chi_{\f a_w}(\Gamma_\zeta) \pBB{\E \prod_{e \in E_R(\Gamma_\zeta)} \cal A_{\f a_b, \f a_w}(e, \Gamma_\zeta)} W_{\f a_b, \f a_w}(\Gamma_\zeta)\,,
\end{equation}
where we set
\begin{equation*}
W_{\f a_b, \f a_w}(\Gamma_\zeta) \;\deq\;
\prod_{i \in V_b(\Gamma_\zeta)} \pBB{\E \prod_{e \in E_i(\Gamma_\zeta)} \cal A_{\f a_b, \f a_w}(e, \Gamma_\zeta)}
\end{equation*}
and abbreviated $E_R(\cdot)$ for the set of $R$-edges and $E_i(\cdot)$ for the set of $X$-edges incident to $i$.
Here we used the independence described above.

Since $\E X_{a \mu} = 0$, we immediately get that $W_{\f a_b, \f a_w}(\Gamma_\zeta) = 0$ unless, for each $i \in V_b(\Gamma)$, each block of $\zeta_i$ has size at least two. By \eqref{moments of X-1} we get in fact that
\begin{equation} \label{bound W}
\absb{W_{\f a_b, \f a_w}(\Gamma_\zeta)} \;\leq\; C_\Gamma (N M)^{- \abs{V_w(\Gamma)}/4} \prod_{i \in V_b(\Gamma)} \indb{\text{each block of $\zeta_i$ has size at least two}}\,.
\end{equation}
The following result is the main power counting estimate for $W_{\f a_b, \f a_w}$. It
shows that each black vertex of degree one in $\Delta$ (corresponding to Latin indices
that remained unpaired in the partition \eqref{E Xp 2}) results in an extra factor $M^{-1/2}$.
This will balance the passage from $\ell^1$- to $\ell^2$-norm of $\f v$, as explained in 
Section~\ref{sec:sketch}.

Note that by definition of $\Gamma_\zeta$ we have $V_b(\Gamma_\zeta) = V_b(\Gamma)$ and $\deg_\Gamma(i) = \deg_{\Gamma_\zeta}(i)$ for all $i \in V_b(\Gamma)$. For the following we therefore drop the argument of $V_b$. Define the subset 
\begin{equation*}
V_b^* \;\deq\; \hb{i \in V_b \col \deg_\Delta(i) = 1}\,.
\end{equation*}
For $i \in V_b(\Gamma)$ let $n_\zeta(i)$ denote the number of  vertices of $\Gamma_\zeta$ adjacent to $i$ (these are
all white since there are no $G$-edges in $\Gamma_\zeta$, which are the only edges that join two black vertices).
\begin{lemma} \label{lem: final power counting}
We have the bound
\begin{equation*}
\absb{W_{\f a_b, \f a_w}(\Gamma_\zeta)} \;\leq\; C_\Gamma \, \phi^{-p/2} \prod_{i \in V_b} N^{-n_\zeta(i)} \prod_{i \in V_b^*} M^{-1/2}\,.
\end{equation*}
\end{lemma}
\begin{proof}
Recalling \eqref{bound W}, we assume without loss of generality that, for each $i \in V_b$, eack block of $\zeta_i$ has size at least two; in particular, we assume that for each $i \in V_b$ we have $\deg_\Gamma(i) \geq 2$.
From \eqref{bound W} we get
\begin{align*}
\absb{W_{\f a_b, \f a_w}(\Gamma_\zeta)} &\;\leq\; C_\Gamma (N M)^{- \abs{V_w(\Gamma)}/4}
\\
&\;=\; C_\Gamma \prod_{i \in V_b \setminus V_b^*} (N M)^{-\deg_\Gamma(i) / 4} \prod_{i \in V_b^*} (N M)^{-\deg_\Gamma(i) / 4}
\\
&\;=\; C_\Gamma \prod_{i \in V_b \setminus V_b^*} \pB{N^{-\deg_\Gamma(i) / 2} \phi^{-\deg_\Gamma(i)/4}} \prod_{i \in V_b^*} \pB{M^{-1/2} N^{-(\deg_\Gamma(i) - 1)/2} \phi^{1/2 - \deg_\Gamma(i) / 4}}\,.
\end{align*}

By definition, $\tau_0$ and $\rho$ leave $\deg(i)$ invariant, and $\tau_1$ increases $\deg(i)$ by $0$ or $4$. In particular, they all leave the parity of $\deg(i)$ invariant for $i \in V_b$. We conclude that $\deg_\Gamma(i)$ is odd for each $i \in V_b^*$. Since each block of $\zeta_i$ has size at least two, we find that
\begin{equation*}
n_\zeta(i) \;\leq\;
\begin{cases}
\frac{\deg_\Gamma(i)}{2} & \text{if } \deg_\Gamma(i) \text{ is even}
\\
\frac{\deg_\Gamma(i) - 1}{2} & \text{if } \deg_\Gamma(i) \text{ is odd}\,.
\end{cases}
\end{equation*}
We therefore conclude that
\begin{equation*}
\absb{W_{\f a_b, \f a_w}(\Gamma_\zeta)} \;\leq\; C_\Gamma \prod_{i \in V_b \setminus V_b^*} \pB{N^{-n_\zeta(i)} \phi^{-\deg_\Gamma(i)/4}} \prod_{i \in V_b^*} \pB{M^{-1/2} N^{- n_\zeta(i)} \phi^{1/2 - \deg_\Gamma(i) / 4}}\,.
\end{equation*}
The proof is then completed by the following claim.

If $\deg_\Gamma(i) \geq 2$ for all $i \in V_b^*$ then
\begin{equation} \label{main power counting}
\prod_{i \in V_b} \phi^{-\deg_\Gamma(i) / 4} \prod_{i \in V_b^*} \phi^{1/2} \;\leq\; \phi^{-p/2}\,.
\end{equation}
What remains is to prove \eqref{main power counting}. 
Since $\Delta$ has $p$ edges, we find
\begin{equation*}
\prod_{i \in V_b} \phi^{-\deg_\Delta(i) / 4} \;=\; \phi^{-p/2}\,.
\end{equation*}
As observed above, $\tau_0$ and $\rho$ leave $\deg(i)$ invariant, and $\tau_1$ increases $\deg(i)$ by $0$ or $4$. Let $i \in V_b^*$. Since by assumption $\deg_\Gamma(i) \geq 2$, we find that in fact $\deg_\Gamma(i) \geq 5$. This yields
\begin{equation*}
\prod_{i \in V_b} \phi^{-\deg_\Gamma(i) / 4} \;\leq\; \prod_{i \in V_b \setminus V_b^*} \phi^{-\deg_\Gamma(i) / 4} \prod_{i \in V_b^*} \phi^{- 5 / 4} \;\leq\; \prod_{i \in V_b \setminus V_b^*} \phi^{-\deg_\Delta(i) / 4} \prod_{i \in V_b^*} \phi^{-\deg_\Delta(i) / 4} \phi^{- 1}\,,
\end{equation*}
from which \eqref{main power counting} follows.
\end{proof}

\subsection{The nontrivial leaves III: summing over $\f a$ and conclusion of the proof of Proposition \ref{prop: nontrivial leaves}} \label{sec: exp 8}
As above, fix a tree vertex $\sigma \in L_1(\cal T)$, a graph $\Gamma \in \fra G(\Theta_\sigma)$, and a partition $\zeta$.
In order to conclude the proof, we use Lemma \ref{lem: final power counting} on each $\Gamma_\zeta$ to sum over $\f a_w$ in \eqref{split Gamma Gamma zeta}.

Recall the quantity $d$ from \eqref{definition of d}, defined as the number of off-diagonal $G$-edges plus the number of off-diagonal $R$-groups. By definition of $\Delta$, $d(\Delta) = p$. Moreover, $\tau_0$, $\tau_1$, and $\rho$ do not decrease $d$. Since by construction $\Gamma$ has no $G$-entries, we conclude that $\Gamma$ has at least $p$ off-diagonal $R$-groups. We may therefore choose a set $E_o(\Gamma) \subset E_R(\Gamma)$ of size at least $p$, such that each $e \in E_o(\Gamma)$ is the centre of an off-diagonal $R$-group (see Section \ref{sec: exp 1.1}). The set $E_o$ is naturally mapped into $E_R(\Gamma_\zeta)$, and is denoted by $E_o(\Gamma_\zeta)$. 
We denote by $\alpha(e)$ and $\beta(e)$ the end points of $e$ in $\Gamma_\zeta$. By \eqref{bound: Rij 2}, we have
\begin{equation*}
\prod_{e \in E_R(\Gamma_\zeta)} \cal A_{\f a_b, \f a_w}(e, \Gamma_\zeta) \;\prec\; \prod_{e \in E_o(\Gamma_\zeta)} \Psi^{\ind{\alpha(e) \neq \beta(e)}}\,.
\end{equation*}
As in Section \ref{sec: exp 5}, it is easy to take the expectation using Lemma \ref{lemma: basic properties of prec} to get
\begin{equation*}
\E \prod_{e \in E_R(\Gamma_\zeta)} \cal A_{\f a_b, \f a_w}(e, \Gamma_\zeta) \;\prec\; \prod_{e \in E_o(\Gamma_\zeta)} \Psi^{\ind{\alpha(e) \neq \beta(e)}}\,.
\end{equation*}
We may now sum over $\f a_w$ on the right-hand side of \eqref{split Gamma Gamma zeta}: from Lemma \ref{lem: final power counting} we get
\begin{align*}
&\mspace{-40mu}\sum_{\f a_w} \chi_{\f a_w}(\Gamma_\zeta) \pBB{\E \prod_{e \in E_R(\Gamma_\zeta)} \cal A_{\f a_b, \f a_w}(e, \Gamma_\zeta)} W_{\f a_b, \f a_w}(\Gamma_\zeta)
\\
&\;=\; \sum_{\f a_w} \chi_{\f a_w}(\Gamma_\zeta) \pBB{\E \prod_{e \in E_R(\Gamma_\zeta)} \cal A_{\f a_b, \f a_w}(e, \Gamma_\zeta)} W_{\f a_b, \f a_w}(\Gamma_\zeta) \prod_{e \in E_o(\Gamma_\zeta)} \pB{\ind{\alpha(e) = \beta(e)} + \ind{\alpha(e) \neq \beta(e)}}
\\
&\;\prec\; C_\Gamma \, \phi^{-p/2} \prod_{i \in V_b} N^{-n_\zeta(i)} \prod_{i \in V_b^*} M^{-1/2} \sum_{k = 0}^p \Psi^{p - k} N^{\abs{V_w(\Gamma_\zeta)} - k}
\\
&\;\leq\; C_\Gamma \phi^{-p/2} \Psi^p \prod_{i \in V_b^*} M^{-1/2}\,.
\end{align*}
In the second step we multiplied out the last $p$-fold product on the second line and classified all terms according to number, $k$, of factors $\ind{\alpha(e) = \beta(e)}$; we used that the total number of free summation variables is $\abs{V_w(\Gamma_\zeta)} - k$. In the third step we used that $\sum_{i \in V_b} n_\zeta(i) = \abs{V_w(\Gamma_\zeta)}$ and the bound $\Psi \geq N^{-1}$.

Returning to \eqref{split Gamma Gamma zeta}, we find
\begin{equation*}
\sum_{\f a_w} \E \cal A_{\f a_b, \f a_w}(\Gamma) \;\prec\; C_\Gamma \phi^{-p/2} \Psi^p \prod_{i \in V_b^*} M^{-1/2}\,.
\end{equation*}
We may now sum over $\f a_b$ to prove \eqref{expanded terms}.
 Using the bound \eqref{estimate for w}, we therefore get
\begin{equation*}
\sum_{\f a_b}^{*} w_{\f a_b}(\Delta) \sum_{\f a_w}  \E \, \cal A_{\f a_b, \f a_w}(\Gamma) \;\prec\; C_\Gamma \phi^{-p/2} \Psi^p \sum_{\f a_b} \prod_{i \in V_b^*} M^{-1/2} \prod_{i \in V(\Delta)} \abs{v_{a_i}}^{\deg_\Delta(i)} \;\prec\; C_\Gamma \phi^{-p/2} \Psi^p\,,
\end{equation*}
where the last step follows from the fact that, by definition of $\Delta$, $\deg_\Delta(i) \geq 1$ for all $i \in V_b$, as well as the estimate
\begin{equation*}
\sum_a \abs{v_a}^k \;\leq\;
\begin{cases}
M^{1/2} & \text{if } k = 1
\\
1 & \text{if } k \geq 2\,.
\end{cases}
\end{equation*}
Summing over $\sigma \in L_1(\cal T)$ concludes the proof of Proposition \ref{prop: nontrivial leaves}.

\subsection{Conclusion of the proof of Theorem \ref{thm: IMP}} \label{sec: exp 9}

Combining \eqref{main estimate of trivial leaves} and Proposition \ref{prop: nontrivial leaves}, and recalling \eqref{Y in terms of tree} and \eqref{E Zp}, yieds
\begin{equation*}
\E \abs{\cal Z}^p \;\prec\; C_{p, \omega} (\phi^{-1} \Psi)^p\,.
\end{equation*}
Now \eqref{claim on X}, and hence \eqref{bound: Gij isotropic}, follows by a simple application of Chebyshev's inequality. Let $\epsilon > 0$ and $D$ be given. Then
\begin{equation*}
\P(\abs{\cal Z} > M^\epsilon \phi^{-1} \Psi) \;\leq\; C_{p, \omega} M^\epsilon M^{-\epsilon p} \;\leq\; M^{-D}
\end{equation*}
for $p \geq \epsilon^{-1} D + 2$. This, together with Remark \ref{rem:stochdomMN} which allows us to interchange $M$ and $N$ in Definition \ref{def:stocdom},  concludes the proof of \eqref{bound: Gij isotropic}.

Finally, we outline the proof of \eqref{bound: Rij isotropic}, which is very similar to that of \eqref{bound: Gij isotropic}. The expansion of Sections \ref{sec: graphs}--\ref{sec: exp 6} may be taken over by swapping the roles of $R$ and $\wt G$. In other words, we use Lemma \ref{lem: res identity 2} instead of Lemma \ref{lem: res identity 1}.  The arguments from Sections \ref{sec: exp 7} and \ref{sec: exp 8} carry over with straightforward adjustments in the power counting of Section \ref{sec: exp 8}. We leave the details to the interested reader. This concludes the proof of
 Theorem \ref{thm: IMP}.

\section{Proof of Theorems \ref{thm: IMP outside} and \ref{thm: I deloc}} \label{sec:5}

\begin{proof}[Proof of Theorem \ref{thm: I deloc}]
Note that, by the definition \eqref{def:gamma_alpha} of $\gamma_\alpha$, we have $\gamma_\alpha \in [\gamma_-, \gamma_+]$ for all $\alpha = 1, \dots, N$. Hence, given $\epsilon > 0$ and $c > 0$ as in Theorem \ref{thm: I deloc}, for small enough $\omega \in (0,1)$ we have $\gamma_\alpha \geq 2 \omega$ provided that either $\alpha \leq (1 - \epsilon) N$ or $\phi \geq 1 + c$. Set $\eta \deq N^{-1 + \omega}$. We therefore conclude from Theorem \ref{thm: cov-rig} and the definition \eqref{def_S_theta} of $\f S$ that $\lambda_\alpha + \ii \eta \in \f S$ with high probability (see Definition \ref{def: high probability}), provided that either $\alpha \leq (1 - \epsilon) N$ or $\phi \geq 1 + c$. Let $\alpha$ be such an index, and abbreviate $\Xi \deq \{\lambda_\alpha + \ii \eta \in \f S\}$. Then we get from \eqref{bound: Rij isotropic}, 
Remark \ref{rem:all_z}, and \eqref{bounds on mg} that $\ind{\Xi} \im \scalar{\f v}{R(\lambda_\alpha + \ii \eta) \f v} \prec 1$. From
\begin{equation*}
\im \scalar{\f v}{R(\lambda_\alpha + \ii \eta) \f v} \;=\; \sum_{\beta = 1}^N \frac{\eta \abs{\scalar{\f u^{(\beta)}}{\f v}}^2}{(\lambda_\alpha - \lambda_\beta)^2 + \eta^2} \;\geq\; \frac{\abs{\scalar{\f u^{(\alpha)}}{\f v}}^2}{\eta}
\end{equation*}
we therefore get $\ind{\Xi} \abs{\scalar{\f u^{(\alpha)}}{\f v}}^2 \prec N^{-1 + \omega}$. Since $\omega \in (0,1)$ can be made arbitrarily small and $1 - \ind{\Xi} \prec 0$, the first estimate of Theorem \ref{thm: I deloc} follows. 

In order to prove the second estimate of Theorem \ref{thm: I deloc}, we use the same $\eta = N^{-1 + \omega}$ as above and write $z = \lambda_\alpha + \ii \eta$. Taking the imaginary part inside the absolute value on the left-hand side of \eqref{bound: Gij isotropic}, we get
\begin{equation*}
\ind{\Xi} \im \scalar{\f w}{G(z) \f w} \;\prec\; \im m_{\phi^{-1}}(z) + \frac{1}{\phi} \;\leq\; \frac{2}{\phi}\,,
\end{equation*}
where in the second step we used \eqref{im_m_swap}, \eqref{bounds on mg}, and $z \in \f S$ with high probability; this latter estimates follows from \eqref{def:gamma_pm}, the fact that $\gamma_\alpha \geq 2 \omega$ by assumption, and Theorem \ref{thm: cov-rig}.  Repeating the above argument, we therefore find $\ind{\Xi} \abs{\scalar{\wt {\f u}\!\,^{(\alpha)}}{\f w}}^2 \prec \phi^{-1} N^{-1 + \omega}$, and the second claim of Theorem \ref{thm: I deloc}  follows.
\end{proof}

\begin{proof}[Proof of Theorem \ref{thm: IMP outside}]
We only prove \eqref{bound: Gij isotropic outside sc}; the proof of \eqref{bound: Rij isotropic outside sc} is the same, using \eqref{bound: Rij isotropic} instead of \eqref{bound: Gij isotropic}. Moreover, to simplify notation, we assume that $E \geq \gamma_+ + N^{-2/3 + \omega}$; the case $E \leq \gamma_- - N^{-2/3 + \omega}$ is handled in exactly the same way.

Note first that if $\eta \geq \kappa$ then it is easy to see that \eqref{bound: Gij isotropic outside sc} follows from \eqref{bound: Gij isotropic}, \eqref{im m gamma}, and the lower bound $\eta \geq \kappa \geq N^{-2/3}$. For the following we therefore assume that $\eta \leq \kappa$.
By Lemma \ref{lemma: mg}, for $\eta \leq \kappa$ we have
\begin{equation*}
\sqrt{\frac{\im m_\phi(z)}{N\eta}} \;\asymp\; N^{-1/2} \kappa^{-1/4}\,.
\end{equation*}
By polarization and linearity, it therefore suffices to prove that
\begin{equation} \label{claim outside of spectrum}
\absb{\scalar{\f v}{G(z) \f v} - m_{\phi^{-1}}(z)} \;\prec\; \phi^{-1} N^{-1/2} \kappa^{-1/4}\,.
\end{equation}

Define $\eta_0 \deq N^{-1/2} \kappa^{1/4}$. By definition of the domain $\wt{\f S}$,
 we have $\eta_0 \leq \kappa$. 
Using \eqref{bound: Gij isotropic}, we find that  \eqref{claim outside of spectrum}  holds if $\eta \geq \eta_0$.
For the following we therefore take $0 < \eta \leq \eta_0$. We proceed by comparison using the two spectral parameters
\begin{equation*}
z \;\deq\; E + \ii \eta\,, \qquad z_0 \;\deq\; E + \ii \eta_0\,.
\end{equation*}
Since \eqref{claim outside of spectrum} holds at $z_0$ by \eqref{bound: Gij isotropic}, it is enough to prove the estimates
\begin{equation} \label{z z0 comparison 1}
\absb{m_\phi(z) - m_\phi(z_0)} \;\leq\; C N^{-1/2} \kappa^{-1/4}
\end{equation}
and
\begin{equation} \label{z z0 comparison 2}
\absb{\scalar{\f v}{G(z) \f v} - \scalar{\f v}{G(z_0) \f v}} \;\prec\; \phi^{-1} N^{-1/2} \kappa^{-1/4}\,.
\end{equation}
(The third required estimate, that of $\absb{\frac{1 - \phi}{\phi z} - \frac{1 - \phi}{\phi z_0}}$, is trivial by $\abs{z}^2 \asymp \phi$ for any $z\in \wt {\f S}$.)
We start with \eqref{z z0 comparison 1}. From the definition $m_\phi(z) = \int \frac{\varrho_\phi(\dd x)}{x - z}$ and the square root decay of the density of $\varrho_\phi$ near $\gamma_+$ from \eqref{def: rhog}, it is not hard to derive the bound $m'_\phi(z) \leq C \kappa^{-1/2}$ for $z \in \wt {\f S}$.
Therefore we get
\begin{equation*}
\absb{m_\phi(z) - m_\phi(z_0)} \;\leq\; C \kappa^{-1/2} \eta_0 \;=\; C N^{-1/2} \kappa^{-1/4}\,,
\end{equation*}
which is \eqref{z z0 comparison 1}.

What remains is to prove \eqref{z z0 comparison 2}. By Theorem \ref{thm: cov-rig} we have $E \geq \lambda_1 + \eta_0$ with high probability
since $\eta_0\ge N^{-2/3+\omega/4}$. Therefore, since $\eta \leq \eta_0 \leq E - \lambda_1 \leq E - \lambda_\alpha$ for all $\alpha \geq 1$, we get
\begin{equation} \label{strong 4}
\im \scalar{\f v}{G(z) \f v} \;=\; \sum_\alpha \frac{\abs{\scalar{\f v}{\f u^{(\alpha)}}}^2 \eta}{(E - \lambda_\alpha)^2 + \eta^2} \;\leq\;
2 \sum_\alpha \frac{\abs{\scalar{\f v}{\f u^{(\alpha)}}}^2 \eta_0}{(E - \lambda_\alpha)^2 + \eta_0^2} \;=\; 2 \im \scalar{\f v}{G(z_0) \f v} \;\prec\; \phi^{-1} N^{-1/2} \kappa^{-1/4}
\end{equation}
by \eqref{bound: Gij isotropic} at $z_0$ and the estimate 
\begin{equation*}
\im \pbb{\frac{1 - \phi}{\phi z_0} + \frac{m_\phi(z_0)}{\phi}} \;\leq\; C \phi^{-1} N^{-1/2} \kappa^{-1/4}\,,
\end{equation*}
as follows from \eqref{im m gamma} and the estimate $\abs{z_0}^2 \asymp \phi$.

Finally, we estimate the real part of the error in \eqref{z z0 comparison 2} using
\begin{multline} \label{strong 5}
\absb{\re \scalar{\f v}{G(z) \f v} - \re \scalar{\f v}{G(z_0) \f v}} \;=\; \sum_\alpha \frac{(E - \lambda_\alpha)(\eta_0^2 - \eta^2) \abs{\scalar{\f u^{(\alpha)}}{\f v}}^2}{\pb{(E - \lambda_\alpha)^2 + \eta^2} \pb{(E - \lambda_\alpha)^2 + \eta^2_0}}
\\
\leq\; \frac{\eta_0}{E - \lambda_1} \sum_\alpha \frac{\eta_0 \abs{\scalar{\f u^{(\alpha)}}{\f v}}^2}{(E - \lambda_\alpha)^2 + \eta^2_0} \;\leq\; \im \scalar{\f v}{G(z_0) \f v}
\end{multline}
with high probability,
where in the last step we used that $\eta_0 \leq E - \lambda_1$ with high probability. Combining \eqref{strong 4} and \eqref{strong 5} completes the proof of \eqref{z z0 comparison 2}, and hence of Theorem \ref{thm: IMP outside}.
\end{proof}

\section{Proofs for generalized Wigner matrices} \label{sec:appendix}

In this section we explain how to modify the arguments of Sections \ref{sec:tools}--\ref{sec:5} to the case of generalized Wigner matrices, and hence how to complete the proof of the results from Section \ref{sec: gen_wigner}. Since we are dealing with generalized Wigner matrices, in this section we consistently use the notations from Section \ref{sec: gen_wigner} instead of Section \ref{sec: sample covariance}.

\subsection{Tools for generalized Wigner matrices}
We begin by recalling some basic facts about the Stieltjes transform $m$ from \eqref{def_msc}.
In analogy to \eqref{def kappa sc}, for $E \in \R$ we define
\begin{equation} \label{def kappa}
\kappa \;\equiv\; \kappa_E \;\deq\; \absb{\abs{E} - 2}\,,
\end{equation}
the distance from $E$ to the spectral edges $\pm 2$.
\begin{lemma} \label{lemma: msc}
For $\abs{z} \leq 2 \omega^{-1}$ we have
\begin{equation} \label{bounds on msc}
\abs{m(z)} \;\asymp\; 1 \,, \qquad \abs{1 - m(z)^2} \;\asymp\; \sqrt{\kappa + \eta}
\end{equation}
and
\begin{equation*}
\im m(z) \;\asymp\;
\begin{cases}
\sqrt{\kappa + \eta} & \text{if $\abs{E} \leq 2$}
\\
\frac{\eta}{\sqrt{\kappa + \eta}} & \text{if $\abs{E} \geq 2$}\,.
\end{cases}
\end{equation*}
(All implicit constants depend on $\omega$.)
\end{lemma}
\begin{proof}
The proof is an elementary calculation; see Lemma 4.2 in \cite{EYY2}.
\end{proof}

The following definition is the analogue of Definitions \ref{def: removing rows} and \ref{def: removing columns}. (Note that for generalized Wigner matrices we always simultaneously remove a row and the corresponding column.)
\begin{definition}[Minors] \label{def: Wigner minors}
For $T \subset \{1, \dots, N\}$ we define $H^{(T)}$ by
\begin{equation*}
(H^{(T)})_{ij} \;\deq\; \ind{i \notin T} \ind{j \notin T} H_{ij}\,.
\end{equation*}
Moreover, for $i,j \notin T$ we define the resolvent of the minor through
\begin{equation*}
G^{(T)}_{ij}(z) \;\deq\; (H^{(T)} - z)^{-1}_{ij}\,.
\end{equation*}
We also set
\begin{equation*}
\sum_i^{(T)} \;\deq\; \sum_{i \col i \notin T}\,.
\end{equation*}
When $T = \{a\}$, we abbreviate $(\{a\})$ by $(a)$ in the above definitions; similarly, we write $(ab)$ instead of $(\{a,b\})$.
\end{definition}

We shall also need the following resolvent identities, proved in Lemma 4.2 of \cite{EYY1} and Lemma 6.10 of \cite{EKYY2}.
\begin{lemma}[Resolvent identities] \label{lemma: res id}
For $i,j \neq k$ and $i,j,k \notin T$ the identity \eqref{Gij Gijk sc} holds.
Moreover, for $i \neq j$ and $i,j \notin T$ we have
\begin{equation} \label{sq root formula}
G_{ij}^{(T)} \;=\; - G_{ii}^{(T)} \sum_{k}^{(Ti)} H_{ik} G_{kj}^{(Ti)} \;=\; - G_{jj}^{(T)} \sum_k^{(Tj)} G_{ik}^{(Tj)} H_{k j}\,.
\end{equation}
Finally, for $i \notin T$ we have Schur's formula
\begin{equation} \label{Schur}
\frac{1}{G_{ii}^{(T)}} \;=\; H_{ii} - z - \sum_{k,l}^{(T i)} H_{ik} G_{kl}^{(T i)} H_{li}\,.
\end{equation}
\end{lemma}
From \eqref{sq root formula} we find for $i \neq j$ and $i,j \notin T$ that
\begin{equation} \label{iterated sq root}
G_{ij}^{(T)} \;=\;  G_{ii}^{(T)} G_{jj}^{(Ti )} \pbb{- H_{ij} + \sum_{k,l}^{(Tij)} H_{ik} G_{kl}^{(Tij)} H_{lj}}\,,
\end{equation}

\subsection{The isotropic law: proof of Theorem \ref{thm: ILSC}}
As in \eqref{claim on X}, and using \eqref{eq:lscold} instead of
\eqref{bound: Gij}, it suffices to prove that
\begin{equation*}
\abs{\cal Z} \;\prec\; \sqrt{\frac{\im m(z)}{N\eta}}+\frac{1}{N\eta}\,,
\end{equation*}
where, as in Section \ref{sec:4}, $\cal Z = \sum_{a \neq b} \ol v_a G_{ab} v_b$. (Note that here there is no factor $\phi$ and hence no rescaled quantities bearing a tilde.) The estimate of $\E \abs{\cal Z}^p$ follows the argument of Section \ref{sec:4}; in particular, it consists of the eight steps sketched at the end of Section \ref{sec:sketch}, which we follow closely.
Throughout the argument we use the identities \eqref{Schur} and \eqref{iterated sq root} instead of \eqref{Gii expanded sc} and \eqref{Gij expanded sc}. In them, the two differences as compared to the argument of Section \ref{sec:4} are apparent.
\begin{enumerate}
\item
The quadratic term in the expansion of $1/G_{ii}$ and $G_{ij}$ contains an entry of $G$ and not of $R$. (The matrix $R$ is not even defined for generalized Wigner matrices.)
\item
Both \eqref{Schur} and \eqref{iterated sq root} contain an additional term, an entry of $H$.
\end{enumerate}
Of these differences, the first one is minor. In order to mimic the bookkepping of Section \ref{sec:4}, we still speak of black and white vertices.
The simplest definition of our colouring is that the vertices of $\Delta$ are black and
any other vertices that were added to $\Delta$ are white;
see the explanation below \eqref{VB}.
The second difference leads to a slightly larger class of graphs, but the new graphs will always be of subleading order. An alternative viewpoint is that the additional entry of $H$ on the right-hand side of \eqref{Schur} and \eqref{iterated sq root} should be regarded as a negligible error term. Almost all of the differences highlighted below stem from this additional term.

We now sketch the argument for each of the eight steps, by highlighting the changes as compared to Section \ref{sec:4}.
\begin{description}
\item[Step 1.]
The reduced family of matrix indices is, as in Section \ref{sec:4}, the set of indices $\f a_b$ associated with the black vertices of $\Delta$.
\item[Step 2.] We use \eqref{Gij Gijk sc} to make all entries of $G$ maximally expanded of the black indices $\f a_b = (a_i)_{i \in V_b}$. The graphical representation is the same as in Section \ref{sec:4}.
\item[Step 3.] We use \eqref{iterated sq root} to expand each maximally expanded off-diagonal entry of $G_{a_i a_j}$. As compared to the expansion based
on  \eqref{Gii expanded sc} and  \eqref{Gij expanded sc} and
used in Section \ref{sec:4}, we get an additional term, $-H_{a_i a_j}$. See Figure \ref{fig:Gij_W} for a graphical representation of this expansion. Note that \eqref{iterated sq root} is always invoked with $T = \f a_b \setminus \{a_i,a_j\}$. Hence, for any graph $\Gamma$ at any point of the argument, each white index is summed over the set $\{1, \dots, N\} \setminus \f a_b$.
\begin{figure}[ht!]
\begin{center}
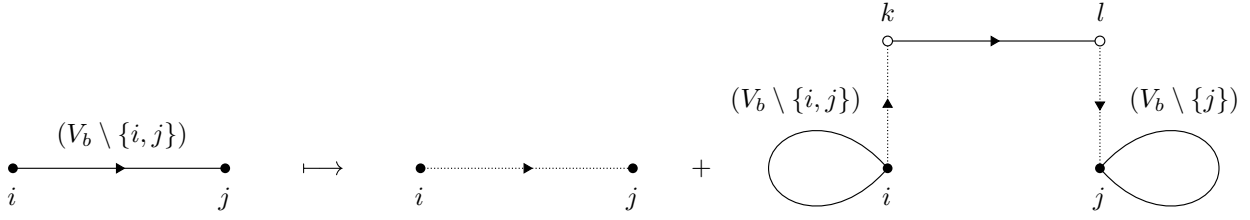
\end{center}
\caption{How Figure \ref{fig: Gab expansion} is modified for generalized Wigner matrices. We use an oriented dotted line from $i$ to $j$ to draw the entry $-H_{a_i a_j}$. \label{fig:Gij_W}}
\end{figure}
\item[Step 4.] Repeating Steps 2 and 3 in tandem yields a sum of monomials which consist only of maximally expanded diagonal entries of $G$ with black indices, entries of $G^{(\f a_b)}$ with white indices, and entries of $H$.
\item[Step 5.]
We apply
\eqref{Schur} to each maximally expanded diagonal entry of $G$. The graphical representation of this operation is similar to Figure \ref{fig: expanding Gaa}, except that we also get a diagonal entry of $H$, depicted as a dotted loop. 
 Note that there are no $R$-entries, but we still use the terminology of Section \ref{sec:4} and speak of $R$-groups; these refer to the
to the same structure as in Figure~\ref{fig: R-group}, except that the edge $e$ connecting the two white vertices encodes
a $G$-edge and not an $R$-edge. 
 We end up with entries of $G^{(\f a_b)}$ and entries of $H$. Note that, by construction, each entry of $H$ carries at least one black index, and that the $G$-edges are only incident to white vertices.  In particular, all entries of $H$ are independent of all entries of $G^{(\f a_b)}$.
\item[Step 6.]
Using the independence of the entries of $H$ and $G^{(\f a_b)}$, we may take the partial expectation in the rows (or, equivalently, columns) indexed by $\f a_b$. Note that now we have two classes of $H$-edges: white-black (incident to a black and a white vertex) and black-black (incident to two black vertices). Since the white indices are distinct from the black ones, the expectation factorizes over these two classes of $H$-edges.
 Exactly as in Section \ref{sec:4}, taking the expectation in the white-black $H$-edges yields, for each $i \in V_b$, a partition of the white vertices adjacent to $i$, whereby each block of the partition must contain at least two vertices. The expectation over the black-black $H$-edges imposes an additional constraint among the loops incident to the white vertices, which are unimportant for the argument. Finally, for $i \neq j \in V_b$, we have the constraint that the number of edges joining $i$ and $j$ cannot be one.
\item[Steps 7 and 8.]
The parity argument from the proof of Theorem \ref{sec: exp 8} may be taken over with minor modifications, which arise from the additional black-black $H$-edges described in Step 6. Recall that the goal is to gain a factor $N^{-1/2}$ from each black vertex $i \in V_b$ that has an odd degree. If $i$ is incident to a black-white $H$-edge, the counting from Section \ref{sec:4} applied unchanged and yields a power of $N^{-1/2}$. If $i$ is not incident to a black-white $H$-edge, it must be incident to a black-black $H$-edge (recall that all graphs must be connected). By the constraints arising from the expectation in Step 6, $i$ must then in fact be incident to at least two black-black $H$-edges which connect $i$ to the same black vertex $j$. This yields a factor $\E \abs{H_{a_i a_j}}^2 \leq C / N$, which is the desired small factor. (We may in general only allocate $N^{-1/2}$ from the factor $N^{-1}$ to the vertex $i$, since $j$ may also be a vertex that has degree one in $\Delta$, in which case we have to allocate the other factor in $N = N^{-1/2} N^{-1/2}$ to $j$.)
\end{description}

This concludes sketch of how the argument of Section \ref{sec:4} is to be modified for the proof of Theorem \ref{thm: ILSC}. We omit further details. Finally, Theorems \ref{thm: ILSC outside} and \ref{thm: I deloc Wig} follow from Theorem \ref{thm: ILSC} by repeating the arguments of Section \ref{sec:5} almost to the letter.

{\small
\providecommand{\bysame}{\leavevmode\hbox to3em{\hrulefill}\thinspace}
\providecommand{\MR}{\relax\ifhmode\unskip\space\fi MR }
\providecommand{\MRhref}[2]{%
  \href{http://www.ams.org/mathscinet-getitem?mr=#1}{#2}
}
\providecommand{\href}[2]{#2}

}

\end{document}